\documentclass[10pt,letterpaper]{amsart}
\usepackage{dslihead}

\title{Courbes et fibr\'es vectoriels en th\'eorie de Hodge $z$-adique globale}

\author{Siyan Daniel Li-Huerta}
\address{Department of Mathematics\\Massachusetts Institute of Technology\\77 Massachusetts Avenue\\Cambridge, MA 02139}
\email{sdlh@mit.edu}

\usepackage{amsmath}
\usepackage{tikz-cd}

\swapnumbers
\theoremstyle{plain}

\newtheorem*{thm*}{Theorem}

\newtheorem*{lem*}{Lemma}
\newtheorem{prop}[subsection]{Proposition}
\newtheorem*{prop*}{Proposition}

\newtheorem*{conj*}{Conjecture}

\newtheorem*{cor*}{Corollary}

\newtheorem*{thmA}{Theorem A}
\newtheorem*{thmB}{Theorem B}
\newtheorem*{thmC}{Theorem C}
\newtheorem*{thmD}{Theorem D}
\newtheorem*{thmE}{Theorem E}
\newtheorem*{conjF}{Conjecture F}
\newtheorem*{thmG}{Theorem G}
\newtheorem*{thmH}{Theorem H}

\theoremstyle{definition}

\newtheorem*{defn*}{Definition}

\theoremstyle{remark}
\newtheorem{rem}[subsection]{Remark}
\newtheorem{rems}[subsection]{Remarks}
\newtheorem*{rem*}{Remark}
\newtheorem*{rems*}{Remarks}
\newtheorem{exmp}[subsection]{Example}

\theoremstyle{plain}

\theoremstyle{definition}

\theoremstyle{remark}

\DeclareMathOperator{\cond}{cond}
\DeclareMathOperator{\AnRing}{AnRing}
\DeclareMathOperator{\basic}{basic}
\DeclareMathOperator{\QCoh}{QCoh}
\DeclareMathOperator{\Isom}{Isom}
\DeclareMathOperator{\LocSht}{LocSht}
\newcommand{\sLocSht}{\sL\!\mathrm{ocSht}}
\newcommand{\sIsoc}{\sI\!\mathrm{soc}}

\renewcommand{\qc}{\operatorname{qc}}
\DeclareMathOperator{\fib}{fib}
\DeclareMathOperator{\AnSpec}{AnSpec}
\DeclareMathOperator{\pc}{pc}

\DeclareMathOperator{\Perfd}{Perfd}
\DeclareMathOperator{\Hig}{Hig}
\DeclareMathOperator{\qcoh}{qcoh}
\DeclareMathOperator{\WD}{WD}
\DeclareMathOperator{\ID}{ID}
\DeclareMathOperator{\Ig}{Ig}
\DeclareMathOperator{\lres}{res}
\DeclareMathOperator{\IndCoh}{IndCoh}

\newcommand*{\sheafhom}{\textnormal{H}\kern -.5pt om}

\newcommand{\Kal}{\tn{Kal}_{F}}

\newcommand\tn{\textnormal}

\begin{document}

\begin{abstract}
We study the global analogue of the Fargues--Fontaine curve over function fields $F$. We prove some foundational results about its moduli of $G$-bundles $\Bun_{G,F}$, which is a geometrization of the global Kottwitz set $B(F,G)$. For example, $\Bun_{G,F}$ plays the role of Igusa stacks over function fields. We use $\Bun_{G,F}$ to reformulate the global Langlands conjecture for $G$ over $F$ in terms of categorical local Langlands, refining conjectures of Arinkin--Gaitsgory--Kazhdan--Raskin--Rozenblyum--Varshavsky and Zhu. Finally, we verify this conjecture when $G$ is commutative. Along the way, we prove a GAGA theorem for smooth proper schemes over sousperfectoid spaces, which is of independent interest.
\end{abstract}

\maketitle
\tableofcontents

\section*{Introduction}
The title of this paper pays homage to the foundational work of Fargues--Fontaine \cite{FF18}. Starting from a nonarchimedean local field $F_v$ with residue field $\bF_q$ (and an auxiliary algebraically closed nonarchimedean field $K$ over $\ov\bF_q$), they constructed an adic space $X_{F_v,K}$ with remarkable properties. Since its introduction, the \emph{Fargues--Fontaine curve} $X_{F_v,K}$ has played a fundamental role in arithmetic geometry over $F_v$. For example, work of Fargues--Scholze \cite{FS21} shows that its moduli of $G$-bundles $\Bun_{G,F_v}$ is central to the Langlands correspondence for $G$ over $F_v$.

When $F_v$ is the function field $\bF_q\lp{z}$, the Fargues--Fontaine curve already appears in work of Hartl--Pink \cite{HP04}, so we also call it the \emph{Hartl--Pink curve}. It can be described as follows. The product $\Spa F_v\times\Spa K$ over $\bF_q$ exists as an adic space (in fact, it is isomorphic to the punctured open unit disk over $K$), and the absolute $q$-Frobenius automorphism $\Frob_K$ of $K$ acts freely and totally discontinuously on $\Spa F_v\times\Spa K$. Hence we can form the quotient adic space
\begin{align}\label{eq:FFcurve}
X_{F_v,K}\coloneqq(\Spa F_v\times\Spa K)/\Frob_K.\tag{$\dagger$}
\end{align}

The goal of this paper is to study a \emph{global} analogue of the above constructions.

\subsection*{$G$-bundles on the global Hartl--Pink curve}
Let $F$ be a global function field with field of constants $\bF_q$. By viewing schemes as adic spaces, one could try to define a \emph{global Hartl--Pink curve} by naively copying (\ref{eq:FFcurve}) and attempting to form 
\begin{align*}\label{eq:gHPcurve}
\text{``}(\Spec{F}\times\Spa K)/\Frob_K\text{''.}\tag{$\ddagger$}
\end{align*}
However, two obstacles arise:
\begin{enumerate}[1)]
\item It is unclear how to interpret ``$\Spec{F}\times\Spa K$'' in a well-behaved way.
\item In any reasonable interpretation of ``$\Spec{F}\times\Spa K$'', $\Frob_K$ will no longer act freely and totally discontinuously.
\end{enumerate}
We will primarily be interested in $G$-bundles on (\ref{eq:gHPcurve}), so we can circumvent 2) via descent. To circumvent 1), write $C$ for the geometrically connected smooth proper curve over $\bF_q$ associated with $F$. Then $\Spec{F}=\varprojlim_UU$, where $U$ runs over dense open subschemes of $C$, so formally we expect
\begin{align*}
\text{``}\Spec{F}\times\Spa K\text{''} = \text{``}\varprojlim_UU\times\Spa K\text{''.}
\end{align*}
Now $U\times\Spa{K}$ \emph{does} have a well-behaved interpretation: this product exists as an adic space and in fact is isomorphic to the analytification $U_K^{\an}$ of $U$ over $K$. Moreover, this interpretation works well in families, which leads to the following:
\begin{defn*}Let $G$ be a connected split\footnote{In the introduction, we work with split $G$ for simplicity. However, everything works for general connected reductive groups $G$ over $F$, and we work in this generality in the body of the paper.} reductive group over $F$. For all affinoid perfectoid spaces $S=\Spa(R,R^+)$ over $\ov\bF_q$,
  \begin{enumerate}[a)]
  \item Write $\Bun_{G,U}(S)$ for the groupoid of $G$-bundles $\sG$ on $U_S^{\an}$ equipped with an isomorphism $\phi:\sG\ra^\sim\Frob_S^*\sG$.
  \item Write $\Bun_{G,F}(S)$ for the groupoid $\varinjlim_U\Bun_{G,U}(S)$.
  \end{enumerate}
\end{defn*}
\begin{rems*}\hfill
\begin{enumerate}[1)]
\item One can recover $\Bun_{G,U}$ from $\Bun_{G,F}$; see Lemma \ref{ss:BunGpullback}. Therefore we will interchangeably use $\Bun_{G,U}$ and $\Bun_{G,F}$.
\item Amusingly, while we do not define \emph{the} global Hartl--Pink curve, nonetheless we can define its moduli of $G$-bundles $\Bun_{G,F}$. This is reminiscent of the situation for \emph{local} Hartl--Pink curves $X_{F_v,K}$, which depend on the auxiliary $K$.
\end{enumerate}
\end{rems*}

\subsection*{Geometry of $\Bun_{G,F}$}
We start with the following basic properties of $\Bun_{G,U}$.
\begin{thmA}
$\Bun_{G,U}$ is a small Artin v-stack over $\ov\bF_q$. It is $\ell$-cohomologically smooth over $\ov\bF_q$, and its dualizing complex with $\bF_\ell$-coefficients is isomorphic to $\bF_\ell$.
\end{thmA}
Theorem A mirrors results of Fargues--Scholze \cite[Theorem IV.1.19]{FS21} and Hamann--Imai \cite[Proposition 3.18]{HI25} in the local setting.

To state our next theorem, let us introduce the \emph{localization map}, which plays an important role throughout this paper. Write $Z$ for the closed complement $C\ssm U$. For all $z$ in $Z$, there is a natural map $\Spa F_z\ra U$, so by restricting $G$-bundles along these maps and using the presentation (\ref{eq:FFcurve}), we get a map of v-stacks
\begin{align*}
\loc_Z:\Bun_{G,U}\ra\prod_{z\in Z}\Bun_{G,F_z}.
\end{align*}

Recall that the local moduli $\Bun_{G,F_z}$ has an open substack $\Bun_{G,F_z}^{\semis}$ \cite[Theorem III.4.5]{FS21}, which is the locus where the corresponding $G$-bundle on the curve $X_{F_z,K}$ is semistable. By taking its preimage under $\loc_Z$, we define a \emph{global semistable locus} $\Bun_{G,U}^{\semis}$. While $\Bun_{G,U}^{\semis}$ is defined in terms of its local analogue $\Bun^{\semis}_{G,F_z}$, we prove that it admits the following intrinsic description.
\begin{thmB}
  $\Bun_{G,F}^{\semis}$ is an open substack of $\Bun_{G,F}$, and we have an isomorphism
  \begin{align*}
    \Bun^{\semis}_{G,F}\cong\coprod_{b\in B(F,G)_{\basic}}*/\ul{G_b(F)},
  \end{align*}
where $B(F,G)$ denotes the \emph{global Kottwitz set} \cite{Kot14}, $B(F,G)_{\basic}$ denotes its subset of \emph{basic} elements, and $G_b$ is the inner twist of $G$ associated with $b$.
\end{thmB}
Theorem B mirrors a result of Fargues--Scholze \cite[Theorem III.4.5]{FS21} in the local setting. This is one incarnation of the idea that $\Bun_{G,F_z}$ is a geometrization of the local Kottwitz set $B(F_z,G)$, but there are other such incarnations. For example, Ansch\"utz \cite[Theorem 10]{Ans19} proved that the underlying topological space $\abs{\Bun_{G,F_z}}$ is naturally in bijection with $B(F_z,G)$.

We similarly regard $\Bun_{G,F}$ as a geometrization of the global Kottwitz set. However, for reasons explained in the next subsection, we do \emph{not} expect a natural bijection between $\abs{\Bun_{G,F}}$ and $B(F,G)$ in general. Nonetheless, we prove the following description of the $\ov\bF_q$-points of $\Bun_{G,F}$.
\begin{thmC}
  $\Bun_{G,F}(\ov\bF_q)$ is naturally equivalent to the groupoid quotient
  \begin{align*}
    G(F\otimes_{\bF_q}\ov\bF_q)/G(F\otimes_{\bF_q}\ov\bF_q),
  \end{align*}
where the action is given by $\Frob_{\ov\bF_q}$-conjugation. Consequently, the set of isomorphism classes in $\Bun_{G,F}(\ov\bF_q)$ is naturally in bijection with $B(F,G)$.
\end{thmC}
Theorem C mirrors a result of Ansch\"utz \cite[Theorem 5.3]{Ans23} in the local setting. Actually, we prove a generalization of Theorem C that fully describes the \emph{reduction} of $\Bun_{G,F}$; see Theorem \ref{ss:BunGred}. The more general Theorem \ref{ss:BunGred} mirrors a result of Gleason--Ivanov--Zillinger \cite[Theorem 7.14.(1)]{GIZ25} in the local setting.

\subsection*{Relation with shtukas}
It turns out that $\Bun_{G,U}$ plays the role of \emph{Igusa stacks} \cite{Zha23} over function fields. Igusa stacks capture the geometry of Shimura varieties, so let us recall the function field analogue of the latter.  Let $I$ be a finite set, and for all $i$ in $I$, let $V_i$ be a representation of the Langlands dual group $\wh{G}$. Write $V$ for the representation $\boxtimes_{i\in I}V_i$ of $\wh{G}^I$. Associated with $I$ and $V$, there is a moduli space $\Sht^I_{G,V}$ of \emph{global shtukas} bounded by $V$ \cite{Var04}, which is a Deligne--Mumford stack that is separated and locally of finite type over $C^I$.

There is also a local version of $\Sht^I_{G,V}$. Let $\{I_z\}_{z\in Z}$ be a partition of $I$, and for all $z$ in $Z$, write $V_z$ for the representation $\boxtimes_{i\in I_z}V_i$ of $\wh{G}^{I_z}$. Associated with $I_z$ and $V_z$, there is a moduli space $\LocSht^{I_z}_{G,V_z}$ of \emph{local shtukas} bounded by $V_z$ \cite{SW20}, which is a small v-stack over $(\Spd\cO_z)^{I_z}$.

We prove the following function field analogue of Scholze's \emph{fiber product conjecture} \cite[Conjecture 1.1]{Zha23}, which is the defining feature of Igusa stacks:
\begin{thmD}
There is a natural cartesian square of small v-stacks
\begin{align*}
  \xymatrixcolsep{2cm}
  \xymatrix{(\Sht^I_{G,V})^\Diamond|_{\prod_{z\in Z}(\Spd\cO_z)^{I_z}_{\ov\bF_q}}\ar[r]^-{\pi_{\HT}}\ar[d] & \displaystyle\prod_{z\in Z}\!\big(\!\LocSht^{I_z}_{G,V_z}\!\big)_{\ov\bF_q}\ar[d]^{\BL} \\
  \Bun_{G,U}\ar[r]^-{\loc_Z} & \displaystyle\prod_{z\in Z}\Bun_{G,F_z},}
\end{align*}
where $\BL$ denotes the product of the \emph{Beauville--Laszlo maps} $(\LocSht^{I_z}_{G,V_z})_{\ov\bF_q}\ra\Bun_{G,F_z}$ \cite[p.~98]{FS21}, and $\pi_{\HT}$ is a function field analogue of the Hodge--Tate period map.
\end{thmD}
Actually, we prove a generalization of Theorem D that allows level structures; see Theorem \ref{ss:fiberproductconj}. Previously, Hartl--Viehmann \cite{HV23} studied a version of the period map $\pi_{\HT}$ for the function field analogue of Rapoport--Zink spaces.
\begin{rems*}\hfill
  \begin{enumerate}[1)]
  \item Theorem D implies that $\Bun_{G,U}$ encodes the function field analogue of \emph{Igusa varieties}; see Proposition \ref{ss:igusavarsfiber}. Because Igusa varieties associated with non-basic elements of the Kottwitz set can be positive-dimensional in general, this explains why we do not expect a natural bijection between $\abs{\Bun_{G,F}}$ and $B(F,G)$.
  \item Gleason--Ivanov--Zillinger \cite[Theorem 7.14.(2)]{GIZ25} proved a natural description of the $\ov\bF_q$-points of $\LocSht^{I_z}_{G,V_z}$. By combining this description with Theorem C and the aforementioned result of Ansch\"utz \cite[Theorem 5.3]{Ans23}, Theorem D immediately implies the \emph{Langlands--Rapoport conjecture} for moduli spaces of global shtukas with arbitrary (in particular, colliding) legs. This generalizes results of Arasteh Rad--Hartl \cite[Theorem 3.21]{AH23}, who proved this when the legs are disjoint.
  \end{enumerate}
\end{rems*}

\subsection*{Relation with Langlands}
In the local setting, $\Bun_{G,F_z}$ plays a central role in the Langlands correspondence for $G$ over $F_z$. For example, let $\La$ be one of $\ov\bQ_\ell$, $\ov\bZ_\ell$, or $\ov\bF_\ell$, and write $D(\Bun_{G,F_z},\La)$ for the associated derived category of \'etale sheaves.\footnote{Strictly speaking, when $\La\neq\ov\bF_\ell$ we use (a base change of) motivic sheaves \cite{Sch24, Sch25}; see \ref{ss:sheafexplain}.} Associated with $I_z$ and $V_z$, Fargues--Scholze \cite{FS21} constructed a \emph{Hecke operator}
\begin{align*}
T_{V_z}:D(\Bun_{G,F_z},\La)\ra D(\Bun_{G,F_z},\La),
\end{align*}
which they used to define the automorphic-to-Galois direction of the Langlands correspondence for $G$ over $F_z$, up to semisimplification.

Using Theorem D, we prove that applying Hecke operators to the object
\begin{align*}
\loc_{Z,!}\!\La\in D\,\big(\!\prod_{z\in Z}\Bun_{G,F_z},\La\big) = \bigotimes_{z\in Z}D(\Bun_{G,F_z},\La)
\end{align*}
recovers the cohomology of the moduli of global shtukas at infinite level. Write
\begin{align*}
\pi_{\infty Z}:(\Sht^I_{G,V,\infty Z})_{\ov\bF_q}\ra U^I_{\ov\bF_q}
\end{align*}
for the moduli of global shtukas with infinite level at $Z$, which has an action of $G(F_Z)\coloneqq\prod_{z\in Z}G(F_z)$, and write $\cS_V$ for its intersection cohomology sheaf. Forthcoming work of Eteve--Gaitsgory--Genestier--Lafforgue will endow $\pi_{\infty Z,!}\cS_V$ with the structure of a representation of the Weil group $W_U^I$ on the level of derived categories.
\begin{thmE}
After restricting to $\prod_{z\in Z}W_{F_z}^{I_z}$, we have a natural isomorphism
\begin{align*}
\pi_{\infty Z,!}\cS_V \cong i_1^*(T_V(\loc_{Z,!}\!\La)),
\end{align*}
where $i_1$ denotes the product of the natural open embeddings $*/\ul{G(F_z)}\hra\Bun_{G,F_z}$, and $T_V$ is the endofunctor $\bigotimes_{z\in Z}T_{V_z}$ of $\bigotimes_{z\in Z}D(\Bun_{G,F_z},\La)$.
\end{thmE}
\begin{rems*}\hfill
\begin{enumerate}[1)]
\item If one is only interested in Theorem E after taking cohomology groups, it suffices to use work of Xue \cite{Xue20b} instead of the forthcoming work of Eteve--Gaitsgory--Genestier--Lafforgue. For example, this is the case for the next remark.
\item One can use Theorem E to recover all of the representation-theoretic results in \cite{LH23}; see Remark \ref{rems:gl_fs}.2).
\end{enumerate}
\end{rems*}
In light of Theorem E, it is natural to ask how to interpret $\loc_{Z,!}\!\La$ on the Galois side of the Langlands correspondence. Write $\LS_{\wh{G},F_z}$ for the moduli of continuous $\wh{G}$-valued representations of $W_{F_z}$ over $\La$ \cite{DHKM25, FS21, Zhu20}, which is an algebraic stack of finite type over $\La$. When the order of $\pi_0(Z(G)_{\ov{F}})$ is invertible in $\La$, Fargues--Scholze \cite[Conjecture X.3.5]{FS21} conjecture that there is a canonical equivalence of categories
\begin{align*}
\bL_{\psi_z}:D(\Bun_{G,F_z},\La)\ra^\sim\Ind D_{\coh}^{\qc}(\LS_{\wh{G},F_z})_{\Nilp},
\end{align*}
where $\Ind D_{\coh}^{\qc}(\LS_{\wh{G},F_z})_{\Nilp}$ is the derived category of \emph{ind-coherent sheaves} on $\LS_{\wh{G},F_z}$ with \emph{nilpotent singular support}, and $\bL_{\psi_z}$ depends on the choice of a maximal unipotent subgroup $N$ of $G$ along with a generic continuous character $\psi_z:N(F_z)\ra\La^\times$.

The following conjecture was suggested by Scholze. Let $\psi:N(\bA)\ra\La^\times$ be a generic continuous character trivial on $N(F)$, and for all $z$ in $Z$, take $\psi_z=\psi|_{N(F_z)}$.

\begin{conjF}
Assume Fargues--Scholze's conjecture, so that $\bigotimes_{z\in Z}\bL_{\psi_z}$ yields 
\begin{align*}
\bigotimes_{z\in Z}D(\Bun_{G,F_z},\La)\ra^\sim\bigotimes_{z\in Z}\Ind D_{\coh}^{\qc}(\LS_{\wh{G},F_z})_{\Nilp} = \Ind D^{\qc}_{\coh}\,\big(\!\prod_{z\in Z}\LS_{\wh{G},F_z}\big)_{\Nilp}.
\end{align*}
Under this equivalence, $\loc_{Z,!}\!\La$ corresponds to $\lres_{Z,*}(\om_{\LS_{\wh{G},U}})$, where $\LS_{\wh{G},U}$ denotes the moduli of continuous $\wh{G}$-valued representations of $W_U$ over $\La$ \cite{Zhu20}, and
\begin{align*}
\lres_Z:\LS_{\wh{G},U}\ra\prod_{z\in Z}\LS_{\wh{G},F_z}
\end{align*}
denotes the restriction map.
\end{conjF}
\begin{rems*}\hfill
\begin{enumerate}[1)]
\item For $U=C$ and $\La=\ov\bQ_\ell$, Conjecture F specializes to asking for a natural isomorphism $C_c(G(F)\bs G(\bA)/G(\bO),\ov\bQ_\ell)\cong\Ga(\LS_{\wh{G},C},\om_{\LS_{\wh{G},C}})$; see Example \ref{exmp:AGKRRV}. Previously, Arinkin--Gaitsgory--Kazhdan--Raskin--Rozenblyum--Varshavsky \cite{AGKRRV20a} conjectured this, and recently Gaitsgory--Raskin \cite{GR25} proved a weak version thereof.

More generally, we use Theorem E to prove that Conjecture F implies a conjecture of Zhu \cite[Conjecture 4.49]{Zhu20} after restricting to $\prod_{z\in Z}W_{F_z}^{I_z}$; see Proposition \ref{ss:spectralcohomologyformula}.

\item Given the previous remark, we regard Conjecture F as \emph{the} Langlands conjecture for $G$ over $F$ when ramification is allowed at places in $Z$. This fulfills a suggestion of Fargues--Scholze \cite[\S7]{FS25} that the global Langlands conjecture ought to study an object that encodes automorphic forms on $G_b$ for all $b$ in $B(F,G)$.

  In the number field setting, Fargues \cite[section 7]{Far25} has already emphasized the importance of studying the analogous object $\ov{\pi}_{\HT,!}\La$, though there are more restrictions on the $b$ in $B(\bQ,G)$ that appear. A more detailed study of $\ov\pi_{\HT,!}\La$ will appear in forthcoming work of Caraiani--Hamann--Zhang.
  
\item Conjecture F fits well with the interpretation of Langlands as an isomorphism of \emph{$4$-dimensional topological quantum field theories}. In particular, it fits well with the recent interpretation of relative Langlands as an identification of \emph{boundary theories} \cite{BZSV24}; we will investigate this in future work.

\item For a version of $\Bun_{G,F}$ adapted to the \emph{global Kaletha set} \cite{Dillery23b} instead of $B(F,G)$, see Appendix C, written by Peter Dillery. This builds on work of Fargues \cite{Fargues} that adapts $\Bun_{G,F_z}$ to the \emph{local Kaletha set} \cite{Dil23a}.
\end{enumerate}
\end{rems*}
We conclude by proving Conjecture F in the commutative case:
\begin{thmG}
When $G$ is a (not necessarily split) torus, Conjecture F holds.
\end{thmG}

\subsection*{Proof sketches}
We start with part of Theorem A. We prove that $\Bun_{G,U}$ and $\Bun_{G,F}$ are small v-stacks by locally choosing a Frobenius-splitting on $U$ and analyzing affinoid charts of $U^{\an}_S$, which reduces the problem to a v-descent result of Scholze--Weinstein \cite[Lemma 17.1.8]{SW20} for vector bundles on perfectoid spaces.

Next, we turn to Theorem D. The $S$-points of $(\Sht^I_{G,V})^\Diamond$ involve $G$-bundles on the algebraic curve $C_R$, while one can glue in the analytic topology to show that $S$-points of the fiber product in Theorem D involve $G$-bundles on the analytification $C^{\an}_S$. To reconcile this, we prove the following GAGA result in Appendix \ref{s:GAGA}:
\begin{thmH}
Let $S=\Spa(R,R^+)$ be a sousperfectoid space, and let $X$ be a smooth proper scheme over $R$. Then the analytification functor
\begin{align*}
\{\mbox{vector bundles on X}\}\ra\{\mbox{vector bundles on }X^{\an}_S\}
\end{align*}
is an equivalence of categories.
\end{thmH}
\begin{rem*}
While we were finalizing this paper, Theorem H was proved independently by Wang \cite[Theorem 4.3.1]{Wan26} along very similar lines; evidently, it is of independent interest. For example, as an application we prove that algebraic stacks of Higgs bundles agree with their analytification, answering a question of Heuer--Xu \cite[Remark 8.1.2]{HX24}; see Theorem \ref{ss:HX}. This was also proved independently by Wang.
\end{rem*}
Theorem H provides a lot of control. In addition to proving Theorem D, we use it to prove Theorem C as follows. An $\ov\bF_q$-point of $\Bun_{G,U}$ involves a $G$-bundle $\sG$ on $U^{\an}_{\Spa\ov\bF_q\lp{t^{1/q^\infty}}}$ equipped with a descent datum with respect to the diagram
\begin{align*}
&U_{\Spa{\ov\bF_q\lp{t^{1/q^\infty}}}\times_{\ov\bF_q}\Spa{\ov\bF_q\lp{t^{1/q^\infty}}}\times_{\ov\bF_q}\Spa{\ov\bF_q\lp{t^{1/q^\infty}}}}^{\an}\\
    &\rrra U_{\Spa{\ov\bF_q\lp{t^{1/q^\infty}}}\times_{\ov\bF_q}\Spa{\ov\bF_q\lp{t^{1/q^\infty}}}}^{\an}\rightrightarrows U_{\Spa{\ov\bF_q\lp{t^{1/q^\infty}}}}^{\an}.
\end{align*}
A result of Gleason--Ivanov--Zillinger \cite[Theorem 3.15]{GIZ25} implies that the restriction of $\sG$ along $\Spa F_z\ra U$ arises uniquely from a $G$-bundle on $\breve{F}_z$. The latter are all trivial, so we can glue in the analytic topology to assume that $U=C$. Finally, we can apply Theorem H and a result of D. Kim \cite[Theorem 4.9]{Kim24} to conclude that $\sG$ arises uniquely from a $G$-bundle on $C_{\ov\bF_q}$.

We now turn to Theorem B. By proving a description of pro-\'etale $\ul{\sO_U}$-local systems analogous to Kedlaya--Liu's description \cite[Theorem 8.5.12]{KL15} of pro-\'etale $\ul{\bQ_p}$-local systems, we show that $\Bun_{G,F}$ has a natural open substack isomorphic to $*/\ul{G(F)}$. From here, twisting and Theorem C let us show that $\Bun_{G,F}$ has an open substack isomorphic to the disjoint union in Theorem B. To prove that this is all of $\Bun^{\semis}_{G,F}$, we use Beauville--Laszlo uniformization \cite[Proposition III.3.1]{FS21} to lift geometric points from $\Bun^{\semis}_{G,F_z}$ to $\LocSht^{I_z}_{G,V_z}$ for some $V_z$. Combining this with Theorem D yields geometric points of $\Sht^I_{G,V}$, and we conclude using a standard nonemptiness criterion for $\Sht^I_{G,V}$. Similarly, we finish proving Theorem A by using the charts provided by Beauville--Laszlo uniformization and Theorem D.

In an ideal world, Theorem E would follow immediately from Theorem D. For $\La=\ov\bF_\ell$, one can successfully execute this using the \'etale sheaf theory of \cite{Sch17}, but for general $\La$ one has to compare the motivic sheaf theory of \cite{Sch24} with classical \'etale $\ell$-adic sheaf theories for algebraic varieties. We develop some tools for facilitating these comparisons in Appendix \ref{s:berkovichmotives}.

Finally, we prove Theorem G by using work of Langlands \cite{Lan97} on his conjectures for tori $T$ to prove an explicit description of $\LS_{\prescript{L}{}T,U}$.

\subsection*{Outline}
In \S\ref{s:vectorbundles}, we study vector bundles on the global Hartl--Pink curve. In \S\ref{s:Gbundles}, we use the Tannakian description of $G$-bundles to convert results from \S\ref{s:vectorbundles} into part of Theorem A, part of Theorem B, and Theorem C. In \S\ref{s:shtukas}, we introduce shtukas and prove Theorem D. In \S\ref{s:geometric}, we use Theorem D to finish proving Theorem A and Theorem B, as well as explain the relation with Igusa varieties. In \S\ref{s:sheaf}, we state a forthcoming result of Eteve--Gaitsgory--Genestier--Lafforgue and prove Theorem E. Finally, in \S\ref{s:langlands} we discuss Conjecture F and prove Theorem G.

In Appendix \ref{s:GAGA}, we prove Theorem H, which is used throughout the paper. In Appendix \ref{s:berkovichmotives}, we prove basic results about the (overconvergent) motivic sheaves of \cite{Sch24, Sch25}, which we use in \S\ref{s:sheaf} and \S\ref{s:langlands}. Finally, in Appendix \ref{s:dillery}, written by Peter Dillery, we consider a version of $\Bun_{G,F}$ adapted to the global Kaletha set.

\subsection*{Notation}
Except for in Appendix \ref{s:GAGA}, all rings and stacks are classical (i.e. not derived) unless otherwise specified.

By nonarchimedean field, we mean a topological field complete with respect to a rank-$1$ nonarchimedean valuation. For any ring homomorphism $A\ra B$, write $A^\sim$ for the integral closure of the image of $A$ in $B$. For any scheme $X$ and affine group $G$ over $X$, write $\Rep_X{G}$ for the category of representations of $G$ in vector bundles on $X$.

We view schemes as adic spaces via the fully faithful functor sending $\Spec{A}$ to $\Spa(A,\bZ^\sim)$ for all rings $A$, where $A$ is endowed with the discrete topology \cite[p.~64]{Sch20}. For any adic space $X$ over $\Spa\bZ_p$, write $X^\Diamond$ for the associated small v-sheaf over $\bF_p$ as in \cite[Lemma 18.1.1]{SW20}.

Unless otherwise specified, all products are over $\bF_q$. For any prestack $X$ on
\begin{align*}
\{\mbox{affine schemes over }\bF_q\},
\end{align*}
write $\Frob_X:X\ra X$ for its absolute $q$-Frobenius endomorphism, and write $X^{\perf}$ for its limit perfection. The transition morphisms for our ind-schemes are required to be closed embeddings. Write $\Perf$ for the category of affinoid perfectoid spaces over $\bF_p$, and for any small v-stack $X$ over $\bF_q$, write $\Frob_X:X\ra X$ for its absolute $q$-Frobenius endomorphism. 

We view derived categories as stable $\infty$-categories, and we view all functors between derived categories as derived functors. Write $\Sym$ for the $\infty$-category of presentably closed symmetric monoidal stable $\infty$-categories.

For any ring $A$, write $D(A)$ for the derived category of $A$-modules, and write $D_{\perf}(A)\subseteq D(A)$ for the full subcategory of perfect objects. For any derived algebraic stack $X$, write $D_{\qcoh}(X)$ for its derived category of quasicoherent sheaves, and write $D_{\perf}(X)\subseteq D(X)$ for the full subcategory of perfect objects. Write $X^{\cl}$ for the underlying classical algebraic stack of $X$.

For any topological space $X$, write $X$ for the associated condensed set as in \cite[Proposition 1.7]{Sch19}.

\subsection*{Acknowledgements}
We thank Peter Scholze for many helpful suggestions about this project; for example, he suggested Conjecture F, as well as the proof strategy of Theorem H. We also thank Peter Dillery for suggesting that a global analogue of (\ref{eq:FFcurve}) should exist, for his collaboration in an early stage of this project, and for providing Appendix C. Our intellectual debt to Laurent Fargues and Xinwen Zhu should be clear. Finally, we thank Arnaud Eteve, Ian Gleason, Linus Hamann, Pol van Hoften, Alexander Petrov, Juan Esteban Rodr\'iguez Camargo, Ziquan Yang, and Zhiwei Yun for useful discussions.

During the completion of this work, the author was partially supported by NSF Grant \#DMS2303195 and the Max Planck Institute for Mathematics. We thank MPIM for its hospitality and excellent working conditions.

\section{Vector bundles on the global Hartl--Pink curve}\label{s:vectorbundles}
In this section, we introduce and study vector bundles on the global Hartl--Pink curve. They are defined using analytifications, so we start by discussing the latter in terms of adic spaces. For example, we prove v-descent for vector bundles on analytifications. This entire discussion applies to any smooth scheme over $\bF_q$.

We then specialize to smooth curves $U$ over $\bF_q$, which are the main case of interest in this paper. We prove a description of pro-\'etale $\ul{\sO_U}$-local systems in terms of vector bundles on the global Hartl--Pink curve, analogous to Kedlaya--Liu's description of pro-\'etale $\ul{\bQ_p}$-local systems. Finally, we prove a vector bundle version of Theorem C.

\subsection{}\label{ss:Huberanalytification}
Analytification has the following interpretation in the category of adic spaces.
\begin{defn*}
  Let $X$ be a scheme. Let $S$ be an adic space over $X$, and let $Y$ be a scheme over $X$. Write $Y^{\an}_S$, if it exists, for the fiber product $Y\times_XS$ in the category of adic spaces.
\end{defn*}
By replacing $X$ and $Y$ with affine open covers, the $\Spec$-global sections adjunction implies that $Y^{\an}_S$ equals the ``fiber product'' $Y\times_XS$ in the sense of \cite[(3.8)]{Hub94}.

\subsection{}\label{ss:nosheafification}
We will use the following generalization of \cite[Lemma 4.4]{LH23}, which describes the functor of points of analytifications without needing to sheafify in the analytic topology. Let $D$ be a noetherian ring. Let $S$ be an adic space over $D$, and let $Y$ be a quasiprojective scheme over $D$.
\begin{lem*}
For all adic spaces $\Spa(A,A^+)$ over $S$, morphisms $\Spa(A,A^+)\ra Y^{\an}_S$ over $S$ are equivalent to morphisms $\Spec{A}\ra Y$ over $D$.
\end{lem*}
\begin{proof}
  Since $D$ is noetherian and $Y$ is quasiprojective over $D$, there exist finitely many homogeneous polynomials $f_1,\dotsc,f_l$ and $g_1,\dotsc,g_m$ in $D[T_0,\dotsc,T_d]$ such that $Y$ is the locus in $\bP_D^d$ where $f_a(T_0,\dotsc,T_d)$ vanishes for all $1\leq a\leq l$ and $g_b(T_0,\dotsc,T_d)$ vanishes nowhere for all $1\leq b\leq m$.

By the universal property of $Y^{\an}_S$, morphisms $\Spa(A,A^+)\ra Y^{\an}_S$ over $S$ are equivalent to morphisms $\Spa(A,A^+)\ra Y$ over $D$. The $\Spec$-global sections adjunction implies that the latter are equivalent to the data of a line bundle $\sL$ on $\Spa(A,A^+)$ equipped with sections $s_0,\dotsc,s_d$ generating $\sL$ such that $f_a(s_0,\dotsc,s_d)$ vanishes for all $1\leq a\leq l$ and $g_b(s_0,\dotsc,s_d)$ vanishes nowhere for all $1\leq b\leq m$. Finally, \cite[Theorem 1.4.2]{Ked19} shows these are equivalent to morphisms $\Spec{A}\ra Y$ over $D$.
\end{proof}

\subsection{}\label{ss:sheafytensorproduct}
The following lemma lets us compute the global sections of analytifications.
\begin{lem*}
  Let $Y=\Spa(A,A^+)$ be a smooth affinoid adic space over a nonarchimedean field $K$, let $\Spa(R,R^+)$ be an affinoid sousperfectoid adic space over $\Spa{K}$, and write $(A\wh\otimes_KR,(A^+\wh\otimes_{K^\circ}R^+)^\sim)$ for the base change of $(K,K^\circ)\ra(A,A^+)$ to $(R,R^+)$. Then $\Spa(A\wh\otimes_KR,(A^+\wh\otimes_{K^\circ}R^+)^\sim)$ is a sousperfectoid adic space.
\end{lem*}
\begin{proof}
  Since $Y$ is smooth over $K$, \cite[(2.2.8)]{Hub96} implies that there is a finite rational open cover $\{Y_i\}_{i=1}^m$ of $Y$ that are rational open subspaces of finite \'etale covers of closed unit polydisks over $K$. For all $1\leq i\leq m$, write $Y_i=\Spa(A_i,A_i^+)$. Then $A_i\wh\otimes_KR$ is a rational localization of a finite \'etale algebra over $R\ang{T_1,\dotsc,T_d}$, so \cite[Proposition 6.3.3]{SW20} shows that $A_i\wh\otimes_KR$ is sousperfectoid and hence strongly sheafy as in \cite[Definition 4.1]{HK25}.

Write $(A_i\wh\otimes_KR,(A_i^+\wh\otimes_{K^\circ}R^+)^\sim)$ for the base change of $(K,K^\circ)\ra(A_i,A^+_i)$ to $(R,R^+)$. Then $\{\Spa(A_i\wh\otimes_KR,(A_i^+\wh\otimes_{K^\circ}R^+)^\sim)\}_{i=1}^m$ is a finite rational open cover of
  \begin{align*}
    \Spa(A\wh\otimes_KR,(A^+\wh\otimes_{K^\circ}R^+)^\sim),
  \end{align*}
and its associated \v{C}ech complex is obtained by applying $-\wh\otimes_KR$ to the \v{C}ech complex associated with $\{Y_i\}_{i=1}^m$. Because $A$ is sheafy, the latter has cohomology $A$ in zeroth degree and $0$ elsewhere. The exactness of $-\wh\otimes_KR$ implies that the \v{C}ech complex associated with $\{\Spa(A_i\wh\otimes_KR,(A_i^+\wh\otimes_{K^\circ}R^+)^\sim)\}_{i=1}^m$ has cohomology $A\wh\otimes_KR$ in zeroth degree and $0$ elsewhere, so the result follows from \cite[Corollary 4.5]{HK25}\footnote{While \cite{HK25} works over $\bQ_p$, the proof of \cite[Corollary 4.5]{HK25} does not use this assumption.}.
\end{proof}

\subsection{}\label{ss:analytification}
We will analytify smooth varieties as follows. Let $Y=\Spec K[T_1,\dotsc,T_d]/I$ be a smooth affine scheme over a nonarchimedean field $K$, and let $S=\Spa(R,R^+)$ be an affinoid sousperfectoid space over $\Spa{K}$. For all pseudouniformizers $\vpi$ of $K$, write $B_{S,\vpi}$ for the topological ring $R\ang{\vpi T_1,\dotsc,\vpi T_d}/I$, and write $B^+_{S,\vpi}$ for the integral closure in $B_{S,\vpi}$ of the image of $R^+\ang{\vpi T_1,\dotsc,\vpi T_d}$.
\begin{prop*}
The pre-adic space $Y_{S,\vpi}\coloneqq\Spa(B_{S,\vpi},B^+_{S,\vpi})$ is a sousperfectoid adic space. Moreover, $Y^{\an}_S$ equals $\varinjlim_\vpi Y_{S,\vpi}$, where $\vpi$ runs over pseudouniformizers of $K$ and the transition morphisms are open embeddings. Consequently, $Y^{\an}_S$ exists and is sousperfectoid.
\end{prop*}
\begin{proof}
Write $A$ for $K\ang{\vpi T_1,\dotsc,\vpi T_d}/I$, and write $A^+$ for the integral closure in $A$ of the image of $K^\circ\ang{\vpi T_1,\dotsc,\vpi T_d}$. Since $Y$ is smooth over $K$, the affinoid adic space $\Spa(A,A^+)$ is smooth over $K$. Hence the first statement follows from Lemma \ref{ss:sheafytensorproduct} and $(B_{S,\vpi},B_{S,\vpi}^+)$ being the base change of $(K,K^\circ)\ra(A,A^+)$ to $(R,R^+)$. The second statement follows from the universal property of $Y^{\an}_S$.
\end{proof}

\subsection{}\label{ss:analytificationperfection}
Next, we specialize to smooth varieties over $\bF_q$. Let $Y=\Spec\bF_q[T_1,\dotsc,T_d]/I$ be a smooth affine scheme over $\bF_q$, and let $S=\Spa(R,R^+)$ be an affinoid perfectoid space over $\bF_q$. Any pseudouniformizer $\vpi$ of $R$ induces a morphism
\begin{align*}
S\ra\Spa\bF_q\lp{\vpi^{1/q^\infty}}.
\end{align*}
By applying Proposition \ref{ss:analytification} to $Y_{\bF_q\lp{\vpi^{1/q^\infty}}}$, we see that $Y^{\an}_S$ exists and equals $\bigcup_\vpi Y_{S,\vpi}$, where $\vpi$ runs over pseudouniformizers of $R$.

To prove v-descent for vector bundles on the analytification of $Y$, we use Frobenius to construct the following natural perfectoid cover. Note that $(\Frob_Y)_S^{\an}$ sends $Y_{S,\vpi^{1/q}}$ to $Y_{S,\vpi}$. Write $B^{+,\perf}_{S,\vpi}$ for the $\vpi$-adic completion of $\varinjlim_iB^+_{S,\vpi^{1/q^i}}$, where $i$ runs over non-negative integers and the transition maps are given by $(\Frob_Y)_S^{\an,*}$. Write $B^{\perf}_{S,\vpi}$ for $B^{+,\perf}_{S,\vpi}[\textstyle\frac1\vpi]$. Finally, write $(Y^{\an}_S)^{\perf}$, if it exists, for the limit $\varprojlim_iY^{\an}_S$ in the category of uniform analytic adic spaces, where the transition maps are given by $(\Frob_Y)^{\an}_S$.
\begin{prop*}
  The pre-adic space $Y^{\perf}_{S,\vpi}\coloneqq\Spa(B^{\perf}_{S,\vpi},B^{\perf,+}_{S,\vpi})$ is affinoid perfectoid. Moreover, $Y^{\perf}_{S,\vpi}$ equals the fiber product $Y_{S,\vpi}\times_{Y^{\an}_S}(Y^{\an}_S)^{\perf}$ in the category of uniform analytic adic spaces, so $(Y^{\an}_S)^{\perf}=\bigcup_\vpi Y^{\perf}_{S,\vpi}$ exists and is perfectoid. Finally, the map $B_{S,\vpi}\ra B_{S,\vpi}^{\perf}$ of topological $B_{S,\vpi}$-modules splits.
\end{prop*}
\begin{proof}
  For the first statement, note that $(B^{\perf}_{S,\vpi},B^{+,\perf}_{S,\vpi})$ is the colimit of
  \begin{align*}
    \{(B_{S,\vpi^{1/q^i}},B^+_{S,\vpi^{1/q^i}})\}_i
  \end{align*}
in the category of uniform Tate Huber pairs. Therefore \cite[Proposition 3.5]{Sch17} indicates that it suffices to check that $B^{\perf}_{S,\vpi}$ is perfect. Because $S$ is perfectoid, $\Frob_S$ induces an isomorphism $Y_{S,\vpi}\ra^\sim Y_{S,\vpi^{1/q}}$. Composing this isomorphism with $(\Frob_Y)^{\an}_S:Y_{S,\vpi^{1/q}}\ra Y_{S,\vpi}$ yields the absolute $q$-Frobenius endomorphism, which implies that $B_{S,\vpi}^{\perf}$ is indeed perfect.

For the second statement, the above shows that $Y^{\perf}_{S,\vpi}$ equals the limit $\varprojlim_i Y_{S,\vpi^{1/q^i}}$ in the category of uniform analytic adic spaces. The universal property of $Y_S^{\an}$ implies that the square
\begin{align*}
  \xymatrixcolsep{2cm}
  \xymatrix{Y_{S,\vpi^{1/q}}\ar[r]^-{(\Frob_Y)_S^{\an}}\ar@{^{(}->}[d] & Y_{S,\vpi}\ar@{^{(}->}[d]\\
  Y^{\an}_S\ar[r]^-{(\Frob_Y)_S^{\an}} & Y^{\an}_S}
\end{align*}
is cartesian, so taking $\varprojlim_i$ in the category of uniform analytic adic spaces yields the desired result.

For the third statement, the smoothness of $Y$ implies that $\Frob_{Y,*}\bF_q[T_1,\dotsc,T_d]/I$ is a finite projective $\bF_q[T_1,\dotsc,T_d]/I$-module \cite[Lemma 1.1.1]{BK05}\footnote{While \cite{BK05} works over an algebraically closed field, the proofs of \cite[Lemma 1.1.1]{BK05} and \cite[Proposition 1.1.6]{BK05} do not use this assumption.}. Since $Y$ is also affine, \cite[Proposition 1.1.6]{BK05} shows that $Y$ is Frobenius-split.

Note that the square
\begin{align*}
  \xymatrixcolsep{1.5cm}
  \xymatrix{\bF_q[T_1,\dotsc,T_d]/I\ar[r]^-{\Frob_Y^*}\ar[d] & \Frob_{Y,*}\bF_q[T_1,\dotsc,T_d]/I\ar[d]\\
  B_{S,\vpi}\ar[r]^-{(\Frob_Y)_S^{\an,*}} & B_{S,\vpi^{1/q}}}
\end{align*}
is a pushout in the category of rings. Therefore applying $-\otimes_{\bF_q[T_1,\dotsc,T_d]/I}B_{S,\vpi}$ to an $\bF_q[T_1,\dotsc,T_d]/I$-module splitting of the top arrow yields a topological $B_{S,\vpi}$-module splitting of the bottom arrow. Finally, taking the completed colimit of $\{B_{S,\vpi^{1/q^i}}\}_i$ yields a topological $B_{S,\vpi}$-module splitting of $B_{S,\vpi}\ra B^{\perf}_{S,\vpi}$.
\end{proof}

\subsection{}\label{ss:vectorbundlesdescent}
We now prove v-descent for vector bundles on (open subspaces of) analytifications. Let $Y$ be a smooth scheme over $\bF_q$, and let $V_S$ be an open subspace of $Y^{\an}_S$. For all affinoid perfectoid spaces $S'$ over $S$, write $V_{S'}$ for the preimage of $V_S$.
\begin{thm*}
The presheaf of categories on $\Perf_S$ given by
  \begin{align*}
    S'\mapsto\{\mbox{vector bundles on }V_{S'}\}
  \end{align*}
satisfies v-descent.
\end{thm*}
\begin{proof}
  By replacing $Y$ and $V_S$ with open covers, we can assume that
  \begin{align*}
    Y=\Spec\bF_q[T_1,\dotsc,T_d]/I
  \end{align*}
  is affine and that $V_S=\Spa(A,A^+)$ is affinoid. Let $S'\ra S$ be an affinoid perfectoid v-cover. Since $Y^{\an}_S=\bigcup_\vpi Y_{S,\vpi}$, we see that $V_S$ lies in $Y_{S,\vpi}$ for some pseudouniformizer $\vpi$ of $R$. Write $V^{\perf}_S$ for the preimage of $V_S$ in $Y^{\perf}_{S,\vpi}$.

  Let $\sE'$ be a vector bundle on $V_{S'}$ with descent datum $\al$ with respect to
  \begin{align*}
    V_{S'\times_SS'\times_SS'}\rrra V_{S'\times_SS'}\rightrightarrows V_{S'}.
  \end{align*}
By Proposition \ref{ss:analytificationperfection}, pullback yields a vector bundle $\wt\sE'$ on $V_{S'}^{\perf}$ with commuting descent data $\wt\al$ with respect to the adic spaces
  \begin{align*}
    V^{\perf}_{S'\times_SS'\times_SS'}\rrra V^{\perf}_{S'\times_SS'}\rightrightarrows V^{\perf}_{S'}
  \end{align*}
  and $\be'$ with respect to the affinoid pre-adic spaces
  \begin{align*}
   V^{\perf}_{S'}\times_{V_{S'}}V^{\perf}_{S'}\times_{V_{S'}}V^{\perf}_{S',\vpi} \rrra V^{\perf}_{S'}\times_{V_{S'}}V^{\perf}_{S'}\rightrightarrows V^{\perf}_{S'}.
  \end{align*}
  Proposition \ref{ss:analytificationperfection} indicates that the $V^{\perf}_{(-)}$ are perfectoid, so \cite[Lemma 17.1.8]{SW20} enables us to descend $(\wt\sE',\wt\al)$ and $(\be,\wt\al)$ to a vector bundle $\wt\sE$ on $V^{\perf}_S$ with descent datum $\be$ with respect to the affinoid pre-adic spaces
  \begin{align*}
V^{\perf}_S\times_{V_S}V^{\perf}_S\times_{V_S}V^{\perf}_S \rrra V^{\perf}_S\times_{V_S}V^{\perf}_S\rightrightarrows V^{\perf}_S.
  \end{align*}
Finally, Proposition \ref{ss:analytificationperfection} shows that the map $B_{S,\vpi}\ra B^{\perf}_{B,\vpi}$ of topological $B_{S,\vpi}$-modules splits, so the map $A\ra A\wh\otimes_{B_{S,\vpi}}B^{\perf}_{B,\vpi}$ of topological $A$-modules also splits. Hence \cite[Tag 08XA]{stacks-project} and \cite[Tag 08XD]{stacks-project} enable us to descend $(\wt\sE,\be)$ to a vector bundle $\sE$ on $V_S$, as desired.
\end{proof}

\subsection{}\label{ss:Frobeniusfixedglobalsections}
The Frobenius-invariants of global sections of analytifications are given as follows. Write $A$ for the global sections of $Y$, endowed with the discrete topology.
\begin{prop*}
The natural map $\ul{A}(S)\ra\Ga(Y^{\an}_S,\sO_{Y^{\an}_S})^{\Frob_S^*=1}$ is an isomorphism.
\end{prop*}
\begin{proof}
  By replacing $Y$ with an open cover, we can assume that
  \begin{align*}
    Y = \Spec\bF_q[T_1,\dotsc,T_d]/I
  \end{align*}
is affine. Let $\{e_m\}_m$ be an $\bF_q$-basis of $A$ consisting of images of monomials. For all pseudouniformizers $\vpi$ of $R$, this induces an identification $\wh\bigoplus_mR\cdot e_m\ra^\sim B_{S,\vpi}$ via sending $T_i$ to $\vpi T_i$ for all $1\leq i\leq d$. Under this identification, $\Frob_S^*:B_{\vpi,S}\ra^\sim B_{\vpi^q,S}$ acts coordinatewise, so $\bigoplus_mR^{\Frob_S^*=1}\cdot e_m\ra^\sim B_{S,\vpi}^{\Frob_S^*=1}$. Now \cite[Corollary 3.1.4]{KL15}\footnote{While \cite[Corollary 3.1.4]{KL15} is stated for $q=p$, the proof applies verbatim for general $q$.} implies that $\Cont(\abs{S},\bF_q)\ra^\sim R^{\Frob_S^*=1}$, so the quasicompactness of $\abs{S}$ shows that
  \begin{align*}
    \ul{A}(S)=\Cont(\abs{S},A)=\Cont\big(\abs{S},\bigoplus_m\bF_q\cdot e_m\big)\ra^\sim\bigoplus_mR^{\Frob_S^*=1}\cdot e_m.
  \end{align*}
Finally, note that the global sections of $Y^{\an}_S$ are given by $\bigcap_\vpi B_{S,\vpi}$, where $\vpi$ runs over pseudouniformizers of $R$. Therefore taking $\bigcap_\vpi$ yields the claim.
\end{proof}

\subsection{}
Finally, we specialize to smooth curves over $\bF_q$, which are the main case of interest in this paper. Let $C$ be a geometrically connected smooth proper curve over $\bF_q$. Write $F$ for its function field, and write $\bA$ for its adele ring. Let $U$ be a dense open subscheme of $C$, and write $\bO_U$ for its ring of integral adeles. When $U$ equals $C$, we omit it from our notation.

For all closed points $v$ of $C$, write $\bF_v$ for its residue field, fix an embedding $\bF_v\ra\ov\bF_q$ over $\bF_q$, write $\cO_v$ for the completed local ring at $v$, and write $F_v$ for the fraction field of $\cO_v$. Let $A_v$ be one of $\{\cO_v,F_v\}$. By applying \cite[Proposition II.1.1]{FS21} to $\Spa{\bF_v}\times S$, we see that $\Spa{A_v}\times S$ is a sousperfectoid adic space.

We define vector bundles on the \emph{global} and \emph{local Hartl--Pink curves} as follows.
\begin{defn*}\hfill
  \begin{enumerate}[a)]
  \item Write $\Bun_U(S)$ for the category of vector bundles $\sE$ on $U^{\an}_S$ equipped with an isomorphism $\phi:\sE\ra^\sim\Frob^*_S\sE$.
  \item Write $\Bun_F(S)$ for the category $\varinjlim_U\Bun_U(S)$, where $U$ runs over dense open subschemes of $C$.
  \item Write $\Bun_{A_v}(S)$ for the category of vector bundles $\sE$ on $\Spa{A_v}\times S$ equipped with an isomorphism $\phi:\sE\ra^\sim\Frob^*_S\sE$.
  \end{enumerate}
\end{defn*}
Theorem \ref{ss:vectorbundlesdescent} indicates that $\Bun_U$ satisfies v-descent. Proposition \ref{ss:Frobeniusfixedglobalsections} implies that the transition morphisms in b) are faithful, so $\Bun_F$ also satisfies v-descent.

\subsection{}\label{ss:localbundles}
Usually, the local Hartl--Pink curve for $F_v$ is defined using $q^{\deg{v}}$-Frobenius instead of $q$-Frobenius. We now relate these two definitions. Write $X_{S,F_v}$ for the relative Fargues--Fontaine curve for $F_v$ as in \cite[Definition I.2.2]{FS21}.
\begin{prop*}
  The presheaf of categories $\Bun_{A_v}$ satisfies v-descent, and its base change to $\Perf_{\bF_v}$ is naturally isomorphic to 
  \begin{align*}
    S\mapsto\left\{
    \begin{tabular}{c}
      vector bundles $\sE$ on $\Spa{A_v}\times_{\bF_v}S$ equipped \\
      with an isomorphism $\phi:\sE\ra^\sim(\Frob_S^{\deg{v}})^*\sE$
    \end{tabular}
    \right\}.
  \end{align*}
  Consequently, the base change of $\Bun_{F_v}$ to $\Perf_{\bF_v}$ is naturally isomorphic to
  \begin{align*}
  S\mapsto\{\mbox{vector bundles on }X_{S,F_v}\}.
\end{align*}
\end{prop*}
\begin{proof}
  For the first statement, let $S'\ra S$ be an affinoid perfectoid v-cover. Then $\Spa\bF_v\times S'\ra\Spa\bF_v\times S$ is also an affinoid perfectoid v-cover, so the first statement follows from \cite[Proposition 19.5.3]{SW20}. The second statement follows from arguing as in the proof of \cite[Lemma 5.12]{LH23}. Finally, the third statement follows from 
  \begin{gather*}
    X_{S,F_v}=(\Spa{F_v}\times_{\bF_v}S)/\Frob^{\deg{v}}_S.\qedhere
  \end{gather*}
\end{proof}

\subsection{}\label{ss:ArtinSchreierWitt}
Thanks to Artin--Schreier--Witt, $\Bun_{\cO_v}$ enjoys the following description.
\begin{prop*}
The presheaf of categories $\Bun_{\cO_v}$ is naturally isomorphic to
\begin{align*}
  S\mapsto\{\mbox{pro-\'etale }\ul{\cO_v}\mbox{-local systems on }S\}.
\end{align*}
\end{prop*}
\begin{proof}
After replacing $(-)\lb{\pi}$ with $-\wh\otimes_{\bF_q}\bF_v\lb{\pi}$, where $\pi$ denotes a uniformizer of $\cO_v$, the proof proceeds as in \cite[Proposition 3.7]{LH23} and \cite[Theorem 3.12]{LH23}.
\end{proof}

\subsection{}
The following relationship between vector bundles on the global and local Hartl--Pink curves plays an important role in this paper. Write $Z$ for the closed complement $C\ssm U$. For all closed points $u$ of $U$, we have a natural morphism of adic spaces $\Spa{\cO_u}\ra U$, and for all $z$ in $Z$, we have a natural morphism of adic spaces $\Spa{F_z}\ra U$. Therefore pullback yields morphisms
\begin{align*}
  \loc_u:\Bun_U\ra\Bun_{\cO_u}\mbox{ and }\loc_z:\Bun_U\ra\Bun_{F_z}.
\end{align*}
\begin{defn*}\hfill
  \begin{enumerate}[a)]
  \item Write
    \begin{align*}
      \loc:\Bun_U\ra\prod_u\Bun_{\cO_u}\times\prod_z\Bun_{F_z}
    \end{align*}
    for the morphism induced by the $\{\loc_u\}_u$ and $\{\loc_z\}_z$, where $u$ runs over closed points of $U$, and $z$ runs over $Z$.
  \item Write $\Bun_\bA$ for the presheaf of categories $\varinjlim_U\textstyle\prod_u\Bun_{\cO_u}\times\textstyle\prod_z\Bun_{F_z}$, where $U$ runs over dense open subschemes of $C$.
  \item Write $\loc:\Bun_F\ra\Bun_\bA$ for the morphism obtained by taking $\varinjlim_U$ of a).
  \end{enumerate}
\end{defn*}
Proposition \ref{ss:localbundles} implies that $\Bun_\bA$ satisfies v-descent.

\subsection{}\label{ss:localcharts}
We will use the following lemma to reduce to the situation where $\Pic{U}=1$. Because $U$ is Dedekind, $\Pic{U}=1$ if and only if all vector bundles on $U$ are trivial. Since $U$ is a smooth curve, this implies that $U$ is affine.
\begin{lem*}
There exists an open cover $\{C_1,C_2\}$ of $C$ satisfying $\Pic{C_1}=\Pic{C_2}=1$ and hence an open cover $\{U_1,U_2\}$ of $U$ satisfying $\Pic{U_1}=\Pic{U_2}=1$.
\end{lem*}
\begin{proof}  
The finitude of $\bF_q$ implies that $\Pic{C}$ is finitely generated. Let $\sL_1,\dotsc,\sL_g$ be generators of $\Pic{C}$, and let $C_1$ be a dense open subscheme of $C$ such that $\sL_i|_{C_1}$ is trivial for all $1\leq i\leq g$. Then excision for class groups implies that $\Pic{C_1}=1$. Because $C\ssm C_1$ is finite, weak approximation shows that the order-of-vanishing map $F^\times\ra\bZ^{\oplus(C\ssm C_1)}$ is surjective, so for all $1\leq i\leq g$, there exists a divisor $D_i$ supported on $C_1$ such that $\sL_i\cong\sO_C(D_i)$. By taking $C_2=C\ssm\bigcup_{i=1}^g\supp{D_i}$, excision for class groups again implies that $\Pic{C_2}=1$. Finally, the analogous statements hold for $U_1=U\cap C_1$ and $U_2=U\cap C_2$ by excision for class groups.
\end{proof}

\subsection{}
We will use the following notion of local systems on $S$ with coefficients in $\sO_U$, which is local both in $S$ and in $U$. For any $\sE$ in $\QCoh(U)$, write $\ul{\sE}$ for the constant $\QCoh(U)$-valued sheaf on the site $*_{\text{pro-\'et}}$ induced by $\sE$.
\begin{defn*}
A \emph{pro-\'etale $\ul{\sO_U}$-local system on $S$} is an $\ul{\sO_U}$-module on the site $S_{\text{pro-\'et}}$ that, locally in $S_{\text{pro-\'et}}$, is of the form $\ul{\sE}$ for some vector bundle $\sE$ on $U$.
\end{defn*}
When $\Pic{U}=1$, note that this agrees with the notion of pro-\'etale $\ul{A}$-local systems, where $A$ denotes the global sections of $U$ endowed with the discrete topology.

\subsection{}\label{ss:Alocalsystems}
We now prove a description of pro-\'etale $\ul{\sO_U}$-local systems in terms of vector bundles on the global Hartl--Pink curve; this is the global analogue of a theorem of Kedlaya--Liu \cite[Theorem 8.5.12]{KL15}.

We say that an object in $\Bun_{F_v}(S)$ \emph{has slope zero} if, under the identification from Proposition \ref{ss:localbundles}, its image in $\Bun_{F_v}(\Spa\bF_v\times S)$ has slope zero Harder--Narasimhan polygon as in \cite[p.~74]{FS21}. Note that this agrees with the usual notion when $S$ lies over $\bF_v$. Moreover, for all $(\sE,\phi)$ in $\Bun_{G,\cO_v}(S)$, its image in $\Bun_{G,F_v}(S)$ has slope zero.
\begin{thm*}
We have an exact tensor equivalence of categories
  \begin{align*}
    \{\mbox{pro-\'etale }\ul{\sO_U}\mbox{-local systems on }S\}\ra^\sim\left\{
    \begin{tabular}{c}
      objects in $\Bun_U(S)$ whose image in \\
      $\Bun_{F_z}(S)$ has slope zero for all $z$ in $Z$
    \end{tabular}
    \right\}
  \end{align*}
given by $\bL\mapsto(\bL\otimes_{\ul{\sO_U}}\sO_{U^{\an}_S},\bL\otimes_{\ul{\sO_U}}(\Frob^{-1}_S)^*)$.
\end{thm*}
\begin{proof}
  By replacing $U$ with the open cover $\{U_1,U_2\}$ from Lemma \ref{ss:localcharts}, we can assume that $\Pic{U}=1$. Note that the above functor preserves tensor products and duals. Hence taking internal homs reduces full faithfulness to proving that, for all $\ul{A}$-local systems $\bL$ on $S$, the map
  \begin{align*}
    \Hom_{\ul{A}}(\ul{A},\bL)\ra\Hom_{\sO_{U^{\an}_S}}((\sO_{U^{\an}_S},(\Frob_S^{-1})^*),(\bL\otimes_{\ul{A}}\sO_{U^{\an}_S},\bL\otimes_{\ul{A}}(\Frob_S^{-1})^*))
  \end{align*}
  is a bijection. By replacing $S$ with a pro-\'etale cover and using Theorem \ref{ss:vectorbundlesdescent}, we can assume that $\bL=\ul{A}$. Then the result follows from Proposition \ref{ss:Frobeniusfixedglobalsections}.

For essential surjectivity, let $(\sE,\phi)$ be an object in $\Bun_U(S)$. By full faithfulness and pro-\'etale descent, it suffices to prove that there exists an affinoid perfectoid pro-\'etale cover $S'\ra S$ such that $(\sE,\phi)|_{S'}$ is trivial. Since $\loc_z(\sE,\phi)$ has slope zero for all $z$ in $Z$, \cite[Theorem II.2.19]{FS21} yields an affinoid perfectoid pro-\'etale cover $S'=\Spa(R',R'^+)\ra S$ such that, for all $z$ in $Z$, the base change $\loc_z(\sE,\phi)|_{S'}$ is trivial. Now $\{U^{\an}_{S'}\}\cup\{\Spa\cO_z\times S'\}_z$ is an open cover of $C^{\an}_{S'}$, so we can glue $(\sE,\phi)|_{S'}$ with the trivial object in $\Bun_{\cO_z}(S')$ for all $z$ in $Z$ to obtain an object $(\sE',\phi')$ in $\Bun_C(S')$. Theorem \ref{ss:appliedGAGA} shows that $(\sE',\phi')$ is the analytification of a vector bundle $\sE'^\alg$ on $C_{R'}$ equipped with an isomorphism $\phi'^{\alg}:\sE'^{\alg}\ra^\sim\Frob_{R'}^*\sE'^{\alg}$. After replacing $\Spec{R'}$ with a clopen cover, \cite[Lemme 3 du paragraphe I.3]{Laf97} indicates that $(\sE'^{\alg},\phi'^{\alg})$ is isomorphic to $(\sE^{\alg}|_{R'},\id_{\sE^{\alg}}\otimes_{\bF_q}(\Frob_{R'}^{-1})^*)$ for some vector bundle $\sE^{\alg}$ on $C$. Finally, $\Pic{U}=1$ implies that $\sE^{\alg}|_U$ is trivial.
\end{proof}

\subsection{}\label{ss:BunUschemepoints}
Let $B$ be a perfect $\bF_q$-algebra, endowed with the discrete topology. We conclude this section by proving that, in many situations, vector bundles on $U_B$ are equivalent to vector bundles on ``$U_{\Spd B}$''.

What do we mean by ``$U_{\Spd B}$''? Note that $\Spa{B\lp{t^{1/q^\infty}}}\ra\Spd{B}$ is surjective and representable in perfectoid spaces, so we have a natural diagram of perfectoids
\begin{align*}\label{eqn:SpdBresolution}
  \tag{$\star$}
  \begin{split}&\Spa{B\lp{t^{1/q^\infty}}}\times_{\Spd{B}}\Spa{B\lp{t^{1/q^\infty}}}\times_{\Spd{B}}\Spa{B\lp{t^{1/q^\infty}}}\\
    &\rrra\Spa{B\lp{t^{1/q^\infty}}}\times_{\Spd{B}}\Spa{B\lp{t^{1/q^\infty}}}\rightrightarrows\Spa{B\lp{t^{1/q^\infty}}}.\end{split}
\end{align*}
\begin{thm*}
  Assume that each connected component of $\Spec{B}$ is a valuation ring. Then analytification induces an exact tensor equivalence from
  \begin{align*}
    \{\mbox{vector bundles on }U_B\}
  \end{align*}
to vector bundles on $U_{\Spa{B\lp{t^{1/q^\infty}}}}^{\an}$ with descent datum with respect to
\begin{align*}\label{eqn:UBresolution}
  \tag{$\star\star$}\begin{split}&U_{\Spa{B\lp{t^{1/q^\infty}}}\times_{\Spd{B}}\Spa{B\lp{t^{1/q^\infty}}}\times_{\Spd{B}}\Spa{B\lp{t^{1/q^\infty}}}}^{\an}\\
  &\rrra U_{\Spa{B\lp{t^{1/q^\infty}}}\times_{\Spd{B}}\Spa{B\lp{t^{1/q^\infty}}}}^{\an}\rightrightarrows U_{\Spa{B\lp{t^{1/q^\infty}}}}^{\an}.\end{split}
\end{align*}
\end{thm*}
\begin{proof}
  For full faithfulness, let $\sE_1$ and $\sE_2$ be vector bundles on $U_B$, write $(\sE'_1,\al_1)$ and $(\sE'_2,\al_2)$ for the corresponding vector bundles on $U_{\Spa{B\lp{t^{1/q^\infty}}}}^{\an}$ with descent data with respect to (\ref{eqn:UBresolution}), and let $f':\sE'_1\ra\sE'_2$ be a morphism of vector bundles compatible with the descent data. By replacing $U$ with an open cover, we can assume that $U=\Spec{A}$ is affine. Write $A=\bF_q[T_1,\dotsc,T_d]/I$. For all affinoid perfectoid spaces $S=\Spa(R,R^+)$ over $\bF_q$, Lemma \ref{ss:sheafytensorproduct} shows that
  \begin{align*}
    U_{S,1}\coloneqq\Spa(R\ang{T_1,\dotsc,T_d}/I,(R^+\ang{T_1,\dotsc,T_d})^\sim)
  \end{align*}
  is an open subspace of $U^{\an}_S$. For all affinoid perfectoid spaces $S'$ over $S$, note that the preimage of $U_{S,1}$ in $U_{S'}^{\an}$ equals $U_{S',1}$, so restricting (\ref{eqn:UBresolution}) to $U_{\Spa{B\lp{t^{1/q^\infty}}},1}$ yields
\begin{align*}
&U_{\Spa{B\lp{t^{1/q^\infty}}}\times_{\Spd{B}}\Spa{B\lp{t^{1/q^\infty}}}\times_{\Spd{B}}\Spa{B\lp{t^{1/q^\infty}}},1}\\
&\rrra U_{\Spa{B\lp{t^{1/q^\infty}}}\times_{\Spd{B}}\Spa{B\lp{t^{1/q^\infty}}},1}\rightrightarrows U_{\Spa{B\lp{t^{1/q^\infty}}},1}.
\end{align*}
Since $I$ is finitely generated, for all positive integers $m$, we have
  \begin{align*}
   B\lp{t_1^{1/q^\infty},\dotsc,t_m^{1/q^\infty}}\ang{T_1,\dotsc,T_d}/I = \big(B[T_1,\dotsc,T_d]/I\big)\lp{t_1^{1/q^\infty},\dotsc,t_m^{1/q^\infty}},
  \end{align*}
so taking global sections yields
  \begin{align*}
    (A\otimes_{\bF_q}B)\lp{t_1^{1/q^\infty},t_2^{1/q^\infty},t_3^{1/q^\infty}}\llla(A\otimes_{\bF_q}B)\lp{t_1^{1/q^\infty},t_2^{1/q^\infty}}\leftleftarrows(A\otimes_{\bF_q}B)\lp{t_1^{1/q^\infty}}.
  \end{align*}
Therefore \cite[Theorem 4.9]{Kim24} indicates that the restriction of $f'$ to $U_{\Spa B\lp{t^{1/q^\infty}},1}$ arises uniquely from a morphism $f:\sE_1\ra\sE_2$ of vector bundles on $U_B$. Since morphisms of vector bundles on $U^{\an}_S$ are determined by their restrictions to $U_{S,1}$, this yields the desired result.

For essential surjectivity, let $\sE'$ be a vector bundle on $U^{\an}_{\Spa{B\lp{t^{1/q^\infty}}}}$ with descent datum $\al$ with respect to (\ref{eqn:UBresolution}). For all $z$ in $Z$, pullback yields a vector bundle $\sE'_z$ on $\Spa{F_z}\times\Spa{B\lp{t^{1/q^\infty}}}$ with descent datum $\al_z$ with respect to the evaluation of $\Spa{F_z}\times(-)$ on (\ref{eqn:SpdBresolution}), and applying \cite[Theorem 3.15]{GIZ25} and \cite[Remark 3.14]{GIZ25} to $\bF_v\otimes_{\bF_q}B$ shows that $(\sE_z',\al_z)$ arises uniquely from a vector bundle $\sE_z$ on $\Spec{F_z\wh\otimes_{\bF_q}B}$. After replacing $\Spec{B}$ by a clopen cover, \cite[Theorem 6.1]{Iva23} indicates that $\sE_z$ is trivial. Now $\{U^{\an}_{(-)}\}\cup\{\Spa{\cO_z}\times(-)\}_z$ is an open cover of $C^{\an}_{(-)}$, so we can glue $(\sE',\al)$ to the trivial vector bundle on $\Spa\cO_z\times\Spa{B\lp{t^{1/q^\infty}}}$ with descent datum with respect to the evaluation of $\Spa{\cO_z}\times(-)$ on (\ref{eqn:SpdBresolution}). This yields a vector bundle $\wh\sE'$ on $C^{\an}_{\Spa{B\lp{t^{1/q^\infty}}}}$ with descent datum $\wh\al$ with respect to 
   \begin{align*}
  &C_{\Spa{B\lp{t^{1/q^\infty}}}\times_{\Spd{B}}\Spa{B\lp{t^{1/q^\infty}}}\times_{\Spd{B}}\Spa{B\lp{t^{1/q^\infty}}}}^{\an}\\
    &\rrra C_{\Spa{B\lp{t^{1/q^\infty}}}\times_{\Spd{B}}\Spa{B\lp{t^{1/q^\infty}}}}^{\an}\rightrightarrows C_{\Spa{B\lp{t^{1/q^\infty}}}}^{\an}.
  \end{align*}
  Theorem \ref{ss:appliedGAGA} shows that this corresponds to a vector bundle $\wh\sE'^{\alg}$ on $C_{B\lp{t_1^{1/q^{\infty}}}}$ with, for all positive integers $n$, a descent datum $\wh\al_n$ with respect to
   \begin{align*}\label{eqn:algCBresolution}
  \tag{$\star\star\star$}\begin{split}&C_{B\lb{t_1^{1/q^\infty},t_2^{1/q^\infty},t_3^{1/q^\infty}}\ang{t_1^n/t_2t_3,t_2^n/t_1t_3,t_3^n/t_1t_2}[1/t_1t_2t_3]}\\
    &\rrra C_{B\lb{t_1^{1/q^\infty},t_2^{1/q^\infty}}\ang{t_1^n/t_2,t_2^n/t_1}[1/t_1t_2]}\rightrightarrows C_{B\lp{t_1^{1/q^\infty}}}.\end{split}
  \end{align*}  
  Let $\{C_1,C_2\}$ be an affine open cover of $C$. For all $i$ in $\{1,2\}$, write $A_i$ for the global sections of $C_i$, and note that $A_i$ is free over $\bF_q$. Hence restricting (\ref{eqn:algCBresolution}) to $(C_i)_{B\lp{t_1^{1/q^\infty}}}$, taking global sections, and taking $\varprojlim_n$ yields
  \begin{align*}
    A_i\otimes_{\bF_q}B\lp{t_1^{1/q^\infty},t_2^{1/q^\infty},t_3^{1/q^\infty}}\llla A_i\otimes_{\bF_q}B\lp{t_1^{1/q^\infty},t_2^{1/q^\infty}}\leftleftarrows A_i\otimes_{\bF_q}B\lp{t_1^{1/q^\infty}}.
  \end{align*}
Therefore \cite[Theorem 4.9]{Kim24} indicates that the restriction of $(\wh\sE',\wh\al)$ to $(C_i)_{(-),1}$ arises uniquely from a vector bundle $\wh\sE_i$ on $(C_i)_B$. Now $\{(C_1)_{(-),1},(C_2)_{(-),1}\}$ is an open cover of $C_{(-)}^{\an}$ and $\{(C_1)_B,(C_2)_B\}$ is an open cover of $C_B$, so $(\wh\sE',\wh\al)$ arises uniquely from a vector bundle $\wh\sE$ on $C_B$. Finally, restrict $\wh\sE$ to $U_B$.
\end{proof}

\section{$G$-bundles}\label{s:Gbundles}
In this section, we introduce $G$-bundles on the global Hartl--Pink curve. After defining the moduli stack $\Bun_{G,F}$ thereof, we define its \emph{localization map} to a product of local moduli stacks $\Bun_{G,F_z}$, which plays an important role in this paper. For example, we use the localization map to define the global semisimple locus $\Bun_{G,F}^{\semis}$ in $\Bun_{G,F}$. Using results from \S\ref{s:vectorbundles}, we prove part of Theorem A, part of Theorem B (though we postpone the official statement of Theorem B to \S\ref{s:geometric}), and Theorem C.

\subsection{}\label{ss:Tannakian}
Let $G$ be a parahoric group scheme over $C$ as in \cite[Definition 2.18]{Ric16}. We also write $G$ for its base change to $U$, $F$, $\cO_v$, or $F_v$ for any closed point $v$ of $C$.

We have the following standard Tannakian description of $G$-torsors. Let $Y$ be a sousperfectoid space over $C$, and let $\sG$ be an \'etale $G$-torsor on $Y$. For all $V$ in $\Rep_C{G}$, write $\sG(V)$ for the locally free $\sO_{Y,\et}$-module $\sG\times^G(V\otimes_{\sO_C}\sO_{Y,\et})$, which arises uniquely from a vector bundle $\sG(V)$ on $Y$ by \cite[Theorem 8.2.22 (d)]{KL15}.
\begin{prop*}
The above induces an equivalence of categories between
  \begin{enumerate}[a)]
  \item \'etale $G$-torsors on $Y$,
  \item exact tensor functors $\Rep_C{G}\ra\{\mbox{vector bundles on }Y\}$.
  \end{enumerate}
  Moreover, when $Y=\Spa(T,T^+)$ is affinoid, all \'etale $G$-torsors on $Y$ are representable by adic spaces, and the above are naturally equivalent to
  \begin{enumerate}[a)]
  \setcounter{enumi}{2}
  \item \'etale $G$-torsors on $\Spec{T}$,
  \item exact tensor functors $\Rep_C{G}\ra\{\mbox{finite projective }T\mbox{-modules}\}$.
  \end{enumerate}
\end{prop*}
\begin{proof}
By \cite[Lemma 3.1]{Bro13}, $\sO_G$ is a filtered colimit of objects in $\Rep_C{G}$. Hence the result follows from the proof of \cite[Proposition 6.10]{LH24}.
\end{proof}

\subsection{}\label{ss:BunG}
We now define the stack of $G$-bundles on the global and local Hartl--Pink curves. Let $A_v$ be one of $\{\cO_v,F_v\}$.
\begin{defn*}\hfill
  \begin{enumerate}[a)]
  \item Write $\Bun_{G,U}(S)$ for the groupoid of $G$-torsors $\sG$ on $U^{\an}_S$ equipped with an isomorphism $\phi:\sG\ra^\sim\Frob_S^*\sG$.
  \item Write $\Bun_{G,F}(S)$ for the groupoid $\varinjlim_U\Bun_{G,U}(S)$, where $U$ runs over dense open subschemes of $C$.
  \item Write $\Bun_{G,A_v}(S)$ for the groupoid of $G$-torsors $\sG$ on $\Spa{A_v}\times S$ equipped with an isomorphism $\phi:\sG\ra^\sim\Frob_S^*\sG$.
  \end{enumerate}
\end{defn*}
Using Proposition \ref{ss:Tannakian}, Theorem \ref{ss:vectorbundlesdescent} indicates that $\Bun_{G,U}$ is a v-stack. Proposition \ref{ss:Frobeniusfixedglobalsections} implies that the transition morphisms in b) are faithful, so $\Bun_{G,F}$ is also a v-stack. Finally, Proposition \ref{ss:localbundles} indicates that $\Bun_{G,A_v}$ is a v-stack, and the proof of \cite[Proposition III.1.3]{FS21} implies that all of these v-stacks are small.

\begin{exmp}\label{exmp:BunGC}
We claim that $\Bun_{G,C}$ is naturally isomorphic to the constant stack over $\bF_q$ associated with the groupoid
\begin{align*}\label{eq:automorphicspace}
\tag{$\heartsuit$}\coprod_{\al\in\ker^1(F,G)}G_\al(F)\bs G_\al(\bA)/G(\bO),
\end{align*}
where $\ker^1(F,G)$ denotes the kernel of the localization map
\begin{align*}
(\loc_v)_v:H^1(F,G)\ra\prod_vH^1(F_v,G), 
\end{align*}
and $G_\al$ over $F$ denotes the inner twist of $G$ associated with $\al$. To see this, write $\cB_G$ for the smooth algebraic stack over $\bF_q$ of $G$-bundles on $C$. Then Theorem \ref{ss:appliedGAGA} identifies $S$-points of $\Bun_{G,C}$ with $\Spec{R}$-points of the algebraic stack $(\cB_G)^{\Frob_{\cB_G}}$ of $\Frob_{\cB_G}$-fixed points of $\cB_G$. By \cite[Lemma 3.3 b)]{Var04}, the latter is naturally isomorphic to the constant stack over $\bF_q$ associated with the groupoid $\cB_G(\bF_q)$, and \cite[Remarque 12.2]{Laf16} identifies $\cB_G(\bF_q)$ with (\ref{eq:automorphicspace}). Since its automorphism groups are finite, the claim follows from the natural equivalence $(\Spec{R})_{\text{f\'et}}\ra^\sim S_{\text{f\'et}}$. 
\end{exmp}

\subsection{}\label{ss:localBunG}
In a manner analogous to Proposition \ref{ss:localbundles}, the trivial locus in the stack of $G$-bundles on the Fargues--Fontaine curve descends to our setting as follows. Write $\Bun^1_{G,F_v}$ for the substack of $\Bun_{G,F_v}(S)$ whose $S$-points consist of objects such that, for all geometric points $\ov{s}$ of $S$, its image in $\Bun_{G,F_v}(\ov{s})$ is trivial. Consider the morphism $*\ra\Bun_{G,F_v}$ corresponding to the trivial object for all $S$, which factors through a morphism $*\ra\Bun^1_{G,F_v}$. One can identify $\ul{G(F_v)}$ with $*\times_{\Bun_{G,F_v}^1}*$ as group v-sheaves, so descent yields a morphism $*/\ul{G(F_v)}\ra\Bun^1_{G,F_v}$.
\begin{prop*}
The substack $\Bun^1_{G,F_v}\subseteq\Bun_{G,F_v}$ is open, and $*/\ul{G(F_v)}\ra\Bun^1_{G,F_v}$ is an isomorphism. Moreover, $\Bun_{G,\cO_v}$ is naturally isomorphic to $*/\ul{G(\cO_v)}$. Finally, the base change of $\Bun_{G,F_v}$ to $\bF_v$ is naturally isomorphic to the stack of $G$-bundles on the Fargues--Fontaine curve for $F_v$ as in \cite[Definition III.1.3]{FS21}.
\end{prop*}
\begin{proof}
After replacing $(-)\lb{\pi}$ with $-\wh\otimes_{\bF_q}\bF_v\lb{\pi}$, where $\pi$ denotes a uniformizer of $\cO_v$, the first statement follows from the proof of \cite[Theorem III.2.4]{FS21}, and the second statement follows from the proof of \cite[Proposition 3.8]{LH23} and \cite[Theorem 3.12]{LH23}. The third statement follows from Proposition \ref{ss:localbundles}.
\end{proof}

\subsection{}\label{ss:BunGdefn}
The following \emph{localization maps} play an important role in this paper. For all closed points $u$ in $U$ and $z$ in $Z$, pullback yields morphisms
\begin{align*}
  \loc_u:\Bun_{G,U}\ra\Bun_{G,\cO_u}\mbox{ and }\loc_z:\Bun_{G,U}\ra\Bun_{G,F_z}.
\end{align*}
\begin{defn*}\hfill
  \begin{enumerate}[a)]
  \item Write
    \begin{align*}
      \loc:\Bun_{G,U}\ra\prod_u\Bun_{G,\cO_u}\times\prod_z\Bun_{G,F_z}
    \end{align*}
    for the morphism induced by the $\{\loc_u\}_u$ and $\{\loc_z\}_z$, where $u$ runs over closed points of $U$, and $z$ runs over $Z$.
  \item Write $\Bun_{G,\bA}$ for the prestack $\varinjlim_U\textstyle\prod_u\Bun_{G,\cO_u}\times\textstyle\prod_z\Bun_{G,F_z}$ on $\Perf_{\bF_q}$, where $U$ runs over dense open subschemes of $C$.
  \item Write $\loc:\Bun_{G,F}\ra\Bun_{G,\bA}$ for the morphism obtained by taking $\varinjlim_U$ of a).
  \end{enumerate}
\end{defn*}
Proposition \ref{ss:localBunG} implies that the transition morphisms in b) are faithful, so $\Bun_{G,\bA}$ is a small v-stack.

\subsection{}\label{ss:BunGpullback}
When studying $\Bun_{G,U}$, one has the following mechanism for changing $U$.
\begin{lem*}
  For all dense open subschemes $U'$ of $U$, the square
  \begin{align*}
    \xymatrix{\Bun_{G,U}\ar[r]^-{\loc}\ar[d] & \displaystyle\prod_u\Bun_{G,\cO_u}\times\prod_z\Bun_{G,F_z}\ar[d]\\
    \Bun_{G,U'}\ar[r]^-{\loc} & \displaystyle\prod_{u'}\Bun_{G,\cO_{u'}}\times\prod_{z'}\Bun_{G,F_{z'}}}
  \end{align*}
  is cartesian, where $u$ runs over closed points of $U'$, and $z'$ runs over $C\ssm U'$. Similarly, the square
  \begin{align*}
    \xymatrix{\Bun_{G,U}\ar[r]^-{\loc}\ar[d] & \displaystyle\prod_u\Bun_{G,\cO_u}\times\prod_z\Bun_{G,F_z}\ar[d]\\
    \Bun_{G,F}\ar[r]^-{\loc} & \Bun_{G,\bA}}
  \end{align*}
  is cartesian.
\end{lem*}
\begin{proof}
Note that $\{U'^{\an}_S\}\cup\{\Spa\cO_c\times S\}_c$ is an open cover of $U^{\an}_S$, where $c$ runs over $U\ssm U'$. Therefore the first statement follows from gluing, and the second statement follows from taking $\varinjlim_{U'}$ of the first statement.
\end{proof}

\subsection{}\label{ss:globalsemistable}
We \emph{use} the semistable locus in the stack of $G$-bundles on the local Hartl--Pink curve to define a global analogue as follows. Write $\Bun^{\semis}_{G,F_v}(S)$ for the full subcategory of objects in $\Bun_{G,F_v}(S)$ whose image in $\Bun_{G,F_v}(\Spa{\bF_v}\times S)$, under the identification from Proposition \ref{ss:localBunG}, lies in the semistable locus as in \cite[Section III.4.2]{FS21}. Then \cite[Theorem III.4.5]{FS21} implies that $\Bun_{G,F_v}^{\semis}$ is an open substack of $\Bun_{G,F_v}$, and the base change of $\Bun^{\semis}_{G,F_v}$ to $\bF_v$ equals the semistable locus as in \cite[Section III.4.2]{FS21}.
\begin{defn*}\hfill
  \begin{enumerate}[a)]
  \item Write $\Bun_{G,U}^{\semis}$ for the preimage of $\prod_u\Bun_{G,\cO_u}\times\prod_z\Bun_{G,F_z}^{\semis}$ in $\Bun_{G,U}$ under $\loc$.
  \item Write $\Bun_{G,F}^{\semis}$ for the prestack $\varinjlim_U\Bun_{G,U}^{\semis}$ on $\Perf_{\bF_q}$.
  \end{enumerate}
\end{defn*}
As in Definition \ref{ss:BunG}.b), $\Bun^{\semis}_{G,F}$ is a small v-stack. Since $\Bun_{G,U}^{\semis}\subseteq\Bun_{G,U}$ is an open substack, we see that $\Bun_{G,F}^{\semis}\subseteq\Bun_{G,F}$ is an open substack as well. Proposition \ref{ss:localBunG} shows that $\Bun_{G,\cO_v}\ra\Bun_{G,F_v}$ factors through $\Bun_{G,F_v}^{\semis}$, so for all dense open subschemes $U'$ of $U$, Lemma \ref{ss:BunGpullback} implies that the preimage of $\Bun^{\semis}_{G,U'}$ in $\Bun_{G,U}$ equals $\Bun^{\semis}_{G,U}$.

\subsection{}\label{ss:triviallocus}
The stack of $G$-bundles on the global Hartl--Pink curve also has a trivial locus in the following sense. Write $\Bun^1_{G,U}$ for the substack of $\Bun_{G,U}$ whose $S$-points consist of objects such that, for all geometric points $\ov{s}$ of $S$, its image in $\Bun_{G,U}(\ov{s})$ is trivial. Consider the morphism $*\ra\Bun_{G,U}$ corresponding to the trivial object for all $S$, which factors through a morphism $*\ra\Bun_{G,U}^1$. Proposition \ref{ss:Frobeniusfixedglobalsections} identifies $\ul{G(A)}$ with $*\times_{\Bun_{G,U}^1}*$ as group v-sheaves, so descent yields a morphism
\begin{align*}
*/\ul{G(A)}\ra\Bun_{G,U}^1.
\end{align*}
\begin{thm*}
The substack $\Bun^1_{G,U}\subseteq\Bun_{G,U}$ is open, and $*/\ul{G(A)}\ra\Bun_{G,U}^1$ is an isomorphism.
\end{thm*}
\begin{proof}
Let $(\sG,\phi)$ be an object in $\Bun_{G,U}(S)$. We claim that
  \begin{align*}
    \big\{\ov{s}\in\abs{S}\mid\mbox{the image of }(\sG,\phi)\mbox{ in }\Bun_{G,U}(\ov{s})\mbox{ is trivial}\big\}
  \end{align*}
  is an open subset of $\abs{S}$. To see this, by replacing $S$ with a pro-\'etale cover, we can assume that $S$ is strictly totally disconnected. Note that the above subset lies in
  \begin{align*}
\big\{\ov{s}\in\abs{S}\mid\mbox{the image of }(\sG,\phi)\mbox{ in }\Bun_{G,F_z}(\ov{s})\mbox{ is trivial for all }z\mbox{ in }Z\big\},
  \end{align*}
  which is an open subset of $\abs{S}$ by \cite[Theorem III.2.4]{FS21}. Therefore, by replacing $S$ with an open subspace, we can assume that the image of $(\sG,\phi)$ in $\Bun_{G,F_z}(\ov{s})$ is trivial for all $z$ in $Z$ and geometric points $\ov{s}$ of $S$. For all $V$ in $\Rep_C{G}$, this implies that the image of $(\sG(V),\phi(V))$ in $\Bun_{F_z}(S)$ has slope zero for all $z$ in $Z$, so Theorem \ref{ss:Alocalsystems} shows that $(\sG,\phi)$ corresponds to an exact tensor functor
  \begin{align*}
    \Rep_C{G}\ra\{\mbox{pro-\'etale }\ul{\sO_U}\mbox{-local systems on }S\}.
  \end{align*}
  
  Let $\{U_1,U_2\}$ be the open cover of $U$ from Lemma  \ref{ss:localcharts}. For all $i$ in $\{1,2\}$, write $A_i$ for the global sections of $U_i$, and write $A_{12}$ for the global sections of $U_1\cap U_2$. Because $\Pic{U_i}=1$, rank-$n$ $\ul{\sO_{U_i}}$-local systems are equivalent to $\ul{\GL_n(A_i)}$-torsors. Hence \cite[Lemma III.2.6]{FS21} implies that $\ul{\sO_{U_i}}$-local systems on $S$ are trivial, so $(\sG,\phi)$ corresponds to two exact tensor functors
  \begin{align*}
    \Big\{\Rep_C{G}\ra\{\mbox{finite free }\ul{A_i}(S)\mbox{-modules}\}\Big\}_i
  \end{align*}
along with an isomorphism between them after postcomposing with
\begin{align*}
\{\mbox{finite free }\ul{A_i}(S)\mbox{-modules}\}\ra\{\mbox{finite free }\ul{A_{12}}(S)\mbox{-modules}\}.
\end{align*}
Since $A_i$ is discrete, we have $\Spec\ul{A_i}(S)=\Spec\Cont(\pi_0(\abs{S}),A_i) = \pi_0(\abs{S})\times U_i$. Therefore Proposition \ref{ss:Tannakian} and gluing show that $(\sG,\phi)$ corresponds to an \'etale $G$-torsor $\sF$ on $\pi_0(\abs{S})\times U$.

For all $\ov{s}$ in $\abs{S}$, note that $\abs{\ov{s}}\times U\ra^\sim\textstyle\varprojlim_N\big(N\times U\big)$, where $N$ runs over compact neighborhoods of the image of $\ov{s}$ in $\pi_0(\abs{S})$, and the transition morphisms are affine. Therefore if $\sF|_{\abs{\ov{s}}\times U}$ is trivial, then $\sF|_{N\times U}$ is trivial for some compact neighborhood $N$ of the image of $\ov{s}$ in $\pi_0(\abs{S})$. Hence the image of $(\sG,\phi)$ in $\Bun_{G,U}(\wt{N})$ is trivial, where $\wt{N}$ denotes the preimage of $N$ in $S$. As $\ov{s}$ varies, this yields the claim.

  The claim yields the first statement. For the second statement, the above shows that $*\ra\Bun^1_{G,U}$ is a pro-\'etale cover, so descent yields the desired result.
\end{proof}

\subsection{}\label{ss:twistingb}
Let us recall some notation on the Kottwitz set. In this subsection, we work over $\Spd\ov\bF_q$. Let $K$ be one of $\{F,F_v\}$, write $B(K,G)$ for the Kottwitz set as in \cite[Subsection 10.2]{Kot14}, and write $B(K,G)_{\basic}$ for its subset of basic elements as in \cite[Definition 10.2]{Kot14}. For all $b$ in $B(K,G)$, write $G_b$ for the associated connected reductive group over $K$ as in \cite[Subsection 2.6]{Kot14}. Write $\loc_v:B(F,G)\ra B(F_v,G)$ for the localization map as in \cite[(10.7)]{Kot14}, and recall that $B(F,G)$ is naturally isomorphic to the set of $\Frob_{\ov\bF_q}$-conjugacy classes in $G(F\otimes_{\bF_q}\ov\bF_q)$ \cite[Theorem 12.2]{Iak22}. 

Similarly to the local setting \cite[Corollary III.4.3]{FS21}, in the global setting we can pass between the trivial locus and other basic loci by twisting as follows. Let $b$ be an element in $G(U_{\ov\bF_q})$ whose image in $B(F,G)$ is basic. Write $\Bun^b_{G,U}$ for the substack of $\Bun_{G,U}$ whose $S$-points consist of objects such that, for all geometric points $\ov{s}$ of $S$, its image in $\Bun_{G,U}(\ov{s})$ is isomorphic to $(G,(\Frob_S^{-1})^*\circ b)$. Twisting $G$ by $\ad{b}$ over $U_{\ov\bF_q}$ yields a parahoric group scheme $G_b$ over $U$, and because $b$ is basic, \cite[(5.2)]{HK21} shows that the generic fiber of $G_b$ agrees with the above notion.
\begin{prop*}
This induces a natural isomorphism $\Bun_{G,U}\ra^\sim\Bun_{G_b,U}$ such that the image of $\Bun_{G,U}^b$ equals $\Bun_{G_b,U}^1$.
\end{prop*}
\begin{proof}
For all objects $(\sG,\phi)$ in $\Bun_{G,U}(S)$, we have a natural isomorphism 
\begin{align*}
\ul\Isom(\sG,G)\ra^\sim\ul\Isom(\Frob^*_S\sG,\Frob_S^*G) = \Frob_S^*\ul\Isom(\sG,G),
\end{align*}
where the left arrow denotes precomposing with $\phi^{-1}:\Frob_S^*\sG\ra^\sim\sG$ and postcomposing with $(\Frob_S^{-1})^*\circ b:G\ra^\sim\Frob^*_SG$. Since $(G_b)_{U_{\ov\bF_q}}=G_{U_{\ov\bF_q}}$ and we work over $\Spd\ov\bF_q$, we see $\ul\Isom(-,G)$ induces the desired isomorphism $\Bun_{G,U}\ra^\sim\Bun_{G_b,U}$.
\end{proof}
\subsection{}\label{ss:BunGred}
For any small v-stack $X$, write $X^{\red}$ its \emph{reduction}, i.e. the perfect v-stack given by $\Spec{B}\mapsto\Hom(\Spd{B},X)$ for all perfect $\bF_p$-algebras $B$ \cite[Proposition 3.7]{Gle20}. We conclude by proving the following generalization of Theorem C.
\begin{thm*}
  The perfect v-stack $\Bun_{G,U}^{\red}$ over $\bF_q$ is the v-sheafification of
  \begin{align}\label{eq:BunGred}
    \tag{$\ast$}\Spec{B}\mapsto\left\{
    \begin{tabular}{c}
      $G$-torsors $\sG$ on $U_B$ equipped with\\
      an isomorphism $\phi:\sG\ra^\sim\Frob_B^*\sG$
    \end{tabular}
    \right\}.
  \end{align}
Moreover, when each connected component of $\Spec{B}$ is a valuation ring, no sheafification is needed.
\end{thm*}
\begin{proof}
  By using \cite[Lemma 6.2]{BS17} and taking perfections, we see that $\Spec{B}$ as in the second statement form a basis for the v-topology on $\{\mbox{affine perfect schemes}\}$. Hence it suffices to prove that the restriction of $\Bun^{\red}_{G,U}$ to such $\Spec{B}$ is naturally equivalent to (\ref{eq:BunGred}). Because $\Spa{B\lp{t^{1/q^\infty}}}\ra\Spd{B}$ is surjective and representable in perfectoid spaces, this follows from Proposition \ref{ss:Tannakian} and Theorem \ref{ss:BunUschemepoints}.
\end{proof}

\section{Shtukas and the fiber product conjecture}\label{s:shtukas}
Our goal in this section is to prove Theorem D. First, we recall the algebraic stack of global shtukas, and we prove that its associated v-stack takes a simple form. Next, we recall the v-stack of local shtukas. Finally, we define the function field analogue of the Hodge--Tate period map, and we prove Theorem D.

\subsection{}
We start with some more notation concerning $G$ over $F$. Fix a separable closure $\ov{F}$ of $F$, and write $W_F$ for the associated absolute Weil group of $F$. Then global class field theory yields an isomorphism $F^\times\bs\bA^\times\ra^\sim W_F^{\ab}$ that sends uniformizers to geometric Frobenii. Let $T$ be a maximal subtorus of $G$ over $F$, let $B\subseteq G_{\ov{F}}$ be a Borel subgroup containing $T_{\ov{F}}$, and write $X_*^+(T)$ for the associated set of dominant cocharacters of $T_{\ov{F}}$. Write $\pi_1(G)$ for the cocharacter group of $T_{\ov{F}}$ quotiented by the subgroup generated by the coroots.

\subsection{}\label{ss:globalshtuka}
Let us recall the stack of global shtukas. Let $I$ be a finite set, and let $\mu_\bullet=(\mu_i)_{i\in I}$ be in $X_*^+(T)^I$. For all $i$ in $I$, write $C_i$ for the normalization of $C$ in the field of definition of $\mu_i$. For all affine schemes $\Spec{B}$ over $\bF_q$ and morphisms
\begin{align*}
x:\Spec{B}\ra C_i
\end{align*}
over $\bF_q$, write $\Ga_x\subseteq C_B$ for the graph of the corresponding morphism $\Spec{B}\ra C$ over $\bF_q$. Let $N$ be a finite closed subscheme of $C$, and write $N_i$ for the preimage of $N$ in $C_i$.
\begin{defn*}
Write $\Sht_{G,\mu_\bullet,N}^I|_{\prod_{i\in I}C_i\ssm N_i}$ for the prestack on $\{\mbox{affine schemes over }\bF_q\}$ whose $\Spec{B}$-points parametrize data consisting of
  \begin{enumerate}[i)]
  \item for all $i$ in $I$, a morphism $x_i:\Spec{B}\ra C_i\ssm N_i$ over $\bF_q$,
  \item a $G$-torsor $\sG$ on $C_B$,
  \item an isomorphism of $G$-torsors
    \begin{align*}
      \phi:\sG|_{C_B\ssm\sum_{i\in I}\Ga_{x_i}}\ra^\sim\Frob_B^*\sG|_{C_B\ssm\sum_{i\in I}\Ga_{x_i}}
    \end{align*}
    whose relative position along $\{\Ga_{x_i}\}_{i\in I}$ is bounded by $\mu_\bullet$,
  \item an isomorphism of $G$-torsors $\psi:\sG|_{N_B}\ra^\sim G$ such that the square
    \begin{align*}
      \xymatrixcolsep{1.5cm}
      \xymatrix{\sG|_{N_B}\ar[r]^-{(\phi)_N}\ar[d]^-\psi & \Frob_B^*\sG|_{N_B}\ar[d]^-\psi\\
      G\ar[r]^-{(\Frob_B^*)^{-1}} &\Frob_B^*G}
    \end{align*}
    commutes.
  \end{enumerate}
  Write $\Sht_{G,\mu_\bullet,N}^{I,\circ}|_{\prod_{i\in I}C_i\ssm N_i}$ for the open substack of $\Sht_{G,\mu_\bullet,N}^I|_{\prod_{i\in I}C_i\ssm N_i}$ whose $\Spec{B}$-points consist of $(\{x_i\}_{i\in I},\sG,\phi,\psi)$ such that the isomorphism $\phi$ has relative position along $\{\Ga_{x_i}\}_{i\in I}$ equal to $\mu_\bullet$.
\end{defn*}
Recall that $\Sht^I_{G,\mu_\bullet,N}|_{\prod_{i\in I}C_i\ssm N_i}$ is a Deligne--Mumford stack that is separated and locally of finite type over $\prod_{i\in I}C_i\ssm N_i$ \cite[Proposition 5.6]{LH23}, and its open substack $\Sht^{I,\circ}_{G,\mu_\bullet,N}|_{\prod_{i\in I}C_i\ssm N_i}$ is smooth over $\prod_{i\in I}C_i\ssm N_i$.

For all finite closed subschemes $N'\supseteq N$ of $C$, write $K^{N'}_N$ for the kernel of $G(\sO_{N'})\ra G(\sO_N)$. Recall that $\Sht^I_{G,\mu_\bullet,N'}|_{\prod_{i\in I}C_i\ssm N'_i}\ra\Sht^I_{G,\mu_\bullet,N}|_{\prod_{i\in I}C_i\ssm N_i'}$ is finite Galois with group $K^{N'}_N$ \cite[Proposition 5.5]{LH23}.

\subsection{}\label{ss:globalshtukatruncation}
For convenience, we package the Harder--Narasimhan stratification for global shtukas into a single parameter $t$ as follows. By \cite[Proposition 2.2(b)]{AH13}, there exists an $\SL_h$-bundle $\sV$ on $C$ along with a closed embedding $\io:G^{\ad}\ra\ul\Aut(\sV)$ of group schemes over $C$ such that $\ul\Aut(\sV)/G^{\ad}$ satisfies \cite[(2.1)]{AH13}. Let $t$ be a positive integer.
\begin{defn*}
  Write $\Sht^{I,\leq t}_{G,\mu_\bullet,N}|_{\prod_{i\in I}C_i\ssm N_i}$ for the open substack of
  \begin{align*}
    \Sht^I_{G,\mu_\bullet,N}|_{\prod_{i\in I}C_i\ssm N_i}
  \end{align*}
  whose $\Spec{B}$-points consist of $(\{x_i\}_{i\in I},\sG,\phi,\psi)$ such that the $\SL_h$-torsor $\io_*(\sG^{\ad})$ on $C_B$ has Harder--Narasimhan polygon bounded by $t2\rho^\vee$, where $2\rho^\vee$ denotes the sum of positive coroots in $\SL_h$.
\end{defn*}
Recall that, when $\deg{N}$ is large enough, $\Sht^{I,\leq t}_{G,\mu_\bullet,N}|_{\prod_{i\in I}C_i\ssm N_i}$ is a disjoint union of quasiprojective schemes over $\prod_{i\in I}C_i\ssm N_i$ \cite[Proposition 5.6]{LH23}.

\subsection{}\label{defn:shtukadiamond}
We now convert the moduli of global shtukas into a v-stack.
\begin{defn*}
  Write $(\Sht^I_{G,\mu_\bullet,N}|_{\prod_{i\in I}C_i\ssm N_i})^\Diamond$ and $(\Sht^{I,\leq t}_{G,\mu_\bullet,N}|_{\prod_{i\in I}C_i\ssm N_i})^\Diamond$ for the prestack on $\Perf_{\bF_q}$ whose $S$-points are given by
  \begin{align*}
    (\Sht^I_{G,\mu_\bullet,N}|_{\prod_{i\in I}C_i\ssm N_i})(R)\mbox{ and }(\Sht^{I,\leq t}_{G,\mu_\bullet,N}|_{\prod_{i\in I}C_i\ssm N_i})(R),
  \end{align*}
respectively.
\end{defn*}

\begin{prop}\label{ss:shtukadiamond}\hfill
  \begin{enumerate}[1)]
  \item When $\deg{N}$ is large enough, $(\Sht^{I,\leq t}_{G,\mu_\bullet,N}|_{\prod_{i\in I}C_i\ssm N_i})^\Diamond$ equals $(-)^\Diamond$ of the adic space $\Sht^{I,\leq t}_{G,\mu_\bullet,N}|_{\prod_{i\in I}C_i\ssm N_i}$ over $\bF_q$.
  \item For any increasing family $\{N_a\}_{a\geq1}$ of finite closed subschemes of $C$ such that $\deg{N_a}\to\infty$ as $a\to\infty$, the morphism
  \begin{align*}
    \varprojlim_{a\geq1}(\Sht^I_{G,\mu_\bullet,N_a}|_{\prod_{i\in I}C_i\ssm(N_a)_i})^\Diamond\ra\varprojlim_{a\geq1}\big(\prod_{i\in I}C_i\ssm(N_a)_i\big)^\Diamond
  \end{align*}
  is separated, representable in locally spatial diamonds, and locally of finite $\dimtrg$. 

\item The prestack $(\Sht^I_{G,\mu_\bullet,N}|_{\prod_{i\in I}C_i\ssm N_i})^\Diamond$ is a small v-stack,
  \begin{align*}
  (\Sht^{I,\leq t}_{G,\mu_\bullet,N}|_{\prod_{i\in I}C_i\ssm N_i})^\Diamond\subseteq(\Sht^I_{G,\mu_\bullet,N}|_{\prod_{i\in I}C_i\ssm N_i})^\Diamond
  \end{align*}
  is an open substack, and for all finite closed subschemes $N'\supseteq N$ of $C$,
  \begin{align*}
    (\Sht^I_{G,\mu_\bullet,N'}|_{\prod_{i\in I}C_i\ssm N'_i})^{\Diamond}\ra(\Sht^I_{G,\mu_\bullet,N}|_{\prod_{i\in I}C_i\ssm N_i'})^{\Diamond}
  \end{align*}
is finite Galois with group $K^{N'}_N$.
  \end{enumerate}
\end{prop}
\begin{proof}
Part 1) follows from Lemma \ref{ss:nosheafification} and \ref{ss:globalshtukatruncation}. For part 2), it follows from part 1) that, for all positive integers $t$, the morphism
  \begin{align*}
    \varprojlim_{a\geq1}(\Sht^{I,\leq t}_{G,\mu_\bullet,N_a}|_{\prod_{i\in I}C_i\ssm(N_a)_i})^\Diamond\ra\varprojlim_{a\geq1}\big(\prod_{i\in I}C_i\ssm(N_a)_i\big)^\Diamond
  \end{align*}
  is separated, representable in locally spatial diamonds, and locally of finite $\dimtrg$. Recall that $\Sht^I_{G,\mu_\bullet,N}|_{\prod_{i\in I}C_i\ssm N_i}$ is the increasing union $\bigcup_{t\geq1}\Sht^{I,\leq t}_{G,\mu_\bullet,N}|_{\prod_{i\in I}C_i\ssm N_i}$. Since the preimage of $\Sht^{I,\leq t}_{G,\mu_\bullet,N}|_{\prod_{i\in I}C_i\ssm N'_i}$ in $\Sht^I_{G,\mu_\bullet,N'}|_{\prod_{i\in I}C_i\ssm N'_i}$ equals
\begin{align*}
 \Sht^{I,\leq t}_{G,\mu_\bullet,N'}|_{\prod_{i\in I}C_i\ssm N'_i}
\end{align*}
and $(-)^\Diamond$ preserves open embeddings, taking $\bigcup_{t\geq1}$ yields the desired result.

For part 3), let $v$ be a closed point of $C\ssm N'$, and let $t$ be a positive integer. For all positive integers $m$, combining the natural equivalence $(\Spec{R})_{\text{f\'et}}\ra^\sim S_{\text{f\'et}}$ with \ref{ss:globalshtuka} implies that $S$-points of $(\Sht^{I,\leq t}_{G,\mu_\bullet,N}|_{\prod_{i\in I}C_i\ssm(N+v)_i})^\Diamond$ parametrize $K^{N+mv}_N$-torsors $S'$ on $S$ equipped with a $K^{N+mv}_N$-equivariant morphism
\begin{align*}
  S'\ra(\Sht^{I,\leq t}_{G,\mu_\bullet,N+mv}|_{\prod_{i\in I}C_i\ssm(N+v)_i})^\Diamond.  
\end{align*}
By letting $m$ be large, part 1) indicates that $(\Sht^{I,\leq t}_{G,\mu_\bullet,N}|_{\prod_{i\in I}C_i\ssm(N+v)_i})^\Diamond$ is a small v-stack, the morphism
\begin{align*}
(\Sht^{I,\leq t}_{G,\mu_\bullet,N'}|_{\prod_{i\in I}C_i\ssm(N'+v)_i})^\Diamond\ra(\Sht^{I,\leq t}_{G,\mu_\bullet,N}|_{\prod_{i\in I}C_i\ssm(N'+v)_i})^\Diamond
\end{align*}
is finite Galois with group $K^{N'}_N$, and varying $t$ yields open embeddings. Hence taking $\bigcup_{t\geq1}$ shows the desired statements after restricting to $\prod_{i\in I}C_i\ssm(N+v)_i$. Finally, because $\prod_{i\in I}C_i\ssm N_i$ is covered by open subspaces of the form $\prod_{i\in I}C_i\ssm(N+v)_i$, taking $\bigcup_v$ yields the desired result.
\end{proof}

\subsection{}\label{ss:shtukadiamondred}
Conversely, we can recover the moduli of global shtukas (up to perfection) from its associated v-stack:
\begin{cor*}\hfill
  \begin{enumerate}[1)]
  \item When $N$ is large enough, $(\Sht^{I,\leq t}_{G,\mu_\bullet,N}|_{\prod_{i\in I}C_i\ssm N_i})^{\Diamond,\red}$ is naturally isomorphic to the perfect scheme $(\Sht^{I,\leq t}_{G,\mu_\bullet,N}|_{\prod_{i\in I}C_i\ssm N_i})^{\perf}$ over $\bF_q$.
  \item The perfect v-stack $(\Sht^I_{G,\mu_\bullet,N}|_{\prod_{i\in I}C_i\ssm N_i})^{\Diamond,\red}$ is naturally isomorphic to
    \begin{align*}
      (\Sht^I_{G,\mu_\bullet,N}|_{\prod_{i\in I}C_i\ssm N_i})^{\perf}.
    \end{align*}
  \end{enumerate}
\end{cor*}
\begin{proof}
  For part 1), Proposition \ref{ss:shtukadiamond}.1) shows that $(\Sht^{I,\leq t}_{G,\mu_\bullet,N}|_{\prod_{i\in I}C_i\ssm N_i})^\Diamond$ equals $(-)^\Diamond$ of the adic space $(\Sht^{I,\leq t}_{G,\mu_\bullet,N}|_{\prod_{i\in I}C_i\ssm N_i})^{\perf}$ over $\bF_q$.\footnote{We warn that this is \emph{not} glued from $\Spa{B}$ as $\Spec{B}$ runs over affine open subspaces of the perfect scheme $(\Sht^{I,\leq t}_{G,\mu_\bullet,N}|_{\prod_{i\in I}C_i\ssm N_i})^{\perf}$. Instead, this is glued from $\Spa(B,\bZ^\sim)$.} Therefore the desired result follows from \cite[Theorem 2.32]{Gle20}.

For part 2), let $v$ be a closed point of $C$. For all perfect $\bF_p$-algebras $B$, the functor $\Spd$ induces a natural equivalence $(\Spec{B})_{\text{f\'et}}\ra^\sim(\Spd{B})_{\text{f\'et}}$ \cite[Corollary 5.4]{Kim24}, so Proposition \ref{ss:shtukadiamond}.3) implies that, for all positive integers $m$, the $\Spec{B}$-points of $(\Sht^{I,\leq t}_{G,\mu_\bullet,N}|_{\prod_{i\in I}C_i\ssm(N+v)_i})^{\Diamond,\red}$ parametrizes $K^{N+mv}_N$-torsors $\Spec{B'}$ on $\Spec{B}$ equipped with a $K^{N+mv}_N$-equivariant morphism
\begin{align*}
\Spec{B'}\ra(\Sht^{I,\leq t}_{G,\mu_\bullet,N+mv}|_{\prod_{i\in I}C_i\ssm(N+v)_i})^{\Diamond,\red}.
\end{align*}
By letting $m$ be large, part 1) and \ref{ss:globalshtuka} indicate that $(\Sht^{I,\leq t}_{G,\mu_\bullet,N}|_{\prod_{i\in I}C_i\ssm(N+v)_i})^{\Diamond,\red}$ is naturally isomorphic to $(\Sht^{I,\leq t}_{G,\mu_\bullet,N}|_{\prod_{i\in I}C_i\ssm(N+v)_i})^{\perf}$. Finally, taking $\bigcup_v\bigcup_{t\geq1}$ yields the desired result.
\end{proof}

\subsection{}\label{ss:localshtuka}
Next, let us recall the moduli of local shtukas. For all closed points $v$ of $C$, fix a separable closure $\ov{F}_v$ of $F_v$, and write $W_{F_v}$ for the associated absolute Weil group of $F_v$. Write $\bC_v$ for the completion of $\ov{F}_v$. Fix an embedding $\ov{F}\ra\ov{F}_v$ over $F\ra F_v$, which induces a homomorphism $W_{F_v}\ra W_F$.

Let $\{I_z\}_z$ be a partition of $I$, where $z$ runs over $Z$. For all $z$ in $Z$, write $\mu_z$ for the projection of $\mu_\bullet$ to $X_*^+(T)^{I_z}$, and for all $i$ in $I_z$, write $z_i$ for the closed point of $C_i$ over $z$ induced by $\ov{F}\ra\ov{F}_z$. For all affinoid perfectoid spaces $S=\Spa(R,R^+)$ over $\bF_z$ and morphisms $x:S\ra\Spa{\cO_{z_i}}$ over $\bF_z$, write $\Ga_x\subseteq\Spa\cO_z\times_{\bF_z}S$ for the graph of the corresponding morphism $S\ra\Spa\cO_z$ over $\bF_z$. Write $\prod_{i\in I_z}\Spa\cO_{z_i}$ for the product of the $\Spa\cO_{z_i}$ over $\bF_z$, and write $\prod_{i\in I_z}\Spa{F_{z_i}}$ for the product of the $\Spa{F_{z_i}}$ over $\bF_z$. Let $n_z$ be a positive integer.

\begin{defn*}\footnote{This definition does not agree with the notation in \cite{LH23}, for which we apologize. This definition is the one that \emph{should} be denoted by $\LocSht$, since it parametrizes local shtukas. The definition in \cite{LH23} additionally parametrizes a framing (up to isogeny), so it is akin to a Rapoport--Zink space.}\hfill
  \begin{enumerate}[a)]
  \item Write $\LocSht^{I_z}_{G,\mu_z}|_{\prod_{i\in I_z}\Spd\cO_{z_i}}$ for the prestack on $\Perf_{\bF_z}$ whose $S$-points parametrize data consisting of
  \begin{enumerate}[i)]
  \item for all $i$ in $I_z$, a morphism $x_i:S\ra\Spa\cO_{z_i}$ over $\bF_z$,
  \item a $G$-torsor $\sG$ on $\Spa\cO_z\times_{\bF_z}S$,
  \item an isomorphism of $G$-torsors
    \begin{align*}
      \phi:\sG|_{\Spa\cO_z\times_{\bF_z}S\ssm\sum_{i\in I_z}\Ga_{x_i}}\ra^\sim(\Frob_S^{\deg{z}})^*\sG|_{\Spa\cO_z\times_{\bF_z}S\ssm\sum_{i\in I_z}\Ga_{x_i}}
    \end{align*}
    whose relative position along $\{\Ga_{x_i}\}_{i\in I_z}$ is bounded by $\mu_z$,
  \end{enumerate}
\item Write $\LocSht^{I_z}_{G,\mu_z,n_z}|_{\prod_{i\in I_z}\Spd{F_{z_i}}}$ for the prestack on $\Perf_{\bF_z}$ whose $S$-points parametrize data consisting of an $S$-point $(\{x_i\}_{i\in I_z},\sG,\phi)$ of $\LocSht^{I_z}_{G,\mu_z}|_{\prod_{i\in I_z}\Spd{F_{z_i}}}$ equipped with
  \begin{enumerate}[i)]
    \setcounter{enumii}{3}
  \item an isomorphism of $G$-torsors $\psi:\sG|_{n_zz\times_{\bF_z}S}\ra^\sim G$ such that the square
    \begin{align*}
      \xymatrixcolsep{1.5cm}
      \xymatrix{\sG|_{n_zz\times_{\bF_z}S}\ar[r]^-{(\phi)_{n_zz}}\ar[d]^-\psi & (\Frob_S^{\deg{z}})^*\sG|_{n_zz\times_{\bF_z}S}\ar[d]^-\psi\\
      G\ar[r]^-{(\Frob_S^*)^{-\deg{z}}} & G}
    \end{align*}
    commutes.
  \end{enumerate}
\end{enumerate}
Write $\LocSht^{I_z,\circ}_{G,\mu_z,n_z}|_{\prod_{i\in I_z}\Spd{F_{z_i}}}$ for the open substack of $\LocSht^{I_z}_{G,\mu_z,n_z}|_{\prod_{i\in I_z}\Spd{F_{z_i}}}$ consisting of $(\{x_i\}_{i\in I_z},\sG,\phi,\psi)$ such that the isomorphism $\phi$ has relative position along $\{\Ga_{x_i}\}_{i\in I_z}$ equal to $\mu_\bullet$.
\end{defn*}
Using \cite[Proposition 19.5.3]{SW20} and the proof of \cite[Proposition III.1.3]{FS21}, we see that $\LocSht^{I_z}_{G,\mu_z}|_{\prod_{i\in I_z}\Spd\cO_{z_i}}$ and $\LocSht^{I_z}_{G,\mu_z,n_z}|_{\prod_{i\in I_z}\Spd{F_{z_i}}}$ are small v-stacks.

\subsection{}
In the local and global settings, shtukas induce bundles on the Hartl--Pink curve in the following way. Recall that choosing a uniformizer of $\cO_z$ identifies $\Spa{\cO_z}\times_{\bF_z}S$ with the open unit disk over $S$. For all objects $(\{x_i\}_{i\in I_z},\sG,\phi)$ in
\begin{align*}
(\LocSht^{I_z}_{G,\mu_z}|_{\prod_{i\in I_z}\Spd\cO_{z_i}})(S),
\end{align*}
the quasicompactness of $\abs{S}$ shows that there exists a closed disk $D\subseteq\Spa{\cO_z}\times_{\bF_z}S$ over $S$ centered at the origin such that $D$ contains $\Ga_{x_i}$ for all $i$ in $I_z$. Hence $\phi$ restricts to an isomorphism $\sG|_{\Spa{F_z}\times_{\bF_z}S\ssm D}\ra^\sim(\Frob_S^{\deg{z}})^*\sG|_{\Spa{F_z}\times_{\bF_z}S\ssm D}$. Via continuation by Frobenius and Proposition \ref{ss:localBunG}, this corresponds to an object in $\Bun_{G,F_z}(S)$, so altogether this construction yields a morphism
\begin{align*}
\LocSht^{I_z}_{G,\mu_z}|_{\prod_{i\in I_z}\Spd{\cO_{z_i}}}\ra\Bun_{G,F_z}.
\end{align*}

Write $\prod_{i\in I}\Spa\cO_{z_i}$ for $\prod_{z\in Z}\big(\prod_{i\in I_z}\Spa\cO_{z_i}\big)$, and write $\prod_{i\in I}^\circ\Spa\cO_{z_i}$ for its open subspace $\big(\prod_{i\in I}\Spa\cO_{z_i}\big)\times_{\prod_{i\in I}C_i}\big(\prod_{i\in I}C_i\ssm N_i\big)$. Form the fiber product
\begin{align*}
  \xymatrix{(\Sht^I_{G,\mu_\bullet,N}|_{\prod_{i\in I}^\circ\Spa\cO_{z_i}})^\Diamond\ar[r]\ar[d] & \displaystyle\prod_{i\in I}\!^\circ\Spa\cO_{z_i}\ar[d]\\
  (\Sht^I_{G,\mu_\bullet,N}|_{\prod_{i\in I}C_i\ssm N_i})^\Diamond\ar[r] & \big(\displaystyle\prod_{i\in I}C_i\ssm N_i\big)^\Diamond.}
\end{align*}
By arguing as above for all $z$ in $Z$, we also obtain a morphism
\begin{align*}
(\Sht^I_{G,\mu_\bullet,N}|_{\prod_{i\in I}^\circ\Spa\cO_{z_i}})^\Diamond\ra\Bun_{G,U}.
\end{align*}

\subsection{}
We now describe the analogue of the Hodge--Tate period map for global shtukas. Write $n_v$ for the multiplicity of $v$ in $N$, and write $K_v^{n_v}$ for the kernel of $G(\cO_v)\ra G(\sO_{n_vv})$. Write $\LocSht^I_{G,\mu_\bullet,N}|_{\prod^\circ_{i\in I}\Spd\cO_{z_i}}$ for the product
\begin{align*}
\prod_{\substack{z\in Z\\n_z=0}}\LocSht^{I_z}_{G,\mu_z}|_{\prod_{i\in I_z}\Spd\cO_{z_i}}\times\prod_{\substack{z\in Z\\n_z\geq1}}\LocSht^{I_z}_{G,\mu_z,n_z}|_{\prod_{i\in I_z}\Spd{F_{z_i}}},
\end{align*}
and write $K_{N\cap U}$ for the kernel of $G(\bO_U)\ra G(\sO_{N\cap U})$.

Let $(\{x_i\}_{i\in I},\sG,\phi,\psi)$ be an object in $(\Sht^I_{G,\mu_\bullet,N}|_{\prod_{i\in I}^\circ\Spa\cO_{z_i}})^\Diamond(S)$. For all $z$ in $Z$, arguing as in the proof of \cite[Lemma 5.12]{LH23} shows that
\begin{align*}
(\{x_i\}_{i\in I_z},\sG|_{\Spa\cO_z},\phi|_{\Spa\cO_z})
\end{align*}
corresponds to an object in $(\LocSht^{I_z}_{G,\mu_z}|_{\prod_{i\in I_z}\Spd\cO_{z_i}})(S)$. If $n_z$ is positive, then $(\psi)_{n_zz}$ corresponds to a lift of this object to $(\LocSht^{I_z}_{G,\mu_z,n_z}|_{\prod_{i\in I_z}\Spd{F_{z_i}}})(S)$, so altogether we obtain a morphism
\begin{align*}
(\Sht^I_{G,\mu_\bullet,N}|_{\prod_{i\in I}^\circ\Spa\cO_{z_i}})^\Diamond\ra\LocSht^I_{G,\mu_\bullet,N}|_{\prod^\circ_{i\in I}\Spd\cO_{z_i}}
\end{align*}
over $\prod^\circ_{i\in I}\Spd\cO_{z_i}$. For all closed points $u$ of $U$, Proposition \ref{ss:localBunG} indicates that
\begin{align*}
(\sG|_{\Spa\cO_u},\phi|_{\Spa\cO_u})
\end{align*}
corresponds to an object in $(*/\ul{G(\cO_u)})(S)$. If $u$ lies in $N$, then $(\psi)_{n_uu}$ corresponds to a lift of this object to $(*/\ul{K_u^{n_u}})(S)$, so altogether we also obtain a morphism
\begin{align*}
(\Sht^I_{G,\mu_\bullet,N}|_{\prod_{i\in I}^\circ\Spa\cO_{z_i}})^\Diamond\ra*/\ul{K_{N\cap U}}.
\end{align*}

\subsection{}\label{ss:fiberproductconj}
Finally, we arrive at Theorem D.
\begin{thm*}
  The square of small v-stacks
    \begin{align*}
      \xymatrix{(\Sht^I_{G,\mu_\bullet,N}|_{\prod_{i\in I}^\circ\Spa\cO_{z_i}})^\Diamond\ar[r]\ar[d] & \ast/\ul{K_{N\cap U}}\times\LocSht^I_{G,\mu_\bullet,N}|_{\prod^\circ_{i\in I}\Spd\cO_{z_i}}\ar[d]\\
      \Bun_{G,U}\ar[r]^-{\loc} & \ast/\ul{G(\bO_U)}\times\displaystyle\prod_{z\in Z}\Bun_{G,F_z}}
  \end{align*}
is cartesian.
\end{thm*}
\begin{proof}
For all $z$ be in $Z$ and closed disks $D_z\subseteq\Spa\cO_z\times_{\bF_z}S$ over $S$ centered at the origin, write $D_z^{\cup}$ for the closed subspace
  \begin{align*}
    \coprod_{r\in\bZ/\!\deg{z}}\Frob_S^{-r}(D_z) \subseteq\coprod_{r\in\bZ/\!\deg{z}}\Spa\cO_z\times_{\bF_z,r}S\ra^\sim\Spa\cO_z\times S,
  \end{align*}
  where $\Spa\cO_z\times_{\bF_z,r}S$ denotes the product of $\Spa\cO_z$ and 
  \begin{align*}
    \xymatrixcolsep{1cm}
    \xymatrix{S\ar[r] & \Spa\bF_z\ar[r]^-{\Frob_{\bF_z}^r}& \Spa\bF_z}
  \end{align*}
  over $\bF_z$, and the isomorphism is from \cite[Lemma 5.11]{LH23}. Then
\begin{align*}
  \Big\{U^{\an}_S\ssm\bigcup_{z\in Z}D_z^{\cup}\Big\}\cup\{\Spa\cO_z\times S\}_{z\in Z}
\end{align*}
is an open cover of $C^{\an}_S$, so we can glue an object in $\Bun_{G,U}(S)$ to an object in $(\ast/\ul{K_{N\cap U}}\times\LocSht^I_{G,\mu_\bullet,N}|_{\prod^\circ_{i\in I}\Spd\cO_{z_i}})(S)$ to obtain
\begin{enumerate}[i)]
\item an $S$-point $\{x_i\}_{i\in I}$ of $\prod_{i\in I}^\circ\Spa\cO_{z_i}$,
\item a $G$-torsor $\sG^{\an}$ on $C^{\an}_S$,
  \item an isomorphism of $G$-torsors
    \begin{align*}
      \phi^{\an}:\sG|_{C^{\an}_S\ssm\sum_{i\in I}\Ga_{x_i}}\ra^\sim\Frob_S^*\sG^{\an}|_{C^{\an}_S\ssm\sum_{i\in I}\Ga_{x_i}}
    \end{align*}
    whose relative position along $\{\Ga_{x_i}\}_{i\in I}$ is bounded by $\mu_\bullet$,
  \item an isomorphism of $G$-torsors $\psi^{\an}:\sG^{\an}|_{N_S^{\an}}\ra^\sim G$ such that the square
    \begin{align*}
      \xymatrixcolsep{1.5cm}
      \xymatrix{\sG^{\an}|_{N^{\an}_S}\ar[r]^-{(\phi^{\an})_N}\ar[d]^-{\psi^{\an}} & \Frob_S^*\sG^{\an}|_{N^{\an}_S}\ar[d]^-{\psi^{\an}}\\
      G\ar[r]^-{(\Frob_S^{-1})^*} &\Frob_S^*G}
    \end{align*}
    commutes.
  \end{enumerate}
Finally, Proposition \ref{ss:Tannakian} and Theorem \ref{ss:appliedGAGA} show that the above corresponds to an object in $(\Sht^I_{G,\mu_\bullet,N}|_{\prod_{i\in I}^\circ\Spa\cO_{z_i}})^\Diamond(S)$.
\end{proof}

\section{Geometric consequences}\label{s:geometric}
In this section, we harvest some of the fruits of Theorem D. We start by deducing an algebraic analogue of Theorem D via taking reductions, and we explain how this implies the Langlands--Rapoport conjecture for global shtukas. Next, we use the charts provided by Theorem D and Beauville--Laszlo uniformization to prove Theorem A and Theorem B. Finally, we conclude by explaining the relation with Igusa varieties.

\subsection{}
For all $z$ in $Z$, write $\prod_{i\in I_z}z_i$ for the product of the $z_i$ over $\bF_z$. Let us recall the following algebraic analogues of $\Bun_{G,F_z}$ and $\LocSht^{I_z}_{G,\mu_z}|_{\prod_{i\in I_z}\Spd\cO_{z_i}}$.
\begin{defn*}\hfill
  \begin{enumerate}[a)]
  \item Write $\sIsoc_{G,F_z}$ for the prestack on $\{\mbox{affine perfect schemes over }\bF_q\}$ whose $\Spec{B}$-points parametrize data consisting of $G$-torsors $\sG$ on $\Spec F_z\wh\otimes_{\bF_q}B$ equipped with an isomorphism $\phi:\sG\ra^\sim\Frob_B^*\sG$.
  \item Write $\sLocSht^{I_z}_{G,\mu_z}|_{\prod_{i\in I_z}z_i}$ for the prestack on $\{\mbox{affine perfect schemes over }\bF_z\}$ whose $\Spec{B}$-points parametrize data consisting of
    \begin{enumerate}[i)]
    \item for all $i$ in $I_z$, a morphism $x_i:\Spec{B}\ra z_i$ over $\bF_z$,
    \item a $G$-torsor $\sG$ on $\Spec{\cO_z\wh\otimes_{\bF_z}B}$,
    \item an isomorphism of $G$-torsors $\phi:\sG|_{\Spec{F_z\wh\otimes_{\bF_z}B}}\ra^\sim(\Frob_B^{\deg{z}})^*\sG|_{\Spec{F_z\wh\otimes_{\bF_z}B}}$ whose relative position along $\{\Ga_{x_i}\}_{i\in I_z}$ is bounded by $\mu_z$,
    \end{enumerate}
  \end{enumerate}
\end{defn*}
Note that \cite[Proposition 5.9]{Iva23} and \cite[Remark 4.2]{BS17} imply that $\sIsoc_{G,F_z}$ and $\sLocSht^{I_z}_{G,\mu_z}|_{\prod_{i\in I_z}z_i}$ are perfect v-stacks.

\subsection{}\label{ss:fiberproductred}
Write $\prod_{i\in I}z_i$ for $\prod_{z\in Z}\big(\prod_{i\in I_z}z_i\big)$. By taking reductions, we recover the following algebraic analogue of Theorem \ref{ss:fiberproductconj}:
\begin{thm*}
  Assume that $N$ and $Z$ are disjoint. Then there is a natural cartesian square of perfect v-stacks
  \begin{align*}
      \xymatrix{(\Sht^I_{G,\mu_\bullet,N}|_{\prod_{i\in I}z_i})^{\perf}\ar[r]\ar[d] & \ast/\ul{K_N}\times\displaystyle\prod_{z\in Z}\sLocSht^{I_z}_{G,\mu_z}|_{\prod_{i\in I_z}z_i}\ar[d]\\
      \Bun_{G,U}^{\red}\ar[r] & \ast/\ul{G(\bO_U)}\times\displaystyle\prod_{z\in Z}\sIsoc_{G,F_z}.}
  \end{align*}
\end{thm*}
\begin{proof}
Because $(-)^{\red}$ preserves limits, base changing the top arrow of Theorem \ref{ss:fiberproductconj} to $\prod_{i\in I}z_i$ and taking $(-)^{\red}$ yields a cartesian square of perfect v-stacks. The classifying stacks are preserved by \cite[Corollary 5.4]{Kim24}, the top left is identified by Corollary \ref{ss:shtukadiamondred}.2), the top right is identified by \cite[Theorem 7.14.(2)]{GIZ25}, and the bottom right is identified by \cite[Theorem 7.14.(1)]{GIZ25}.
\end{proof}
\begin{rems}\hfill
  \begin{enumerate}[1)]
  \item By evaluating Theorem \ref{ss:fiberproductred} on $\ov\bF_q$ and using the description of $\Bun_{G,U}^{\red}(\ov\bF_q)$ from Theorem \ref{ss:BunGred}, we immediately deduce the Langlands--Rapoport conjecture for moduli spaces of global shtukas with arbitary (in particular, colliding) legs.
  \item If one takes the v-sheafification of (\ref{eq:BunGred}) as the definition of $\Bun_{G,U}^{\red}$, one can directly prove Theorem \ref{ss:fiberproductred} using Beauville--Laszlo gluing. However, without using Theorem \ref{ss:BunGred}, it is unclear why no sheafification is needed when evaluating $\Bun^{\red}_{G,U}$ on $\Spec{B}$ for which each connected component is a valuation ring.
  \end{enumerate}
\end{rems}

\subsection{}\label{ss:BdRgrassmannian}
Next, we recall some facts about affine Grassmannians. \textbf{For the rest of this paper, we work over $\ov\bF_q$.} For all $z$ in $Z$, write $\Gr_G$ for the $B_{\dR}$-affine Grassmannian over $\Spd{\breve{F}_z}$ as in \cite[Definition 20.2.1]{SW20}.\footnote{While \cite{SW20} works over $\bQ_p$, the definition of the $B_{\dR}$-affine Grassmannian and its basic properties hold over any nonarchimedean local field. Indeed, this is implicitly used in \cite{FS21}.} When $I_z=\{i\}$ is a singleton, recall from \cite[p.~184]{SW20} the closed affine Schubert variety $\Gr_{G,\mu_z}|_{\Spd{\breve{F}_{z_i}}}$ in $\Gr_G|_{\Spd{\breve{F}_{z_i}}}$ and the open affine Schubert variety $\Gr_{G,\mu_z}^\circ|_{\Spd{\breve{F}_{z_i}}}$ in $\Gr_{G,\mu_z}|_{\Spd{\breve{F}_{z_i}}}$. The proof of \cite[Proposition 23.3.1]{SW20} shows that $\Gr_{G,\mu_z}|_{\Spd{\breve{F}_{z_i}}}$ is naturally isomorphic to
\begin{align*}
\varprojlim_{m\geq1}\LocSht^{I_z}_{G,\mu_z,m}|_{\Spd{\breve{F}_{z_i}}}.
\end{align*}
The proof of \cite[Proposition VI.2.3]{FS21} implies that
\begin{align*}
\Gr_G|_{\Spa\bC_z}=\bigcup_{\mu_z\in X_*^+(T)}\Gr_{G,\mu_z}|_{\Spa\bC_z},
\end{align*}
\cite[Proposition III.3.1]{FS21} indicates that the natural morphism $\Gr_G|_{\Spa{\breve{F}_z}}\ra\Bun_{G,F_z}$ is surjective in the pro-\'etale topology, and the proof of \cite[Theorem IV.1.19]{FS21} shows 
\begin{align*}
\coprod_{\mu_z\in X_*^+(T)}\big(\Gr_{G,\mu_z}^\circ|_{\Spd\breve{F}_{z_i}}\big)\big/\ul{G(F_z)}\ra\Bun_{G,F_z}\times\Spd\breve{F}_z
\end{align*}
is a v-cover that is separated, representable in locally spatial diamonds, and cohomologically smooth.

\subsection{}\label{ss:localizationproperties}
Write $\loc_Z:\Bun_{G,U}\ra\textstyle\prod_{z\in Z}\Bun_{G,F_z}$ for the morphism $(\loc_z)_{z\in Z}$, which plays an important role in \S\ref{s:sheaf} and \S\ref{s:langlands}. It enjoys the following finitude properties.
\begin{prop*}
When $Z$ is nonempty, $\loc_Z:\Bun_{G,U}\ra\textstyle\prod_{z\in Z}\Bun_{G,F_z}$ is compactifiable, representable in locally spatial diamonds, and locally of finite $\dimtrg$. Consequently, $\Bun_{G,U}$ is an Artin v-stack.
\end{prop*}
When $Z$ is empty, Example \ref{exmp:BunGC} shows that $\loc_\varnothing:\Bun_{G,C}\ra*$ is \emph{not} representable in locally spatial diamonds but that $\Bun_{G,C}$ remains an Artin v-stack.
\begin{proof}
  Let $m$ be a positive integer. In Theorem \ref{ss:fiberproductconj}, take $I_z=\{i\}$ to be a singleton for all $z$ in $Z$, and take $N$ to be $mZ$. Then taking $\varprojlim_{m\geq1}$ yields a cartesian square
  \begin{align*}
      \xymatrix{\displaystyle\varprojlim_{m\geq1}(\Sht^I_{G,\mu_\bullet,mZ}|_{\prod_{i\in I}\Spa{\breve{F}_{z_i}}})^\Diamond\ar[r]\ar[d] & \displaystyle\varprojlim_{m\geq1}\LocSht^I_{G,\mu_\bullet,mZ}|_{\prod_{i\in I}\Spd{\breve{F}_{z_i}}}\ar[d]\\
      \Bun_{G,U}\ar[r]^-{\loc_Z} & \displaystyle\prod_{z\in Z}\Bun_{G,F_z}.}
  \end{align*}
  Proposition \ref{ss:shtukadiamond}.2) indicates that $\varprojlim_{m\geq1}(\Sht^I_{G,\mu_\bullet,mZ}|_{\prod_{i\in I}\Spa{\breve{F}_{z_i}}})^\Diamond$ is a locally spatial diamond that is separated and locally of finite $\dimtrg$ over $\prod_{i\in I}\Spd{\breve{F}_{z_i}}$. By \ref{ss:BdRgrassmannian} and \cite[Proposition 20.2.3]{SW20}, we see that $\varprojlim_{m\geq1}\LocSht^I_{G,\mu_\bullet,mZ}|_{\prod_{i\in I}\Spd{\breve{F}_{z_i}}}$ is also a locally spatial diamond that is separated and locally of finite $\dimtrg$ over $\prod_{i\in I}\Spd{\breve{F}_{z_i}}$, so the top arrow is separated, representable in locally spatial diamonds by \cite[Proposition 13.4 (ii)]{Sch17}, and locally of finite $\dimtrg$.

  Next, \ref{ss:BdRgrassmannian} implies that the morphism
  \begin{align*}
\coprod_{\mu_\bullet\in X_*^+(T)^I}\varprojlim_{m\geq1}\LocSht^I_{G,\mu_\bullet,mZ}|_{\prod_{z\in Z}\Spa\bC_z}\ra\prod_{z\in Z}\Bun_{G,F_z}
  \end{align*}
  is surjective in the pro-\'etale topology. Hence the bottom arrow is separated by \cite[Proposition 10.11 (ii)]{Sch17}, representable in locally spatial diamonds by \cite[Proposition 13.4 (iv)]{Sch17} and locally of finite $\dimtrg$, as desired.

Finally, \cite[Proposition IV.1.8 (i)]{FS21} and \cite[Theorem IV.1.19]{FS21} show $\prod_{z\in Z}\Bun_{G,F_z}$ is an Artin v-stack, so the same holds for $\Bun_{G,U}$ \cite[Proposition IV.1.8 (iii)]{FS21}.
\end{proof}

\subsection{}
At this point, we can finish the proof of Theorem A. For any finite set $Q$ of closed points of $C$, write $\cO_Q$ for $\prod_{v\in Q}\cO_v$, and write $F_Q$ for $\prod_{v\in Q}F_v$. Write $2\rho$ in $X^*(T)$ for the sum of all positive roots, and write $d_{\mu_\bullet}$ for $\sum_{i\in I}\ang{2\rho,\mu_i}$.
\begin{prop*}
The Artin v-stack $\Bun_{G,U}$ is cohomologically smooth over $\Spd\ov\bF_q$, and its dualizing complex over $\Spd\ov\bF_q$ with $\bF_\ell$-coefficients is isomorphic to $\bF_\ell$.
\end{prop*}
\begin{proof}
  When $Z$ is empty, this follows immediately from Example \ref{exmp:BunGC}. When $Z$ is nonempty, take $I_z=\{i\}$ to be a singleton for all $z$ in $Z$. Then \ref{ss:BdRgrassmannian} and the cartesian square from the proof of Proposition \ref{ss:localizationproperties} induce a cartesian square
  \begin{align*}
    \xymatrix{\Big[\displaystyle\varprojlim_{m\geq1}(\Sht^{I,\circ}_{G,\mu_\bullet,mZ}|_{\prod_{i\in I}\Spa{\breve{F}_{z_i}}})^\Diamond\Big]\Big/\ul{G(F_Z)}\ar[r]\ar[d]^-g &\Big[\displaystyle\prod_{z\in Z}\Gr^\circ_{G,\mu_z}|_{\Spd\breve{F}_{z_i}}\Big]\Big/\ul{G(F_Z)}\ar[d]^-f\\
  \Bun_{G,U}\times\displaystyle\prod_{z\in Z}\Spd\breve{F}_{z_i}\ar[r]^-{\loc_Z} & \displaystyle\prod_{z\in Z}\big(\Bun_{G,F_z}\times\Spd\breve{F}_{z_i}\big).}
\end{align*}
Now \ref{ss:BdRgrassmannian} indicates that $f$ is separated, representable in locally spatial diamonds, and cohomologically smooth, and the proof of \cite[Theorem IV.1.19]{FS21} indicates that $f^!\bF_\ell$ is isomorphic to $\bF_\ell(d_{\mu_\bullet})[2d_{\mu_\bullet}]$. Hence the same holds for $g$ and $g^!\bF_\ell$.

Because $\Sht^{I,\circ}_{G,\mu_\bullet,mZ}|_{\prod_{i\in I}C_i\ssm Z_i}$ is smooth over $\Spec\bF_q$, Proposition \ref{ss:shtukadiamond} implies
\begin{align}\label{eq:openglobalshtuka}
\Big[\displaystyle\varprojlim_{m\geq1}(\Sht^{I,\circ}_{G,\mu_\bullet,mZ}|_{\prod_{i\in I}\Spa{\breve{F}_{z_i}}})^\Diamond\Big]\Big/\ul{G(F_Z)}\tag{$\circ$}
\end{align}
admits a separated, representable in locally spatial diamonds, and cohomologically smooth v-cover from a locally spatial diamond $V$ such that $V$ is cohomologically smooth over $\Spd\ov\bF_q$. Moreover, the proof of \cite[Lemma 8.3.4]{DvHKZ24} shows that the dualizing complex of (\ref{eq:openglobalshtuka}) over $\textstyle\prod_{i\in I}\Spd\breve{F}_{z_i}$ with $\bF_\ell$-coefficients is isomorphic to $\bF_\ell(d_{\mu_\bullet})[2d_{\mu_\bullet}]$.

Next, \ref{ss:BdRgrassmannian} implies that the morphism
\begin{align*}
\coprod_{\mu_\bullet\in X_*^+(T)^I}\Big[\displaystyle\varprojlim_{m\geq1}(\Sht^{I,\circ}_{G,\mu_\bullet,mZ}|_{\prod_{i\in I}\Spa{\breve{F}_{z_i}}})^\Diamond\Big]\Big/\ul{G(F_Z)}\ra\Bun_{G,U}\times\prod_{z\in Z}\Spd\breve{F}_{z_i}
\end{align*}
is a v-cover. Therefore the above shows that $\Bun_{G,U}\times\prod_{z\in Z}\Spd\breve{F}_{z_i}$ is cohomologically smooth over $\prod_{z\in Z}\Spd\breve{F}_{z_i}$ and that its dualizing complex over $\prod_{z\in Z}\Spd\breve{F}_{z_i}$ with $\bF_\ell$-coefficients is isomorphic to $\bF_\ell$. By \cite[Proposition IV.1.8 (ii)]{FS21}, this implies that $\Bun_{G,U}$ is cohomologically smooth over $\Spd\ov\bF_q$, and by \cite[Theorem 19.5 (ii)]{Sch17}, this implies that its dualizing complex over $\Spd\ov\bF_q$ with $\bF_\ell$-coefficients is isomorphic to $\bF_\ell$.
\end{proof}

\subsection{}\label{ss:igusavarieties}
We will use the following analogue of (perfect) Igusa varieties for shtukas. For all $z$ in $Z$, let $b_z$ be an element in $B(F_z,G)$, and write $b_\bullet$ for $(b_z)_{z\in Z}$.
\begin{defn*}
  Write $\Ig_G^{b_\bullet}$ for the prestack on $\{\mbox{affine schemes over }\ov\bF_q\}$ whose $\Spec{B}$-points parametrize data consisting of
  \begin{enumerate}[i)]
  \item a $G$-torsor $\sG$ on $U_B$,
  \item an isomorphism of $G$-torsors $\phi:\sG\ra^\sim\Frob^*_B\sG$,
  \item for all $z$ in $Z$, an isomorphism of $G$-torsors $\psi_z:\sG|_{\Spec F_z\wh\otimes_{\bF_z}B}\ra^\sim G$ such that
    \begin{align*}
      \xymatrixcolsep{1.5cm}
      \xymatrix{\sG|_{\Spec F_z\wh\otimes_{\bF_z}B}\ar[r]^-{(\phi^{\deg{z}})_{F_z}}\ar[d]^-{\psi_z} & (\Frob_B^{\deg{z}})^*\sG|_{\Spec F_z\wh\otimes_{\bF_z}B}\ar[d]^-{\psi_z}\\
      G\ar[r]^-{(\Frob_B^*)^{-\deg{z}}\circ b_z} & G}
    \end{align*}
    commutes.
  \end{enumerate}
For any $\mu_\bullet=(\mu_z)_{z\in Z}$ in $X_*^+(T)^Z$, write $\Ig^{b_\bullet}_{G,\mu_\bullet}$ for the closed subprestack of $\Ig^{b_\bullet}_G$ whose $\Spec{B}$-points consist of $(\sG,\phi,\{\psi_z\}_{z\in Z})$ such that the relative position of $\phi$ along $\{z_B\}_{z\in Z}$ is bounded by $\mu_\bullet$.
\end{defn*}
Write $d$ for the least common multiple of $\{\deg{z}\}_{z\in Z}$. Then Definition \ref{ss:igusavarieties}.iii) implies that the morphism $\Frob_{\Ig_{G,\mu_\bullet}^{b_\bullet}}^d:\Ig^{b_\bullet}_{G,\mu_\bullet}\ra(\Frob_{\ov\bF_q}^d)^*\Ig^{b_\bullet}_{G,\mu_\bullet}$ is naturally equivalent to the isomorphism sending $(\sG,\phi,\{\psi_z\}_{z\in Z}\})\mapsto(\sG,\phi,\{\psi_z\circ(\phi^{(d/\!\deg{z})})_{F_z}\}_{z\in Z})$. In particular, $\Ig_{G,\mu_\bullet}^{b_\bullet}$ and $\Ig_G^{b_\bullet}=\varinjlim_{\mu_\bullet\in X_*^+(T)^Z}\Ig_{G,\mu_\bullet}^{b_\bullet}$ are perfect.

\begin{prop}\label{ss:igusavarietyindscheme}
When $Z$ is nonempty, $\Ig_{G,\mu_\bullet}^{b_\bullet}$ is a filtered open union of schemes that are cofiltered limits of disjoint unions of quasiprojective schemes over $\ov\bF_q$. In particular, $\Ig_G^{b_\bullet}$ is an ind-scheme.
\end{prop}

When $Z$ is empty, Example \ref{exmp:BunGC} shows that $\Ig_G^{()}$ is the constant stack over $\ov\bF_q$ associated with the groupoid (\ref{eq:automorphicspace}). In particular, $\Ig_G^{()}$ is \emph{not} an ind-scheme.

\begin{proof}
By Beauville--Laszlo gluing, a $\Spec{B}$-point of $\Ig^{b_\bullet}_G$ is equivalent to the data:
  \begin{enumerate}[i')]
  \item a $G$-torsor $\sG$ on $C_B$,
  \item an isomorphism of $G$-torsors $\phi:\sG|_{U_B}\ra^\sim\Frob^*_B\sG|_{U_B}$,
  \item for all $z$ in $Z$, an isomorphism of $G$-torsors $\psi_z:\sG|_{\Spec\cO_z\wh\otimes_{\bF_z}B}\ra^\sim G$ such that
    \begin{align*}
      \xymatrixcolsep{1.5cm}
      \xymatrix{\sG|_{\Spec F_z\wh\otimes_{\bF_z}B}\ar[r]^-{(\phi^{\deg{z}})_{F_z}}\ar[d]^-{(\psi_z)_{F_z}} & (\Frob_B^{\deg{z}})^*\sG|_{\Spec F_z\wh\otimes_{\bF_z}B}\ar[d]^-{(\psi_z)_{F_z}}\\
      G\ar[r]^-{(\Frob_B^*)^{-\deg{z}}\circ b_z} & G}
    \end{align*}
    commutes.
  \end{enumerate}
Next, write $\cB_{G,\infty Z}$ for the prestack on $\{\mbox{affine schemes over }\ov\bF_q\}$ whose $\Spec{B}$-points parametrize data consisting of
  \begin{enumerate}
  \item[i')] a $G$-torsor $\sG$ on $C_B$,
  \item[iii'')] for all $z$ in $Z$, an isomorphism of $G$-torsors $\psi_z:\sG|_{\Spec\cO_z\wh\otimes_{\bF_z}B}\ra^\sim G$.
  \end{enumerate}
  Recall that $\cB_{G,\infty Z}$ is a filtered open union of schemes that are cofiltered limits of disjoint unions of quasiprojective schemes \cite[p.~15]{AH13}. Given a $\Spec{B}$-point $(\sG,\{\psi_z\}_{z\in Z})$ of $\cB_{G,\infty Z}$, the additional data of 
  \begin{enumerate}
  \item[ii'')] an isomorphism of $G$-torsors $\phi:\sG|_{U_B}\ra^\sim\Frob^*_B\sG|_{U_B}$ whose relative position along $\{z_B\}_{z\in Z}$ is bounded by $\mu_\bullet$
  \end{enumerate}
corresponds to a projective scheme over $\Spec{B}$ \cite[Proposition 3.12]{AH13}, and the commutativity condition of iii') corresponds to a closed subscheme. Therefore the morphism $\Ig^{b_\bullet}_{G,\mu_\bullet}\ra\cB_{G,\infty Z}$ that sends $(\sG,\phi,\{\psi_z\}_{z\in Z})\mapsto(\sG,\{\psi_z\}_{z\in Z})$ is schematic and projective, which yields the desired result.
\end{proof}

\subsection{}
Similarly to Definition \ref{defn:shtukadiamond}, we now convert Igusa varieties into v-stacks.
\begin{defn*}
Write $(\Ig^{b_\bullet}_G)^\Diamond$ and $(\Ig^{b_\bullet}_{G,\mu_\bullet})^\Diamond$ for the prestack on $\Perf_{\ov\bF_q}$ whose $S$-points are given by $\Ig^{b_\bullet}_G(R)$ and $\Ig^{b_\bullet}_{G,\mu_\bullet}(R)$, respectively.
\end{defn*}
When $Z$ is empty, Example \ref{exmp:BunGC} shows that $(\Ig^{b_\bullet}_G)^\Diamond$ is a small v-stack. When $Z$ is nonempty, Lemma \ref{ss:nosheafification} and Proposition \ref{ss:igusavarietyindscheme} imply that
\begin{align*}
(\Ig^{b_\bullet}_{G,\mu_\bullet})^\Diamond\mbox{ and }(\Ig^{b_\bullet}_G)^\Diamond=\varinjlim_{\mu_\bullet\in X_*^+(T)^Z}(\Ig^{b_\bullet}_{G,\mu_\bullet})^\Diamond
\end{align*}
are small v-sheaves.

\subsection{}\label{ss:igusavarsfiber}
As suggested by Theorem \ref{ss:fiberproductconj}, the fibers of $\loc_Z$ are given by Igusa varieties in the following sense. We get a morphism $(\Ig^{b_\bullet}_G)^\Diamond\ra\Bun_{G,U}$ by sending
\begin{align*}
(\sG,\phi,\{\psi_z\}_z)\mapsto(\sG^{\an},\phi^{\an}).
\end{align*}
The element $b_z$ induces a morphism $b_z:\ast\ra\Bun_{G,F_z}$, so as $z$ varies we also get a morphism $b_\bullet:\ast\ra\textstyle\prod_{z\in Z}\Bun_{G,F_z}$.
\begin{prop*}
The square of small v-stacks
  \begin{align*}
    \xymatrix{(\Ig^{b_\bullet}_G)^\Diamond\ar[r]\ar[d] & \ast\ar[d]^-{b_\bullet} \\
    \Bun_{G,U}\ar[r]^-{\loc_Z} & \displaystyle\prod_{z\in Z}\Bun_{G,F_z}}
  \end{align*}
is cartesian.
\end{prop*}
\begin{proof}
In Theorem \ref{ss:fiberproductconj}, take $I_z=\{i\}$ to be a singleton for all $z$ in $Z$, take $N$ to be $\varnothing$, and base change the top arrow along the morphism $\Spd\ov\bF_q\ra\textstyle\prod_{i\in I}z_i\ra\textstyle\prod_{i\in I}\Spd\cO_{z_i}$ corresponding to the embeddings $\bF_z\ra\ov\bF_q$. This yields a cartesian square
  \begin{align*}
    \xymatrix{(\Sht^I_{G,\mu_\bullet}|_{\ov\bF_q})^\Diamond\ar[r]\ar[d] & \LocSht^I_{G,\mu_\bullet}|_{\Spd\ov\bF_q}\ar[d]\\
    \Bun_{G,U}\ar[r]^-{\loc_Z}& \displaystyle\prod_{z\in Z}\Bun_{G,F_z}.}
  \end{align*}
  Note that the morphism $b_\bullet:*\ra\textstyle\prod_{z\in Z}\Bun_{G,F_z}$ naturally lifts to a morphism
  \begin{align*}
  b_\bullet:*\ra\LocSht^I_{G,\mu_\bullet}|_{\Spd\ov\bF_q},
  \end{align*}
  so it suffices to prove that the square of small v-stacks
  \begin{align*}
    \xymatrix{(\Ig^{b_\bullet}_{G,\mu_\bullet})^\Diamond\ar[r]\ar[d] & \ast\ar[d]^-{b_\bullet}\\
    (\Sht^I_{G,\mu_\bullet}|_{\ov\bF_q})^\Diamond\ar[r] & \LocSht^I_{G,\mu_\bullet}|_{\Spd\ov\bF_q}}
  \end{align*}
is cartesian. An $S$-point in the fiber product parametrizes data consisting of
\begin{enumerate}[i)]
\item a $G$-torsor $\sG$ on $C_R$,
\item an isomorphism of $G$-torsors $\phi:\sG|_{U_R}\ra^\sim\Frob_R^*\sG|_{U_R}$ whose relative position along $\{z_R\}_{z\in Z}$ is bounded by $\mu_\bullet$,
\item for all $z$ in $Z$, an isomorphism of $G$-torsors $\psi_z^{\an}:\sG|_{\Spa\cO_z\times_{\bF_z}S}\ra^\sim G$ such that
  \begin{align*}
    \xymatrixcolsep{1.5cm}
    \xymatrix{\sG|_{\Spa F_z\times_{\bF_z}S}\ar[r]^-{(\phi^{\deg{z}})_{F_z}}\ar[d]^-{\psi_z^{\an}} & (\Frob_S^{\deg{z}})^*\sG|_{\Spa F_z\times_{\bF_z}S}\ar[d]^-{\psi_z^{\an}}\\
    G\ar[r]^-{(\Frob_S^*)^{-\deg{z}}} & G.}
  \end{align*}
\end{enumerate}
Note that everything is independent of $R^+$, so we can assume that $R^+$ equals $R^\circ$. Because products of points as in \cite[Definition 1.2]{Gle20} form a basis for the v-topology \cite[Example 1.1]{Gle20}, we can assume that $S$ is a product of points. Then \cite[Corollary 1.9]{GIZ25} implies that the above corresponds to an object in $(\Ig^{b_\bullet}_{G,\mu_\bullet})^\Diamond(S)$.
\end{proof}

\subsection{}\label{ss:basiclocus}
Let us recall some more notation concerning the Kottwitz set. Let $K$ be one of $\{F,F_v\}$, write $\bX_K$ for the $W_K$-module given by
\begin{align*}
  \bX_K\coloneqq
  \begin{cases}
  \bZ  & \mbox{if }K=F_v, \\
  \Div(\ov{F})_0\mbox{ as in \cite[p.~75]{HK21}} & \mbox{if }K=F,
  \end{cases}
\end{align*}
and write $\ka:B(K,G)\ra(\pi_1(G)\otimes_\bZ\bX_K)_{W_K}$ for the Kottwitz map as in \cite[Subsection 11.5]{Kot14}. Recall that $\ka$ restricts to a bijection $B(K,G)_{\basic}\ra^\sim(\pi_1(G)\otimes_\bZ\bX_K)_{W_K}$ \cite[Proposition 13.1, Proposition 15.1]{Kot14}.

Recall the global semistable locus from Definition \ref{ss:globalsemistable}, which was defined \emph{in terms of} its local analogue. We now prove that it admits the following \emph{intrinsic} description; this is Theorem B. For all $b$ in $B(F,G)_{\basic}$, write $\Bun^b_{G,F}$ for the substack of $\Bun_{G,F}$ whose $S$-points consist of objects such that, for all geometric points $\ov{s}$ of $S$, its image in $\Bun_{G,F}(\ov{s})$ is isomorphic to $(G,(\Frob^*_S)^{-1}\circ b)$.
\begin{thm*}
The substack $\Bun^b_{G,F}\subseteq\Bun_{G,F}$ is open and isomorphic to $*/\ul{G_b(F)}$. Moreover, the open substack $\Bun^{\semis}_{G,F}\subseteq\Bun_{G,F}$ equals $\coprod_{b\in B(F,G)_{\basic}}\Bun^b_{G,F}$.
\end{thm*}
\begin{proof}
  Theorem \ref{ss:triviallocus} implies that $\Bun^1_{G,F}\subseteq\Bun_{G,F}$ is open and naturally isomorphic to $*/\ul{G(F)}$. Proposition \ref{ss:twistingb} yields an isomorphism $\Bun_{G,F}\ra^\sim\Bun_{G_b,F}$ of v-stacks such that the image of $\Bun^b_{G,F}$ equals $\Bun^1_{G_b,F}$, so applying the above to $G_b$ shows that $\Bun^b_{G,F}\subseteq\Bun_{G,F}$ is open and isomorphic to $*/\ul{G_b(F)}$. Theorem \ref{ss:BunGred} identifies the isomorphism classes in $\Hom_{\ov\bF_q}(\Spd\ov\bF_q,\Bun_{G,F})$ with $B(F,G)$, which indicates that the $\Bun^b_{G,F}$ are disjoint for distinct $b$. Hence $\coprod_{b\in B(F,G)_{\basic}}\Bun^b_{G,F}$ is an open substack of $\Bun_{G,F}$, and it evidently lies in $\Bun^{\semis}_{G,F}$.

  For the reverse inclusion, let $\ov{s}=\Spa(K,K^+)$ be a geometric point of $\Bun^{\semis}_{G,U}$, where $U$ is some dense open subscheme of $C$. Then \cite[Theorem III.4.5]{FS21} implies that, for all $z$ in $Z$, the image of $\ov{s}$ in $\Bun_{G,F_z}$ lies in $\Bun^{b_z}_{G,F_z}$ for some $b_z$ in $B(F_z,G)_{\basic}$. Take $I_z=\{i\}$ to be a singleton, and take $\mu_z$ to be a lift of $\ka(b_z)$ to $X_*^+(T)$. Then \ref{ss:BdRgrassmannian} and the proof of \cite[Proposition A.9]{Sch18}\footnote{While \cite[Appendix A]{Sch18} works over $p$-adic local fields and assumes that the cocharacter is minuscule, this is not used in the proof of \cite[Proposition A.9]{Sch18}.} indicate that the image of $\ov{s}$ in $\Bun_{G,F_z}$ lifts to a $\Spa(K,K^+)$-point of $\varprojlim_{m\geq1}\LocSht^{I_z}_{G,\mu_z,mz}|_{\Spd{F_{z_i}}}$. As $z$ varies, the cartesian square from the proof of Proposition \ref{ss:localizationproperties} shows that $\ov{s}$ lifts to a $\Spa(K,K^+)$-point of $\varprojlim_{m\geq1}(\Sht^I_{G,\mu_\bullet,mZ}|_{\prod_{i\in I}\Spa F_{z_i}})^\Diamond$, which yields a $\Spec{K}$-point $(\sG,\phi)$ of $\Sht^I_{G,\mu_\bullet,\varnothing}|_{\prod_{i\in I}C_i}$.

Recall that the algebraic stack $\cB_G$ over $\bF_q$ of $G$-bundles on $C$ satisfies $\pi_0(\cB_G)=\pi_1(G)_{W_F}$. Since absolute $q$-Frobenius induces the identity on topological spaces, $\sG$ and $\Frob_K^*\sG$ have the same image in $\pi_0(\cB_G)$. Now $\phi$ exhibits $\Frob^*_K\sG$ as a modification of $\sG$ of relative position $\sum_{z\in Z}\mu_z$, so the image of $\Frob_K^*\sG$ in $\pi_0(\cB_G)$ equals the image of $\sG$ in $\pi_0(\cB_G)$ plus the image of $\sum_{z\in Z}\mu_z$ in $\pi_1(G)_{W_F}$. Therefore the image of $\sum_{z\in Z}\mu_z$ in $\pi_1(G)_{W_F}$ is trivial. By \cite[Proposition 15.6]{Kot14}, this shows that there exists a $b$ in $B(F,G)_{\basic}$ such that $\loc_z(b)=b_z$ for all $z$ in $Z$ and $\loc_u(b)=1$ for all closed points $u$ of $U$.

Let $U'$ be a dense open subscheme of $U$ such that $b$ has a representative in $G(U'_{\ov\bF_q})$, and fix such a representative. By setting $b_c=1$ for all $c$ in $U\ssm U'$, we may replace $U$ with $U'$. Then Proposition \ref{ss:twistingb} yields an isomorphism $\Bun_{G,U}\ra^\sim\Bun_{G_b,U}$ such that, for all $z$ in $Z$, the image of $\ov{s}$ in $\Bun_{G_b,F_z}$ lies in $\Bun^1_{G_b,F_z}$. After using fpqc descent to extend $G_b$ to a parahoric group scheme over $C$, Proposition \ref{ss:localBunG} implies that $\Bun_{G_b,\cO_z}\ra\Bun^1_{G_b,F_z}$ is surjective in the pro-\'etale topology for all $z$ in $Z$. Hence the image of $\ov{s}$ in $\Bun_{G_b,F_z}$ lifts to a $\Spa(K,K^+)$-point of $\Bun_{G_b,\cO_z}$, and as $z$ varies, Lemma \ref{ss:BunGpullback} shows that $\ov{s}$ lifts to a $\Spa(K,K^+)$-point $(\sG_b,\phi_b)$ of $\Bun_{G_b,C}$.

Proposition \ref{ss:Tannakian} and Theorem \ref{ss:appliedGAGA} indicate that $(\sG_b,\phi_b)$ is the analytification of a $G_b$-torsor $\sG^{\alg}_b$ on $C_K$ equipped with an isomorphism $\phi^{\alg}_b:\sG^{\alg}_b\ra^\sim\Frob_K^*\sG^{\alg}_b$. Applying \cite[Lemma 3.3 b)]{Var04} to $\cB_{G_b}$ shows that $(\sG^{\alg}_b,\phi^{\alg}_b)$ is isomorphic to
\begin{align*}
(\ov\sG^{\alg}_b,\id_{\ov\sG^{\alg}_b}\otimes_{\bF_q}(\Frob_K^{-1})^*)
\end{align*}
for some $G_b$-torsor $\ov\sG^{\alg}_b$ on $C$. Write $b'$ for the image of $(\ov\sG^{\alg}_b)_F$ under
\begin{align*}
H^1(F,G_b)\hra B(F,G_b)_{\basic}\ra^\sim B(F,G)_{\basic},
\end{align*}
where the left arrow denotes the injection from \cite[(10.5)]{Kot14}, and the right arrow denotes the bijection from \cite[Lemma 5.3]{HK21}. After shrinking $U$ such that $b'$ has a representative in $G(U_{\ov\bF_q})$, we see that $\ov{s}$ lies in $\Bun_{G,U}^{b'}$, as desired.
\end{proof}

\section{Sheaf-theoretic consequences}\label{s:sheaf}
Our goal in this section is to prove Theorem E, which heavily relies on material from Appendix \ref{s:berkovichmotives}. We start by briefly recalling the theory of motivic sheaves and explaining why we need it. Next, we explain a general mechanism for constructing (derived) Hecke actions, which we use in both \S\ref{s:sheaf} and \S\ref{s:langlands}. We then state a forthcoming result of Eteve--Gaitsgory--Genestier--Lafforgue, which is necessary to even formulate Theorem E.

At this point, we can prove Theorem E. As an application, we explain how Theorem E can be used to recover all of the representation-theoretic results of \cite{LH23}; for this, it suffices to use work of Xue instead of the forthcoming result of Eteve--Gaitsgory--Genestier--Lafforgue.

\subsection{}\label{ss:sheafexplain}
We begin by recalling the sheaf theories that we will use. Let $L/\bQ_\ell$ be a finite extension containing a fixed $\sqrt{q}$. Write $\cO_L$ for the ring of integers of $L$, and write $\bF_\la$ for the residue field of $\cO_L$. Let $\La$ be one of $\{L,\cO_L,\bF_\la\}$.

Recall from \ref{ss:ladicrealization} the $\bZ[\frac1q]$-linear $6$-functor formalism $D_{\mot}(-)$ on
\begin{align*}
\{\mbox{small v-stacks over }\ov\bF_q\},
\end{align*}
and recall from \ref{ss:classifyingstack} the $\La$-linear $6$-functor formalism
\begin{align*}
D(-,\La)\coloneqq D_{\mot}(-)\otimes_{D_{\mot}(*)}D(\La).
\end{align*}
We use (this base change of) the theory of motivic sheaves because it enjoys $!$-pushforwards and encodes ``non-completed'' sheaves even when $\La\neq\bF_\la$ (e.g. see Corollary \ref{ss:classifyingstack}). However, to facilitate comparisons with classical \'etale $\ell$-adic sheaf theories for algebraic varieties, we will also use the various $3$-functor formalisms from Definition \ref{ss:ladicrealizationvarieties}, Proposition \ref{prop:ladicrealizationvarieties}, and \ref{ss:indextension}.

If one is only interested in the $\La=\bF_\la$ case, one can instead take
\begin{align*}
D(-,\bF_\la)\coloneqq D_{\et}(-,\bF_\la)
\end{align*}
and just use the \'etale sheaf theory of \cite{Sch17}. In fact, since all the sheaves that we use are overconvergent as in \cite[Proposition IV.2.4]{FS21}, the theory of motivic sheaves specializes to this when $\La=\bF_\la$; see \ref{ss:ladicrealization}.

\subsection{}\label{ss:derivedHeckeformalism}
We will use the following mechanism for constructing actions of (derived) endomorphism algebras on both the automorphic and spectral sides. Let $(\cC,E)$ be a geometric setup as in \cite[Definition 2.1.1]{HM24}, and let $D(-)$ be a $\La$-linear $3$-functor formalism on $(\cC,E)$. Consider a diagram in $\cC$
\begin{align*}
  \xymatrix{Y_O\ar[r]^-{g_O}\ar[d]^-{j_Y} & O\times X\ar[r]^-{p_1}\ar[d]^-{j_X} & O\ar[d]^-j\\
  Y_F\ar[r]^-{g_F} & F \times X\ar[r]^-{p_1}\ar[d]^-{p_2} & F\\
  & X, &}
\end{align*}
where the squares are cartesian, and the morphisms are all $!$-able. Write $f$ for the morphism $Y_O\ra X$, and let $A$ be an object in $D(Y_F)$.
\begin{lem*}
The object $f_!j_Y^*A$ in $D(X)$ is naturally a module for the $\bE_1$-algebra $\End_{D(F)}(j_!\mathbf{1})$ over $\La$.
\end{lem*}
\begin{proof}
Since $f=p_2\circ g_F\circ j_Y$, the projection formula and proper base change yield
\begin{align*}
  f_!j_Y^*A &= p_{2,!}g_{F,!}j_{Y,!}j_Y^*A = p_{2,!}g_{F,!}(A\otimes j_{Y,!}\mathbf{1}) = p_{2,!}g_{F,!}(A\otimes j_{Y,!}g_O^*p_1^*\mathbf{1}) \\
  &= p_{2,!}g_{F,!}(A\otimes g_F^*p_1^*j_!\mathbf{1}) = p_{2,!}(p_1^*j_!\mathbf{1}\otimes g_{F,!}A).
\end{align*}
Because $j_!\mathbf{1}$ is naturally a module for the $\bE_1$-algebra $\End_{D(F)}(j_!\mathbf{1})$ over $\La$, the same holds for $f_!j_Y^*A$ by applying the functor $p_{2,!}(p_1^*-\otimes g_{F,!}A)$.
\end{proof}

\subsection{}\label{ss:derivedHeckeaction}
For places in $U$, we have the following (derived) Hecke algebra action on $\loc_{Z,!}\!\La$. Let $Q$ be a finite nonempty set of closed points of $U$, and write
\begin{align*}
\loc_Q:\Bun_{G,U}\ra*/\ul{G(\cO_Q)}
\end{align*}
for the morphism $(\loc_u)_{u\in Q}$. Then Lemma \ref{ss:BunGpullback} implies that the square
\begin{align*}
  \xymatrix{\Bun_{G,U}\ar[r]^-{(\loc_Q,\loc_Z)}\ar[d] & \ast/\ul{G(\cO_Q)}\times\displaystyle\prod_{z\in Z}\Bun_{G,F_z}\ar[d]\\
  \ast/\ul{G(F_Q)}\times_{\prod_{u\in Q}\Bun_{G,F_u}}\Bun_{G,U\ssm Q} \ar[r] & \ast/\ul{G(F_Q)}\times\displaystyle\prod_{z\in Z}\Bun_{G,F_z}}
\end{align*}
is cartesian. Proposition \ref{ss:localizationproperties} implies that the bottom arrow is $!$-able, so the same holds for the top arrow. Therefore Lemma \ref{ss:derivedHeckeformalism} and Corollary \ref{ss:classifyingstack} endow $\loc_{Z,!}\!\La$ in $D(\prod_{z\in Z}\Bun_{G,F_z},\La)$ with the structure of a module for the $\bE_1$-algebra
\begin{align*}
\End_{G(F_Q)}(\cInd_{G(\cO_Q)}^{G(F_Q)}\La) = \bigotimes_{u\in Q}\End_{G(F_u)}(\cInd^{G(F_u)}_{G(\cO_u)}\La)
\end{align*}
over $\La$, which we remind the reader can have cohomology in nonzero degrees when $\La\neq L$. Taking the colimit over $Q$ shows that $\loc_{Z,!}\!\La$ is naturally a module for the $\bE_1$-algebra $\bigotimes_u\End_{G(F_u)}(\cInd^{G(F_u)}_{G(\cO_u)}\La)$ over $\La$, where $u$ runs over closed points of $U$.

\subsection{}
Let us recall some notation on Langlands dual groups. \textbf{For the rest of this paper, assume that $G$ is reductive over $U$.} Write $W_U$ for the absolute Weil group of $U$ with respect to $\ov{F}$, write $\wt{F}/F$ for the finite Galois extension such that $\Gal(\wt{F}/F)$ equals the image of the $W_U$-action on $X_*^+(T)$, and write $\wt{C}$ for the normalization of $C$ in $\wt{F}$. Write $\wt{U}$ for the preimage of $U$ in $\wt{C}$, and write $\wt{Z}$ for the closed complement $\wt{C}\ssm\wt{U}$. Write $\wh{G}$ for the dual group of $G$ over $\La$, and write $\prescript{L}{}{G}$ for $\wh{G}\rtimes\ul{\Gal(\wt{F}/F)}$. Similarly, for all closed points $v$ of $C$, write $\wt{F}_v/F_v$ for the finite Galois extension such that $\Gal(\wt{F}_v/F_v)$ equals the image of the $W_{F_v}$-action on $X_*^+(T)$, and write $\prescript{L}{}{G}_v$ for $\wh{G}\rtimes\ul{\Gal(\wt{F}_v/F_v)}$.

Using $\prescript{L}{}G$, we re-index stacks of global shtukas as follows. Recall from \ref{ss:globalshtuka} the Deligne--Mumford stack $\Sht^I_{G,\mu_\bullet,N}|_{\prod_{i\in I}C_i\ssm N_i}$. Let $\dot{V}$ be an object in $\Rep_\La(\prescript{L}{}G)^I$, write $\wt{N}$ for the preimage of $N$ in $\wt{C}$, and note that the disjoint union
\begin{align*}
\coprod_{\substack{\mu_\bullet\in X_*^+(T)^I\\\text{arising\,in\,}\dot{V}|_{\wh{T}^I}}}\Sht^I_{G,\mu_\bullet,N}|_{(\wt{C}\ssm\wt{N})^I}
\end{align*}
naturally descends to a Deligne--Mumford stack $\Sht^I_{G,\dot{V},N}|_{(C\ssm N)^I}$ that is separated and locally of finite type over $(C\ssm N)^I$. Write $\pi_N:\Sht^I_{G,\dot{V},N}|_{(C\ssm N)^I}\ra(C\ssm N)^I$ for the structure morphism.

\subsection{}\label{ss:ladicrealizationglobalshtukas}
Briefly, write $\Sht$ for $\Sht^I_{G,\dot{V},N}|_{(C\ssm N)^I_{\ov\bF_q}}$. Using \ref{ss:globalshtuka}, Proposition \ref{ss:shtukadiamond}, and finite \'etale descent, we can extend the $\La$-linear $3$-functor formalisms from Proposition \ref{prop:ladicrealizationvarieties} and \ref{ss:indextension} to the mildly stacky setting of $\pi_N:\Sht\ra(C\ssm N)^I_{\ov\bF_q}$. More precisely, we have a natural diagram
  \begin{align*}
    \xymatrix{\Ind D_{\mot}^{(C\ssm N)^I_{\ov\bF_q}\bs\prod_{i\in I}\Spa\bC_{z_i}}\big((\Sht|_{\prod_{i\in I}\Spa\bC_{z_i}})^\Diamond\big)\ar[r]^-\Ups\ar[d]^-{r_{\ell,(\Sht|_{\prod_{i\in I}\Spa\bC_{z_i}})^\Diamond}} & D_{\mot}\big((\Sht|_{\prod_{i\in I}\Spa\bC_{z_i}})^\Diamond\big)\\
    \Ind D_{\et}^{(C\ssm N)^I_{\ov\bF_q}\bs\prod_{i\in I}\Spa\bC_{z_i}}\big((\Sht|_{\prod_{i\in I}\Spa\bC_{z_i}})^\Diamond,\La\big) & \ar[l]_-{\rho}D(\Sht,\bZ_\ell)}
  \end{align*}
  such that $!$-pushforward along $\pi_N$ yields a morphism of diagrams to
  \begin{align*}
    \xymatrix{\Ind D_{\mot}^{(C\ssm N)^I_{\ov\bF_q}\bs\prod_{i\in I}\Spa\bC_{z_i}}\big(\prod_{i\in I}\Spa\bC_{z_i}\big)\ar[r]^-\Ups\ar[d]^-{r_{\ell,\prod_{i\in I}\Spa\bC_{z_i}}} & D_{\mot}\big(\prod_{i\in I}\Spa\bC_{z_i}\big)\\
    \Ind D_{\et}^{(C\ssm N)^I_{\ov\bF_q}\bs\prod_{i\in I}\Spa\bC_{z_i}}\big(\prod_{i\in I}\Spa\bC_{z_i},\La\big) & \ar[l]_-{\rho}D((C\ssm N)^I,\bZ_\ell).}
  \end{align*}

\subsection{}\label{ss:EGGL}
We now state a forthcoming result of Eteve--Gaitsgory--Genestier--Lafforgue. Roughly, it asserts that the cohomology of stacks of global shtukas satisfies the conclusion of Drinfeld's lemma even on the level of derived categories. On the level of cohomology groups, this is already known by work of Xue; see Remark \ref{rem:Xue} below.

For any separated finite type scheme $X$ over $\ov\bF_q$, write $D_{\lis}(X,\La)\subseteq D_{\cons}(X,\La)$ for the full subcategory of lisse objects. Since $\Frob_{U\ssm N}$ is a universal homeomorphism, the functor
\begin{align*}
(\Frob_{U\ssm N})_{\ov\bF_q}^*:D_{\cons}((U\ssm N)_{\ov\bF_q},\La)\ra D_{\cons}((U\ssm N)_{\ov\bF_q},\La)
\end{align*}
restricts to an equivalence $(\Frob_{U\ssm N})_{\ov\bF_q}^*:D_{\lis}((U\ssm N)_{\ov\bF_q},\La)\ra^\sim D_{\lis}((U\ssm N)_{\ov\bF_q},\La)$.

Write $\cS_{\dot{V}}$ for the object in $D_{\cons}(\Sht^I_{G,\dot{V},N}|_{(U\ssm N)^I_{\ov\bF_q}},\Lambda)$ associated with $\dot{V}$ via geometric Satake. Unless otherwise specified, all Lurie tensor products are over $D(\La)$.
\begin{thm*}[Eteve--Gaitsgory--Genestier--Lafforgue]
  The object $\pi_{N,!}\cS_{\dot{V}}$ in
  \begin{align*}
D((U\ssm N)_{\ov\bF_q}^I,\La)
  \end{align*}
  lies in the image of the fully faithful external tensor product functor
  \begin{align*}
    \Big[\Ind D_{\lis}((U\ssm N)_{\ov\bF_q},\La)\Big]^{\otimes I}\hra D((U\ssm N)_{\ov\bF_q}^I,\La)
  \end{align*}
and naturally lifts to an object of $\Big[\big(\Ind D_{\lis}((U\ssm N)_{\ov\bF_q},\La)\big)^{(\Frob_{U\ssm N})_{\ov\bF_q}^*}\Big]^{\otimes I}$.
\end{thm*}

\begin{rem}\label{rem:Xue}
  Work of Xue \cite[Theorem 3.2.3]{Xue20b}\footnote{While \cite[Theorem 3.2.3]{Xue20b} assumes that $G$ is split, its proof does not use this. See \cite[\S4]{Xue20b}.} already implies the first statement in Theorem \ref{ss:EGGL}, as well as the second statement after taking cohomology groups \cite[Corollary 3.2.6]{Xue20b}. In particular, if one is only interested in Theorem \ref{ss:cohomologyformula} below after taking cohomology groups, one does not need to appeal to Theorem \ref{ss:EGGL}. For example, this is the case for Corollary \ref{ss:gl_fs} below.
\end{rem}

\subsection{}
For places in $U$, we have the following (derived) Hecke algebra action on the cohomology of stacks of global shtukas. Let $Q$ be a finite set of closed points of $U$. Write
\begin{align*}
\pi_{\infty(Q\cup Z)}:\varprojlim_{m\geq1}\Sht^I_{G,\dot{V},m(Q\cup Z)}|_{(U\ssm Q)_{\ov\bF_q}^I}\ra(U\ssm Q)_{\ov\bF_q}^I
\end{align*}
for the structure morphism, and recall that $\varprojlim_{m\geq1}\Sht^I_{G,\dot{V},m(Q\cup Z)}|_{(U\ssm Q)_{\ov\bF_q}^I}$ has a natural $\ul{G(F_{Q\cup Z})}$-action over $(U\ssm Q)_{\ov\bF_q}^I$ \cite[3.9]{LH21}. Because the square
\begin{align*}
  \xymatrixcolsep{1.5cm}
  \hspace{-.5cm}\xymatrix{\Big[\displaystyle\varprojlim_{m\geq1}\Sht^I_{G,\dot{V},mZ}|_{(U\ssm Q)_{\ov\bF_q}^I}\Big]\Big/\ul{G(F_Z)}\ar[r]^-{\pi_{\infty Z}}\ar[d] & \ast/\ul{G(\cO_Q)}\times\ast/\ul{G(F_Z)}\times(U\ssm Q)_{\ov\bF_q}^I\ar[d]\\
  \Big[\displaystyle\varprojlim_{m\geq1}\Sht^I_{G,\dot{V},m(Q\cup Z)}|_{(U\ssm Q)_{\ov\bF_q}^I}\Big]\Big/\ul{G(F_{Q\cup Z})}\ar[r]^-{\pi_{\infty(Q\cup Z)}} & \ast/\ul{G(F_Q)}\times\ast/\ul{G(F_Z)}\times(U\ssm Q)^I_{\ov\bF_q} }
\end{align*}
is cartesian, applying Lemma \ref{ss:derivedHeckeformalism} and taking the colimit over $Q$ endow
\begin{align*}
\pi_{\infty Z,!}\cS_{\dot{V}}=\varinjlim_{m\geq1}\pi_{mZ,!}\cS_{\dot{V}}\in D(*/\ul{G(F_Z)},\La)\otimes\Big[\big(\Ind D_{\lis}(U_{\ov\bF_q},\La)\big)^{(\Frob_U)_{\ov\bF_q}^*}\Big]^{\otimes I}
\end{align*}
with the structure of a module for the $\bE_1$-algebra $\bigotimes_u\End_{G(F_u)}(\cInd^{G(F_u)}_{G(\cO_u)}\La)$.

\subsection{}
Our formula for $\pi_{\infty Z,!}\cS_{\dot{V}}$ will involve the following \emph{Hecke operators} of Fargues--Scholze. For all closed points $v$ of $C$, recall from \ref{ss:localbundles} the relative Fargues--Fontaine curve $X_{S,F_v}$. Write $\Div^1_v(S)$ for the set of degree $1$ effective Cartier divisors on $X_{S,F_v}$ as in \cite[Definition II.1.19]{FS21}, so that $\Div_v^1$ is a small v-sheaf \cite[p.~52]{FS21}. For all objects $V_v$ in $\Rep_\La(\prescript{L}{}{G}_v)^I$, recall that we have a natural functor \cite[p.~18]{Sch25}
\begin{align*}
T_{V_v}:D(\Bun_{G,F_v},\La)\ra D(\Bun_{G,F_v},\La)\otimes D(\Div^1_v,\La)^{\otimes I}.
\end{align*}

For all $z$ in $Z$, let $\dot{V}_z$ be an object in $\Rep_\La(\prescript{L}{}G)^{I_z}$, and take $\dot{V}$ to be the object $\boxtimes_{z\in Z}\dot{V}_z$ in $\Rep_\La(\prescript{L}{}G)^I$. Write $V_z$ for the restriction of $\dot{V}_z$ to $(\prescript{L}{}G_z)^{I_z}$, and write $V$ for the restriction of $\dot{V}$ to $\prod_{z\in Z}(\prescript{L}{}G_z)^{I_z}$. Then $\bigotimes_{z\in Z}T_{V_z}$ yields a natural functor
\begin{align*}
T_V:\bigotimes_{z\in Z}D(\Bun_{G,F_z},\La)\ra\bigotimes_{z\in Z}\Big[D(\Bun_{G,F_z},\La)\otimes D(\Div^1_z,\La)^{\otimes I_z}\Big].
\end{align*}
Recall from \cite[Proposition 3.2]{Sch25} that $\bigotimes_{z\in Z}D(\Bun_{G,F_z},\La)=D(\prod_{z\in Z}\Bun_{G,F_z},\La)$.

\subsection{}\label{ss:choices}
We restrict $\ell$-adic sheaves on $U_{\ov\bF_q}$ to local Galois representations as follows. Write $\breve{F}_v$ for the completion of the maximal unramified extension of $F_v$. Fix an isomorphism $\ov\bF_q^\times\cong\bQ/\bZ[\frac1p]$. For all $z$ in $Z$, choose a uniformizer $\pi$ of $\cO_z$, and choose a compatible system $\{\pi^{1/n}\}_{p\nmid n}$ of $n$-th roots of $\pi$ in $\ov{F}_z$. These choices induce
\begin{enumerate}[$\bullet$]
\item group schemes $\ID_{F_z}$ and $\WD_{F_z}$ over $\La$ by \ref{ss:WeilDeligne},
\item a functor $\Psi_z:D(\Spa\bC_z,\La)\ra D(*,\La)$ by \ref{ss:nearbycycles},
\item compatible identifications
\begin{align*}
  D(\Spd\breve{F_z},\La)\cong D_{\qcoh}\big((\Spec\La)/\!\ID_{F_z}\!\big)\mbox{ and }D(\Div^1_z,\La)\cong D_{\qcoh}\big((\Spec\La)/\!\WD_{F_z}\!\big)
\end{align*}
such that the composition $D(\Spd\breve{F}_z,\La)\lra D(\Spa\bC_z,\La)\lra^{\Psi_z} D(*,\La)$ corresponds to pullback along $\Spec\La\ra(\Spec\La)/\!\ID_{F_z}$, by Lemma \ref{ss:WeilDeligne}.
\end{enumerate}
Moreover, Grothendieck's $\ell$-adic monodromy theorem implies that these choices also induce a functor $D_{\cons}(U_{\ov\bF_q},\La)\ra D_{\perf}((\Spec\La)/\!\ID_{F_z}\!)$, arising from pullback along $\Spec\breve{F}_z\ra U_{\ov\bF_q}$. $\Ind$-extending and taking Frobenius-equivariant objects yields a functor $D(U_{\ov\bF_q},\La)^{(\Frob_U)_{\ov\bF_q}^*}\ra D_{\qcoh}\big((\Spec\La)/\!\WD_{F_z}\!\big)$.

\subsection{}\label{ss:cohomologyformula}
Finally, we prove Theorem E. Recall from \ref{ss:localBunG} the natural open embedding $i_1:*/\ul{G(F_z)}\hra\Bun_{G,F_z}$, so that taking the product over all $z$ in $Z$ yields an open embedding $i_1:*/\ul{G(F_Z)}\hra\textstyle\prod_{z\in Z}\Bun_{G,F_z}$.
\begin{thm*}
The image of $\pi_{\infty Z,!}\cS_{\dot{V}}$ in
\begin{align*}
D(*/\ul{G(F_Z)},\La)\otimes\bigotimes_{z\in Z}D_{\qcoh}\big((\Spec\La)/\!\WD_{F_z}\!\big)^{\otimes I_z}
\end{align*}
is naturally isomorphic to $i^*_1(T_V(\loc_{Z,!}\!\La))$ as modules for $\bigotimes_u\End_{G(F_u)}(\cInd^{G(F_u)}_{G(\cO_u)}\La)$.
\end{thm*}
\begin{proof}
  After pulling back along $\Spec\La\ra(\Spec\La)/\!\WD_{F_z}$, we claim that the images of $\pi_{\infty Z,!}\cS_{\dot{V}}$ and $i_1^*(T_V(\loc_{Z,!}\!\La))$ in $D(*/\ul{G(F_Z)},\La)$ are naturally isomorphic as modules for $\bigotimes_u\End_{G(F_u)}(\cInd^{G(F_u)}_{G(\cO_u)}\La)$. To see this, first let us write
  \begin{align*}
    \Sht\coloneqq\Big[\displaystyle\varprojlim_{m\geq1}\Sht^I_{G,\dot{V},mZ}|_{(U\ssm Q)_{\ov\bF_q}^I}\Big]\Big/\ul{G(F_Z)}
  \end{align*}
  for convenience.

Next, Theorem \ref{ss:EGGL} enables us, after replacing $V_z$ with its restriction to the diagonal, to assume that $I_z$ is a singleton for all $z$ in $Z$. Consider the diagram
  \begin{align*}
    \xymatrix{(\Sht|_{\prod_{z\in Z}\Spa\bC_z})^\Diamond\ar[r]\ar[dd] &\displaystyle\Big[\prod_{z\in Z}\Gr_{G,V_z}|_{\Spa\bC_z}\Big]\Big/\ul{G(F_Z)}\ar@{^{(}->}[d]\ar[r] & \ast/\ul{G(F_Z)}\times\displaystyle\prod_{z\in Z}\Spa\bC_z\ar@{^{(}->}[d]^-{i_1\times\id}\\
    & \displaystyle\prod_{z\in Z}\Hck_{G,V_z}|_{\Spa\bC_z}\ar[d]^-{h_1}\ar[r]^-{h_2} & \displaystyle\prod_{z\in Z}\Bun_{G,F_z}\times\prod_{z\in Z}\Spa\bC_z\\
  \Bun_{G,U}\ar[r]^-{\loc_Z} & \displaystyle\prod_{z\in Z}\Bun_{G,F_z},}
  \end{align*}
  where $\Hck_{G,V_z}|_{\Spa\bC_z}$ denotes the Hecke stack over $\Spa\bC_z$ as in \cite[p.~16]{FS21}, and $h_1$ and $h_2$ denote the morphisms as in \cite[p.~317]{FS21}. The left square is cartesian by \ref{ss:BdRgrassmannian} and Theorem \ref{ss:fiberproductconj}, and the right square is cartesian by \cite[p.~97]{FS21}. Note that the composition in the top row equals the restriction of $\pi_{\infty Z}^\Diamond$ to $\textstyle\prod_{z\in Z}\Spa\bC_z$.

Proper base change and the above diagram imply that the image of $i^*_1(T_V(\loc_{Z,!}\!\La))$ in $D(*/\ul{G(F_Z)},\La)$ has the following description. For all $z$ in $Z$, write $\cS_{V_z}$ for the object in $D_{\mot}(\Hck_{G,V_z}|_{\Spa\bC_z})$ associated with $V_z$ via motivic geometric Satake as in \cite[Theorem 5.7]{Sch25}, and write $\cS_V$ for the pullback of $\boxtimes_{z\in Z}\cS_{V_z}$ to $(\Sht|_{\prod_{z\in Z}\Spa\bC_z})^\Diamond$. The discussion on \cite[p.~18]{Sch25} shows that $\pi_{\infty Z,!}^\Diamond\cS_V$ lies in the full subcategory
\begin{align*}
  D_{\mot}(*/\ul{G(F_Z)},\La)\otimes_{D_{\mot}(*)}\bigotimes_{z\in Z}D_{\mot}(\Spa\bC_z)\hra D_{\mot}\big(\!*\!/\ul{G(F_Z)}\times\prod_{z\in Z}\Spa\bC_z\big),
\end{align*}
where $\bigotimes_{z\in Z}$ is over $D_{\mot}(*)$, and Proposition \ref{ss:nearbycycles} implies that applying 
\begin{align*}
\bigotimes_{z\in Z}\Psi_z:D_{\mot}(*/\ul{G(F_Z)})\otimes_{D_{\mot}(*)}\bigotimes_{z\in Z}D_{\mot}(\Spa\bC_z)\ra D_{\mot}(*/\ul{G(F_Z)},\La)
\end{align*}
yields an object whose image under
\begin{align*}
D_{\mot}(*/\ul{G(F_Z)})\ra D(*/\ul{G(F_Z)},\La)
\end{align*}
is precisely the image of $i_1^*(T_V(\loc_{Z,!}\!\La))$ in $D(*/\ul{G(F_Z)},\La)$.

We lift this to a Zariski-constructible version as follows. Note that $\cS_V$ lies in 
\begin{align*}
D_{\mot}^{U^I_{\ov\bF_q}\bs\prod_{z\in Z}\Spa\bC_z}\big((\Sht|_{\prod_{z\in Z}\Spa\bC_z})^\Diamond\big)\subseteq D_{\mot}\big((\Sht|_{\prod_{z\in Z}\Spa\bC_z})^\Diamond\big),
\end{align*}
so \ref{ss:ladicrealizationglobalshtukas} yields an object 
\begin{gather*}\label{eq:pre-Ups}
\pi_{\infty Z,!}^\Diamond\cS_{V}\in D_{\mot}(*/\ul{G(F_Z)})\otimes_{D_{\mot}(*)}\Ind D_{\mot}^{U^I_{\ov\bF_q}\bs\prod_{z\in Z}\Spa\bC_z}\big(\prod_{z\in Z}\Spa\bC_z\big)\tag{$\vartriangle$}
\end{gather*}
whose image under $\Ups$ is identified with the previously mentioned 
\begin{gather*}
\pi^\Diamond_{\infty Z,!}\cS_{V}\in D_{\mot}(*/\ul{G(F_Z)})\otimes_{D_{\mot}(*)}D_{\mot}\big(\prod_{z\in Z}\Spa\bC_z\big).
\end{gather*}
The Zariski-constructible version enables us to compare with $\ell$-adic realizations. More precisely, the proof of \cite[Lemma 6.4]{LH23} identifies $r_{\ell,(\Sht|_{\prod_{z\in Z}\Spa\bC_z})^\Diamond}(\cS_V)$ with $\rho(\cS_{\dot{V}})$, so \ref{ss:ladicrealizationglobalshtukas} identifies the image of (\ref{eq:pre-Ups}) under
\begin{align*}
  r_{\ell,\prod_{z\in Z}\Spa\bC_z}\!:\,&D_{\mot}(*/\ul{G(F_Z)})\otimes_{D_{\mot}(*)}\Ind D_{\mot}^{U^I_{\ov\bF_q}\bs\prod_{z\in Z}\Spa\bC_z}\big(\prod_{z\in Z}\Spa\bC_z\big)\\
  &\ra D(*/\ul{G(F_Z)},\La)\otimes\Ind D^{U^I_{\ov\bF_q}\bs\prod_{z\in Z}\Spa\bC_z}_{\et}\big(\prod_{z\in Z}\Spa\bC_z,\La\big)
\end{align*}
with $\rho(\pi_{Z\infty,!}\cS_{\dot{V}})$. Theorem \ref{ss:EGGL} indicates that $\pi_{Z\infty,!}\cS_{\dot{V}}$ lies in the full subcategory
\begin{align*}
D(*/\ul{G(F_Z)},\La)\otimes D(U_{\ov\bF_q},\La)^{\otimes I}\hra D(*/\ul{G(F_Z)},\La)\otimes D(U^I_{\ov\bF_q},\La).
\end{align*}

Finally, consider the diagram
\begin{align*}
\xymatrixcolsep{1.5cm}
  \xymatrix{\displaystyle\bigotimes_{z\in Z}\Ind D_{\mot}^{U_{\ov\bF_q}\bs\Spa\bC_z}(\Spa\bC_z)\ar[r]^-{\bigotimes_{z\in Z}\Ups}\ar[d]^-{\bigotimes_{z\in Z}r_{\ell,\Spa\bC_z}} & \displaystyle\bigotimes_{z\in Z}D_{\mot}(\Spa\bC_z)\\
  \displaystyle\bigotimes_{z\in Z}\Ind D_{\et}^{U_{\ov\bF_q}\bs\Spa\bC_z}(\Spa\bC_z,\La) &\ar[l]_-{\bigotimes_{z\in Z}\rho} D(U_{\ov\bF_q},\La)^{\otimes I},}
\end{align*}
where the top arrow is an equivalence by Proposition \ref{ss:controlSpaC} and \cite[Proposition 10.1]{Sch24}, and the left arrow is identified with $\bigotimes_{z\in Z}\Psi_z$ by Proposition \ref{ss:nearbycycles}. Then the claim follows from the observation that external tensor product induces a morphism from this diagram to
\begin{align*}
\xymatrixcolsep{1.5cm}
  \xymatrix{\Ind D_{\mot}^{U_{\ov\bF_q}^I\bs\prod_{z\in Z}\Spa\bC_z}\big(\displaystyle\prod_{z\in Z}\Spa\bC_z\big)\ar[r]^-{\Ups}\ar[d]^-{r_{\ell,\prod_{z\in Z}\Spa\bC_z}} & D_{\mot}\big(\displaystyle\prod_{z\in Z}\Spa\bC_z\big)\\
  \Ind D_{\et}^{U_{\ov\bF_q}\bs\Spa\bC_z}\big(\displaystyle\prod_{z\in Z}\Spa\bC_z,\La\big) &\ar[l]_-{\rho} D(U_{\ov\bF_q}^I,\La).}
\end{align*}
By instead working over $\prod_{z\in Z}\Spd\breve{F}_z$ and keeping track of the partial Frobenii, we obtain the stated result.
\end{proof}

\subsection{}\label{ss:gl_fs}
As an application of Theorem \ref{ss:cohomologyformula}, we can give a quick proof that the excursion operators constructed by V. Lafforgue and Xue agree with those constructed by Fargues--Scholze. We proved this originally in \cite[Theorem 6.13]{LH23}.

Recall from Example \ref{exmp:BunGC} the set $\ker^1(F,G)$. For any finite set $I$, object $\dot{V}$ in $\Rep_\La(\prescript{L}{}G)^I$, morphism $x:\mathbf1\ra\dot{V}|_{\De(\wh{G})}$, morphism $\xi:\dot{V}|_{\De(\wh{G})}\ra\mathbf1$, and element $\ga_\bullet$ in $W_F$, recall that work of V. Lafforgue and Xue defines a $\La$-linear endomorphism $S_{I,\dot{V},x,\xi,\ga_\bullet}$ of $C_c^\infty(\coprod_{\al\in\ker^1(F,G)}G_\al(F)\bs G(\bA),\La)$ \cite[Proposition 2.2.1]{Xue20b}.

Write $\fz(G(F_v),\La)$ for the Bernstein center of $G(F_v)$ over $\La$, and write $V$ for the restriction of $\dot{V}$ to $(\prescript{L}{}G_v)^I$. When $\ga_\bullet$ lies in $W_{F_v}$, recall that work of Fargues--Scholze defines an element $\fz_{I,V,x,\xi,\ga_\bullet}$ of $\fz(G(F_v),\La)$ \cite[Theorem VIII.4.1]{FS21}.
\begin{cor*}
The element $\fz_{I,V,x,\xi,\ga_\bullet}$ acts on $C_c^\infty(\coprod_{\al\in\ker^1(F,G)}G_\al(F)\bs G(\bA),\La)$ via $S_{I,\dot{V},x,\xi,\ga_\bullet}$.
\end{cor*}
\begin{proof}
After shrinking $U$, we can assume that $v$ lies in $Z$. Take $I_z$ to be empty for all $z$ in $Z\ssm v$, so that $I_v=I$. The counit $i_{1,!}i^*_1\ra\id$ induces a commutative diagram
\begin{align*}
  \hspace{-.2cm}\xymatrix{i^*_1T_{\mathbf1}i_{1,!}i^*_1\!\loc_{Z,!}\!\La\ar[r]^-x\ar[d] & i^*_1T_{V}i_{1,!}i^*_1\!\loc_{Z,!}\!\La\ar[r]^-{\ga_\bullet}\ar[d] & i^*_1T_{V}i_{1,!}i^*_1\!\loc_{Z,!}\!\La\ar[r]^-\xi\ar[d] & i^*_1T_{\mathbf1}i_{1,!}i^*_1\!\loc_{Z,!}\!\La\ar[d]\\
  i^*_1T_{\mathbf1}\!\loc_{Z,!}\!\La\ar[r]^-x & i^*_1T_{V}\!\loc_{Z,!}\!\La\ar[r]^-{\ga_\bullet} & i^*_1T_{V}\!\loc_{Z,!}\!\La\ar[r]^-\xi & i^*_1T_{\mathbf1}\!\loc_{Z,!}\!\La.}
\end{align*}
Because $T_{\mathbf1}$ is naturally isomorphic to the identity, the left and right arrows are isomorphisms, and the corners are identified with $i^*_1\!\loc_{Z,!}\!\La$. By using Lemma \ref{ss:BunGpullback} and arguing as in Example \ref{exmp:BunGC}, we see that $H^0(i^*_1\!\loc_{Z,!}\!\La)$ is naturally isomorphic to $C_c^\infty(\coprod_{\al\in\ker^1(F,G)}G_\al(F)\bs G(\bA)/G(\bO_U),\La)$. After applying $H^0(-)$ to the above diagram, the top arrow equals the action of $\fz_{I,V,x,\xi,\ga_\bullet}$ by construction, and the bottom arrow equals $S_{I,\dot{V},x,\xi,\ga_\bullet}$ by Theorem \ref{ss:cohomologyformula}. This indicates that the restrictions of $\fz_{I,V,x,\xi,\ga_\bullet}$ and $S_{I,\dot{V},x,\xi,\ga_\bullet}$ to $C_c^\infty(\coprod_{\al\in\ker^1(F,G)}G_\al(F)\bs G(\bA)/G(\bO_U),\La)$ coincide. Finally, taking $\bigcup_Z$ yields the desired result.
\end{proof}
\begin{rems}\label{rems:gl_fs}\hfill
  \begin{enumerate}[1)]
  \item The proof of Corollary \ref{ss:gl_fs} only relies on Theorem \ref{ss:cohomologyformula} after taking cohomology groups. In particular, one does not need to appeal to Theorem \ref{ss:EGGL}; see Remark \ref{rem:Xue}.
  \item The main ingredient for the representation-theoretic results in \cite{LH23} is \cite[Theorem 6.13]{LH23}. Therefore one can use Corollary \ref{ss:gl_fs} to prove all the representation-theoretic results in \cite{LH23}; see the proof of \cite[Theorem 6.16]{LH23}.
  \end{enumerate}
\end{rems}

\section{Relation with Langlands duality}\label{s:langlands}
In this section, we discuss the role that $\Bun_{G,F}$ plays in the Langlands correspondence. First, we recall the moduli of Galois representations and their derived categories of ind-coherent sheaves, which are the main players on the Galois side. By restricting Galois representations to decomposition groups, we obtain a map $\lres_Z$ analogous to $\loc_Z$; we study the pushforward of the coherent dualizing complex $\om$ along $\lres_Z$.

Next, we recall the spectral action of Fargues--Scholze, which enables us to state their categorical Langlands conjecture. We can then state Conjecture F, and we explain how Conjecture F implies conjectures of Arinkin--Gaitsgory--Kazhdan--Raskin--Rozenblyum--Varshavsky and Zhu. Finally, we prove Conjecture F when $G$ is commutative; this is Theorem G.

\subsection{}\label{ss:LS}
We start by recalling the stacks of Galois representations that we will use. \textbf{For the rest of this paper, assume that $\ell\geq3$.}\footnote{This arises from the fact that de Jong's conjecture \cite[Conjecture 1.1]{deJ01}, which is used to prove basic properties of $\LS_{\prescript{L}{}G,U}$ \cite[Theorem 3.29]{Zhu20}, is only known when $\ell\geq3$ \cite{Ga07}.} Write $\LS_{\prescript{L}{}{G},U}$ for the derived algebraic stack over $\La$ of continuous $\prescript{L}{}{G}$-valued representations of $W_U$
\begin{enumerate}[$\bullet$]
\item as in \cite[(3.24)]{Zhu20} when $Z$ is nonempty, or
\item as in \cite[Remark 3.36]{Zhu20} when $Z$ is empty,
\end{enumerate}
using our choice of $\sqrt{q}$ to replace $\prescript{c}{}{G}$ with $\prescript{L}{}{G}$ \cite[Remark 3.1]{Zhu20}. When $Z$ is nonempty, write $\LS^{\square}_{\prescript{L}{}{G},U}$ for its framed version as in \cite[(3.24)]{Zhu20}, which the proof of \cite[Theorem 3.29]{Zhu20} shows is a disjoint union of derived affine schemes almost of finite type over $\La$.

Similarly, for all closed points $v$ of $C$, write $\LS_{\prescript{L}{}{G},F_v}$ for the algebraic stack over $\La$ of continuous $\prescript{L}{}{G}_v$-valued representations of $W_{F_v}$ as in \cite[(3.3)]{Zhu20}, and write $\LS_{\prescript{L}{}{G},F_v}^{\square}$ for its framed version as in \cite[(3.3)]{Zhu20}. Then \cite[Corollary 4.5]{Sch25} implies that $\LS_{\prescript{L}{}{G},F_v}$ is the base change to $\La$ of the algebraic stack from \cite[Definition 4.3]{Sch25}, and \cite[Theorem VIII.1.3]{FS21} shows that $\LS_{\prescript{L}{}{G},F_v}^{\square}$ is a disjoint union of affine schemes of finite type over $\La$.

For all closed points $u$ of $U$, write $\LS_{\prescript{L}{}{G},\cO_u}\subseteq\LS_{\prescript{L}{}{G},F_u}$ for the closed substack of unramified $\prescript{L}{}{G}_v$-valued representations of $W_{F_v}$ as in \cite[p.~82]{Zhu20}. Write $\LS_{\prescript{L}{}{G},\cO_u}^{\square}$ for its framed version, which is naturally isomorphic to $\wh{G}$.

\subsection{}
We will consider ind-coherent sheaves on these stacks. Namely, recall the $\La$-linear ind-coherent $6$-functor formalism $\IndCoh(-)$ from \cite[Lecture VIII.7]{Sch22} on
\begin{align*}
\{\mbox{derived qcqs schemes almost of finite type over }\La\},
\end{align*}
where we use the notation of Gaitsgory--Rozenblyum \cite[p.~273]{GR17a} for the associated functors instead of the notation of \cite[Definition A.5.6]{Man22}. Then \cite[Theorem 3.4.11]{HM24} extends this to a $6$-functor formalism on
\begin{align*}
\{\mbox{derived algebraic stacks over }\La\}.
\end{align*}
Let $D$ be one of $\{U\neq C,F_v\}$. Then \ref{ss:LS} implies that the value of $\IndCoh$ on $\LS_{\prescript{L}{}{G},D}$ is naturally equivalent to $\Ind D_{\coh}^{\qc}(\LS_{\prescript{L}{}{G},D})$.

\subsection{}\label{ss:cohversusperf}
When restricted to bounded below objects, there is no difference between ind-perfect and ind-coherent objects on $\LS_{\prescript{L}{}G,D}$. More precisely, applying $\Ind$ to the inclusion $D_{\perf}^{\qc}(\LS_{\prescript{L}{}{G},D})\subseteq D_{\coh}^{\qc}(\LS_{\prescript{L}{}{G},D})$ yields a fully faithful functor
\begin{align*}
\Xi:\Ind D_{\perf}^{\qc}(\LS_{\prescript{L}{}{G},D})\hra\Ind D_{\coh}^{\qc}(\LS_{\prescript{L}{}{G},D}).
\end{align*}
Write $\Psi:\Ind D_{\coh}^{\qc}(\LS_{\prescript{L}{}{G},D})\ra\Ind D_{\perf}^{\qc}(\LS_{\prescript{L}{}{G},D})$ for its right adjoint. Then arguing as in the proof of \cite[Proposition 1.2.4]{Gai13} shows that $\Psi$ restricts to an equivalence
\begin{align*}
\bigcup_{n\in\bZ}\big[\Ind D_{\coh}^{\qc}(\LS_{\prescript{L}{}{G},D})\big]^{\geq n}\ra^\sim\bigcup_{n\in\bZ}\big[\Ind D_{\perf}^{\qc}(\LS_{\prescript{L}{}{G},D})\big]^{\geq n},
\end{align*}
which we use to identify $D^{\qc}_{\coh}(\LS_{\prescript{L}{}{G},D})$ with its image under $\Psi$.

\subsection{}
Restriction induces the following morphisms between our stacks of Galois representations. For all closed points $u$ of $U$ and $z$ in $Z$, write
\begin{align*}
\lres_u:\LS_{\prescript{L}{}{G},U}\ra\LS_{\prescript{L}{}{G},\cO_u}\mbox{ and }\lres_z:\LS_{\prescript{L}{}{G},U}\ra\LS_{\prescript{L}{}{G},F_z}
\end{align*}
for the resulting morphisms over $\La$. Write $\lres_Z:\LS_{\prescript{L}{}{G},U}\ra\textstyle\prod_{z\in Z}\LS_{\prescript{L}{}{G},F_z}$ for the morphism $(\lres_z)_{z\in Z}$, where all products are over $\La$. By \ref{ss:LS}, the value of $\IndCoh$ on $\prod_{z\in Z}\LS_{\prescript{L}{}{G},F_z}$ is naturally equivalent to
\begin{align*}
\bigotimes_{z\in Z}\Ind D_{\coh}^{\qc}(\LS_{\prescript{L}{}{G},F_z}).
\end{align*}

\subsection{}
For places in $U$, we have the following (derived) \emph{spectral} Hecke algebra action on $\lres_{Z,*}(\om_{\LS_{\prescript{L}{}G,U}})$. Let $Q$ be a finite nonempty set of closed points of $U$, and write
\begin{align*}
\lres_Q:\LS_{\prescript{L}{}{G},U}\ra\textstyle\prod_{u\in Q}\LS_{\prescript{L}{}{G},\cO_u}
\end{align*}
for the morphism $(\lres_u)_{u\in U}$. Then \cite[Lemma 3.34]{Zhu20} implies that the square
\begin{align*}
  \xymatrixcolsep{1.5cm}
  \xymatrix{\LS_{\prescript{L}{}{G},U}\ar[r]^-{(\lres_Q,\lres_Z)}\ar[d] & \displaystyle\big(\prod_{u\in Q}\LS_{\prescript{L}{}{G},\cO_u}\big)\times\big(\prod_{z\in Z}\LS_{\prescript{L}{}{G},F_z}\big)\ar[d]\\
 \LS_{\prescript{L}{}{G},U\ssm Q}\ar[r] & \big(\displaystyle\prod_{u\in Q}\LS_{\prescript{L}{}{G},F_u}\big)\times\big(\prod_{z\in Z}\LS_{\prescript{L}{}{G},F_z}\big)& }
\end{align*}
is derived cartesian. Moreover, \ref{ss:LS} implies that bottom arrow is $*$-able, so the same holds for the top arrow. Therefore Lemma \ref{ss:derivedHeckeformalism} endows $\lres_{Z,*}(\om_{\LS_{\prescript{L}{}{G},U}})$ in $\bigotimes_{z\in Z}\Ind D_{\coh}^{\qc}(\LS_{\prescript{L}{}{G},F_z})$ with the structure of a module for the $\bE_1$-algebra
\begin{align*}
\End_{D_{\coh}(\prod_{u\in Q}\LS_{\prescript{L}{}{G},F_u})}(\sO_{\prod_{u\in Q}\LS_{\prescript{L}{}{G},\cO_u}}) = \bigotimes_{u\in Q}\End_{D_{\coh}(\LS_{\prescript{L}{}G,F_u})}(\sO_{\LS_{\prescript{L}{}{G},\cO_u}})
\end{align*}
over $\La$. Taking the colimit over $Q$ shows that $\lres_{Z,*}(\om_{\LS_{\prescript{L}{}{G},U}})$ is naturally a module for the $\bE_1$-algebra $\bigotimes_u\End_{D_{\coh}(\LS_{\prescript{L}{}G,F_u})}(\LS_{\prescript{L}{}{G},\cO_u})$, where $u$ runs over closed points of $U$.

\subsection{}
Write $D_{\coh}^{\qc}(\LS_{\prescript{L}{}{G},F_v})_{\Nilp}\subseteq D_{\coh}^{\qc}(\LS_{\prescript{L}{}{G},F_v})$ for the full subcategory of objects with nilpotent singular support as in \cite[VIII.2.2.2]{FS21}. It is $\Ind D_{\coh}^{\qc}(\LS_{\prescript{L}{}{G},F_v})_{\Nilp}$ that appears in the categorical local Langlands conjecture, so we need to check that $\lres_{Z,*}(\LS_{\prescript{L}{}G,U})$ lies in this subcategory.
\begin{prop*}
  The object $\lres_{Z,*}(\om_{\LS_{\prescript{L}{}{G},U}})$ in $\bigotimes_{z\in Z}\Ind D_{\coh}^{\qc}(\LS_{\prescript{L}{}{G},F_z})$ lies in
  \begin{align*}
    \bigotimes_{z\in Z}\Ind D_{\coh}^{\qc}(\LS_{\prescript{L}{}{G},F_z})_{\Nilp}.
  \end{align*}
\end{prop*}
\begin{proof}
  When $Z$ is empty, this condition is also empty. When $Z$ is nonempty, \cite[Theorem 3.29]{Zhu20} implies that $\LS_{\prescript{L}{}G,U}^{\square}$ is quasismooth over $\La$, so \cite[Corollary 2.2.8]{AG15}\footnote{While \cite{AG15} works over an algebraically closed field of characteristic zero, the proof of \cite[Corollary 2.2.8]{AG15} does not use this assumption.} shows that its dualizing complex $\om_{\LS_{\prescript{L}{}G,U}^{\square}}$ is a shift of a line bundle. By smooth descent, the same holds for $\om_{\LS_{\prescript{L}{}G,U}}$, so the result follows from \ref{ss:cohversusperf}.
\end{proof}

\subsection{}\label{ss:spectralaction}
The local automorphic and spectral sides are related by the following \emph{spectral action} of Fargues--Scholze. Write $Z_G$ for the center of $G$ over $F$, and \textbf{for the rest of this paper, assume that $\pi_0((Z_G)_{\ov{F}})$ is invertible in $\La$.} Then \cite[Theorem 6.1]{Sch25} and \ref{ss:LS} show that there is a natural $\La$-linear action of $D_{\perf}(\LS_{\prescript{L}{}{G},F_v})$ on $D(\Bun_{G,F_v},\La)^\om$. Write $V_v$ for the vector bundle on $\LS_{\prescript{L}{}{G},F_v}$ associated with the object $V_v$ in $\Rep_\La(\prescript{L}{}{G})^I$, so that by construction $V_v$ in $D_{\perf}(\LS_{\prescript{L}{}{G},F_v})$ acts on $D(\Bun_{G,F_v},\La)^\om$ via $T_{V_v}$.

By restricting to $D_{\perf}^{\qc}(\LS_{\prescript{L}{}{G},F_v})$ and applying $\Ind$, the above naturally extends to a $\La$-linear colimit-preserving action of $\Ind D_{\perf}^{\qc}(\LS_{\prescript{L}{}{G},F_v})$ on $D(\Bun_{G,F_v},\La)$.

\subsection{}\label{ss:FS'sconjecture}
We now use the spectral action to state the categorical local Langlands conjecture of Fargues--Scholze. \textbf{For the rest of this paper, assume that $G$ is quasisplit over $F$.} Then we can choose $B$ to be a Borel subgroup of $G$ over $F$ and $T$ to be a maximal subtorus of $B$ over $F$. Write $N$ for the unipotent radical of $B$, and let $\psi:N(\bA)\ra\La^\times$ be a continuous homomorphism trivial on $N(F)$ such that, for all closed points $v$ of $C$, its restriction $\psi_v$ to $N(F_v)$ is generic.

Consider the colimit-preserving functor $a_{\psi_v}:\Ind D_{\perf}^{\qc}(\LS_{\prescript{L}{}{G},F_v})\ra D(\Bun_{G,F_v},\La)$ given by acting as in \ref{ss:spectralaction} on $i_{1,!}\!\cInd_{N(F_v)}^{G(F_v)}\psi_v$. Denote its right adjoint by
\begin{align*}
c_{\psi_v}:D(\Bun_{G,F_v},\La)\ra\Ind D_{\perf}^{\qc}(\LS_{\prescript{L}{}{G},F_v}).
\end{align*}
\begin{conj*}[{\cite[Conjecture X.3.5]{FS21}}]
The functor $c_{\psi_v}$ restricts to an equivalence
  \begin{align*}
D(\Bun_{G,F_v},\La)^\om\ra^\sim D_{\coh}^{\qc}(\LS_{\prescript{L}{}{G},F_v})_{\Nilp}.
  \end{align*}
Consequently, applying $\Ind$ yields an equivalence
  \begin{align*}
\bL_{\psi_v}:D(\Bun_{G,F_v},\La)\ra^\sim\Ind D_{\coh}^{\qc}(\LS_{\prescript{L}{}{G},F_v})_{\Nilp}.
  \end{align*}
\end{conj*}
Conjecture \ref{ss:FS'sconjecture} is known when $G$ is a torus \cite[Theorem 6.4.1]{Zou24}.

\subsection{}\label{ss:Zhu'sconjecture}
To compare the actions of the (derived) automorphic and spectral Hecke algebras for places in $U$, we consider the following conjecture of Zhu.
\begin{conj*}[{\cite[Conjecture 4.11]{Zhu20}}]Let $u$ be a closed point of $U$.
  \begin{enumerate}[1)]
  \item Assume Conjecture \ref{ss:FS'sconjecture} for $u$. Then $\bL_{\psi_u}(i_{1,!}\!\cInd_{G(\cO_u)}^{G(F_u)}\La)\cong\sO_{\LS_{\prescript{L}{}{G},\cO_u}}$.
  \item There is a natural isomorphism of $\bE_1$-algebras over $\La$
\begin{align*}
\End_{G(F_u)}(\cInd_{G(\cO_u)}^{G(F_u)}\La)\cong\End_{D_{\coh}(\LS_{\prescript{L}{}{G},F_u})}(\sO_{\LS_{\prescript{L}{}{G},\cO_u}}).
\end{align*}
\end{enumerate}
\end{conj*}
Note that Conjecture \ref{ss:Zhu'sconjecture}.1) implies Conjecture \ref{ss:Zhu'sconjecture}.2). The proof of \cite[Theorem 6.4.1]{Zou24} shows that Conjecture \ref{ss:Zhu'sconjecture}.1) is known when $G$ is a torus; see \cite[Proposition 4.13]{Zhu20} for an explicit description of the resulting isomorphism in Conjecture \ref{ss:Zhu'sconjecture}.2).

Conjecture \ref{ss:Zhu'sconjecture}.2) is known when $\La$ equals $L$ \cite[Theorem 5.3 (2)]{Zhu25}.

\subsection{}\label{ss:globalconjecture}
Finally, we arrive at Conjecture F.
\begin{conj*}
  Assume Conjecture \ref{ss:FS'sconjecture} for all $z$ in $Z$, which yields an equivalence
  \begin{align*}
    \bigotimes_{z\in Z}\bL_{\psi_z}:\bigotimes_{z\in Z}D(\Bun_{G,F_z},\La)\ra^\sim\bigotimes_{z\in Z}\Ind D^{\qc}_{\coh}(\LS_{\prescript{L}{}G,F_z})_{\Nilp}.
  \end{align*}
  Assume Conjecture \ref{ss:Zhu'sconjecture}.2) for all closed points $u$ in $U$, which yields an isomorphism
  \begin{align*}
    \bigotimes_u\End_{G(F_u)}(\cInd_{G(\cO_u)}^{G(F_u)}\La)\cong\bigotimes_u\End_{D_{\coh}(\LS_{\prescript{L}{}{G},F_u})}(\sO_{\LS_{\prescript{L}{}{G},\cO_u}}).
  \end{align*}
  Then the $\bigotimes_u\End_{G(F_u)}(\cInd_{G(\cO_u)}^{G(F_u)}\La)$-module
  \begin{align*}
    \loc_{Z,!}\!\La\in\bigotimes_{z\in Z}D(\Bun_{G,F_z},\La)
  \end{align*}
  corresponds under $\bigotimes_{z\in Z}\bL_{\psi_z}$ to the $\bigotimes_u\End_{D_{\coh}(\LS_{\prescript{L}{}{G},F_u})}(\sO_{\LS_{\prescript{L}{}{G},\cO_u}})$-module
  \begin{align*}
    \lres_{Z,*}(\om_{\LS_{\prescript{L}{}{G},U}})\in\bigotimes_{z\in Z}\Ind D_{\coh}^{\qc}(\LS_{\prescript{L}{}{G},F_z})_{\Nilp}.
  \end{align*}
\end{conj*}

\begin{exmp}\label{exmp:AGKRRV}
Assume that $Z$ is empty. Then Example \ref{exmp:BunGC} identifies $\Bun_{G,C}$ with the constant stack over $\ov\bF_q$ associated with the groupoid (\ref{eq:automorphicspace}), so 
\begin{align*}
\loc_{\varnothing,!}\!\La=\bigoplus_{\al\in\ker^1(F,G)}\Ga_c(\ul{G_\al(F)\bs G(\bA)/G(\bO)},\La),
\end{align*}
where we remind the reader that $\Ga_c(\ul{G_\al(F)\bs G(\bA)/G(\bO)},\La)$ can have cohomology in nonzero degrees when $\La\neq L$ \cite[Remark 0.0.3]{Xue20b}. On the other hand, by definition 
\begin{align*}
\lres_{\varnothing,*}(\om_{\LS_{\prescript{L}{}{G},C}}) = \Ga(\LS_{\prescript{L}{}{G},C},\om_{\LS_{\prescript{L}{}{G},C}}).
\end{align*}
Altogether, we see that Conjecture \ref{ss:globalconjecture} specializes to \cite[Conjecture 24.8.6]{AGKRRV20a} when $G$ is split, $\La=L$, and $Z$ is empty.
\end{exmp}

\subsection{}\label{ss:spectralcohomologyformula}
More generally, Conjecture \ref{ss:globalconjecture} implies the following formula for the cohomology of stacks of global shtukas for \emph{extended pure inner twists} of $G$, which is a conjecture of Zhu \cite[Conjecture 4.49]{Zhu20}. Let $b$ be an element in $G(U_{\ov\bF_q})$ whose image in $B(F,G)$ is basic, and recall from \ref{ss:twistingb} the parahoric group scheme $G_b$ over $U$. Using fpqc descent, extend $G_b$ to a parahoric group scheme over $C$. Since $G_b$ is an inner twist of $G$ over $F$, we have compatible identifications $\prescript{L}{}G=\prescript{L}{}G_b$ and $\prescript{L}{}G_v=\prescript{L}{}G_{b,v}$. Write
\begin{align*}
\pi_{\infty Z}^{G_b}:\varprojlim_{m\geq1}\Sht^I_{G_b,\dot{V},mZ}|_{(U\ssm Q)^I_{\ov\bF_q}}\ra(U\ssm Q)_{\ov\bF_q}^I
\end{align*}
for the structure morphism.

Recall the open embedding $i_b:*/\ul{G_b(F_z)}\hra\Bun_{G,F_z}$ \cite[Theorem III.4.5]{FS21}, so that taking the product over all $z$ in $Z$ yields an open embedding
\begin{align*}
i_b:*/\ul{G_b(F_Z)}\hra\textstyle\prod_{z\in Z}\Bun_{G,F_z}.
\end{align*}
For all $z$ in $Z$, let $K_z$ be a compact open subgroup of $G_b(F_z)$. Write $K$ for $\prod_{z\in Z}K_z$, and consider the structure morphism
\begin{align*}
\pi_K^{G_b}:\Big[\varprojlim_{m\geq1}\Sht^I_{G_b,\dot{V},mZ}|_{(U\ssm Q)^I_{\ov\bF_q}}\Big]\Big/\ul{K}\ra(U\ssm Q)_{\ov\bF_q}^I.
\end{align*}

Write $\dot{V}$ for the vector bundle on $\LS_{\prescript{L}{}G,U}$ associated with the object $\dot{V}$ in $\Rep_\La(\prescript{L}{}G)^I$, and write $\bD(-)$ for the Grothendieck--Serre dual.
\begin{prop*}
  Assume Conjecture \ref{ss:globalconjecture}. Then the image of $\pi_{K,!}^{G_b}\cS_{\dot{V}}$ in
  \begin{align*}
    \bigotimes_{z\in Z}D_{\qcoh}\big((\Spec\La)/\!\WD_{F_z}\big)^{\bigotimes I_z}
  \end{align*}
  is isomorphic as modules for $\bigotimes_u\End_{G(F_u)}(\cInd_{G(\cO_u)}^{G(F_u)}\La)$ to
  \begin{align*}
\Ga\big(\LS_{\prescript{L}{}G,U},\dot{V}\otimes\lres_Z^!(\boxtimes_{z\in Z}\bD(\bL_{\psi_z}(i_{b,!}\cInd_{K_z}^{G_b(F_z)})))\big).
  \end{align*}
\end{prop*}
\begin{proof}
Because $b$ is basic, Proposition \ref{ss:twistingb} and Theorem \ref{ss:cohomologyformula} identify
  \begin{align*}
    \pi_{K,!}^{G_b}\cS_{\dot{V}} = (\pi_{\infty Z,!}^{G_b}\cS_{\dot{V}})^K &= \Hom_{G_b(F_Z)}\!\big(\cInd_K^{G_b(F_Z)}\La,\pi_{\infty Z,!}^{G_b}\cS_{\dot{V}}\big)\\
                                                                           &= \Hom_{G_b(F_Z)}\!\big(\cInd_K^{G_b(F_Z)}\La,i_b^*(T_V(\loc_{Z,!}\!\La))\big) \\
    &= \Hom\!\big(i_{b,!}\cInd_K^{G_b(F_Z)}\La,T_V(\loc_{Z,!}\!\La)\big).
  \end{align*}
By Conjecture \ref{ss:globalconjecture} and the projection formula, this is isomorphic to
  \begin{align*}
    &\Hom\!\big(\boxtimes_{z\in Z}\bL_{\psi_z}(i_{b,!}\cInd_{K_z}^{G_b(F_z)}\La),V\otimes\lres_{Z,*}(\om_{\LS_{\prescript{L}{}G,U}})\big)\\
    =\,&\Hom\!\big(\boxtimes_{z\in Z}\bL_{\psi_z}(i_{b,!}\cInd_{K_z}^{G_b(F_z)}\La),\lres_{Z,*}(\dot{V}\otimes\om_{\LS_{\prescript{L}{}G,U}})\big).
  \end{align*}
  Finally, applying Grothendieck--Serre duality yields
  \begin{align*}
    &\Hom\!\big(\lres_{Z,!}(\bD(\dot{V}\otimes\om_{\LS_{\prescript{L}{}G,U}})),\boxtimes_{z\in Z}\bD(\bL_{\psi_z}(i_{b,!}\cInd_{K_z}^{G_b(F_z)}))\big)\\
    =\,&\Hom\!\big(\lres_{Z,!}(\dot{V}^\vee),\boxtimes_{z\in Z}\bD(\bL_{\psi_z}(i_{b,!}\cInd_{K_z}^{G_b(F_z)}))\big) \\
    =\,&\Ga\big(\LS_{\prescript{L}{}G,U},\dot{V}\otimes\lres_Z^!(\boxtimes_{z\in Z}\bD(\bL_{\psi_z}(i_{b,!}\cInd_{K_z}^{G_b(F_z)})))\big).\qedhere
  \end{align*}
\end{proof}
\begin{rem}
When $b=1$, it is expected that $\bL_{\psi_z}(i_{1,!}\cInd_{K_z}^{G(F_z)}\La)$ is isomorphic to its Grothendieck--Serre dual \cite[Remark 4.27]{Zhu20}.
\end{rem}

\subsection{}\label{ss:toruscohomology}
In this rest of this section, our goal is to prove Conjecture \ref{ss:globalconjecture} when $G=T$ is a torus. We start by recalling work of Langlands \cite{Lan97} on his conjectures for tori. Write $W_{\wt{F}/F}\coloneqq W_F/\ov{[W_{\wt{F}},W_{\wt{F}}]}$ for the relative Weil group of $\wt{F}/F$, so that global class field theory gives a natural short exact sequence of locally profinite groups
\begin{align*}
\xymatrix{1\ar[r] & \wt{F}^\times\bs\bA_{\wt{F}}^\times\ar[r] & W_{\wt{F}/F}\ar[r] & \Gal(\wt{F}/F)\ar[r] & 1.}
\end{align*}

Write $\bT$ for the locally profinite group $\big[T(\wt{F})\bs T(\bA_{\wt{F}})\big]^{\Gal(\wt{F}/F)}$. Then the long exact sequence for group cohomology yields a natural short exact sequence
\begin{align*}
\xymatrix{1\ar[r] & T(F)\bs T(\bA)\ar[r] & \bT\ar[r] & \ker^1(F,T)\ar[r] & 1}
\end{align*}
of locally profinite groups. Recall that the transfer map induces an isomorphism
\begin{align*}
  H_1(W_{\wt{F}/F},X_*(T))\ra^\sim H_1(\wt{F}^\times\bs\bA_{\wt{F}}^\times,X_*(T))^{\Gal(\wt{F}/F)} = \bT
\end{align*}
\cite[p.~233]{Lan97}, and for all $\Gal(\wt{F}/F)$-stable open subgroups $P$ of $\wt{F}^\times\bs\bA_{\wt{F}}^\times$, write
\begin{align*}
\Te(P)\coloneqq\ker\big[H_1(W_{\wt{F}/F},X_*(T))\ra H_1(W_{\wt{F}/F}/P,X_*(T))\big].
\end{align*}

\begin{lem*}
Under the identification $H_1(W_{\wt{F}/F},X_*(T))\ra^\sim\bT$, the subgroup $\Te(P)$ corresponds to the image of $P\otimes_\bZ X_*(T)\subseteq T(\wt{F})\bs T(\bA_{\wt{F}})$ under the norm map.
\end{lem*}
\begin{proof}
  The five term exact sequence for group homology identifies $\Te(P)$ with the image of $H_1(P,X_*(T))=P\otimes_\bZ X_*(T)$ under the corestriction map. Postcomposing the corestriction map with the transfer map yields the norm map, as desired.
\end{proof}

\subsection{}\label{ss:boundedramification}
Assume that $Z$ is nonempty. We will stratify $\LS_{\prescript{L}{}T,U}$ by bounding ramification as follows. Let $K$ be a $\Gal(\wt{F}/F)$-stable compact open subgroup of $\wt{F}_{\wt{Z}}^\times$, and write $\LS_{\prescript{L}{}T,K}^{\square}$ for the presheaf on $\{\mbox{affine schemes over }\La\}$ whose $\Spec{A}$-points consist of $1$-cocycles $W_{\wt{F}/F}/\bO_{\wt{U}}^\times K\ra\wh{T}(A)$. Since $W_{\wt{F}/F}/\bO_{\wt{U}}^\times K$ is finitely generated, $\LS_{\prescript{L}{}T,K}^{\square}$ is an affine scheme of finite type over $\La$. Write $\LS_{\prescript{L}{}T,K}$ for the algebraic stack $\LS_{\prescript{L}{}T,K}^{\square}\!/\wh{T}$ over $\La$.

For any finitely generated abelian group $M$, write $\ul{\Hom}(M,\bG_m)$ for its Cartier dual over $\La$. Note that $\ul{\Hom}(M,\bG_m)$ is a flat affine scheme of finite type over $\La$.
\begin{prop*}
  The algebraic stack $\LS_{\prescript{L}{}T,K}$ is naturally a gerbe over
  \begin{align*}
  \ul{\Hom}\big(\bT/\bO_U^\times\!\Nm_{\wt{F}/F}(K\otimes_\bZ X_*(T)),\bG_m\big)
  \end{align*}
  banded by $\wh{T}^{\Gal(\wt{F}/F)}$.
\end{prop*}
\begin{proof}
  Write $H^{\mathrm{pre}}$ for the presheaf on $\{\mbox{affine schemes over }\La\}$ whose $\Spec{A}$-points consist of the set-theoretic quotient
\begin{align*}
\LS_{\prescript{L}{}T,K}^{\square}(A)\big/\wh{T}(A) = H^1\big(W_{\wt{F}/F}/\bO_{\wt{U}}^\times K,\wh{T}(A)\big),
\end{align*}
and write $H$ for its fpqc-sheafification. Using the presentation $\LS_{\prescript{L}{}T,K} = \LS_{\prescript{L}{}T,K}^{\square}\!/\wh{T}$, we see the inertia stack of $\LS_{\prescript{L}{}T,K}$ is naturally isomorphic to $\wh{T}^{\Gal(\wt{F}/F)}\times\LS_{\prescript{L}{}T,K}$ over $\LS_{\prescript{L}{}T,K}$. Therefore \cite[Tag 06QJ]{stacks-project} shows that $\LS_{\prescript{L}{}T,K}$ is naturally a gerbe over $H$ banded by $\wh{T}^{\Gal(\wt{F}/F)}$.

Since $X^*(T)\otimes_\bZ A^\times=\wh{T}(A)$, the natural map
  \begin{align*}
    H^1\big(W_{\wt{F}/F}/\bO_{\wt{U}}^\times K,X^*(T)\otimes_\bZ A^\times\big)\ra\Hom\big(H_1\big(W_{\wt{F},F}/\bO_{\wt{U}}^\times K,X_*(T)\big),A^\times\big)
  \end{align*}
  yields a morphism $\io^{\mathrm{pre}}:H^{\mathrm{pre}}\ra\ul{\Hom}(H_1(W_{\wt{F},F}/\bO_{\wt{U}}^\times K,X_*(T)),\bG_m)$. When $A^\times$ is divisible, the above map is an isomorphism, so \cite[Lemma 4.1.1]{Zou24} implies that $\io^{\mathrm{pre}}$ fpqc-sheafifies to an isomorphism $\io:H\ra^\sim\ul{\Hom}(H_1(W_{\wt{F},F}/\bO_{\wt{U}}^\times K,X_*(T)),\bG_m)$. Finally, Lemma \ref{ss:toruscohomology} shows that the transfer map yields an isomorphism
\begin{gather*}
H_1\big(W_{\wt{F},F}/\bO_{\wt{U}}^\times K,X_*(T)\big)\ra^\sim\bT\big/\bO_U^\times\!\Nm_{\wt{F}/F}(K\otimes_\bZ X_*(T)).\qedhere
\end{gather*}
\end{proof}

\subsection{}\label{ss:LStorus}
Assume that $Z$ is nonempty. Using work of Langlands, we can prove the following explicit description of $\LS_{\prescript{L}{}T,U}$.
\begin{thm*}
  The derived algebraic stack $\LS_{\prescript{L}{}T,U}$ is classical and equals
  \begin{align*}
    \bigcup_K\LS_{\prescript{L}{}T,K},
  \end{align*}
where $K$ runs over $\Gal(\wt{F}/F)$-stable pro-$p$ open subgroups of $\wt{F}_{\wt{Z}}^\times$ and the transition morphisms are clopen embeddings.
\end{thm*}
\begin{proof}
Recall that the $\Spec{A}$-points of $\LS_{\prescript{L}{}T,U}^{\square,\cl}$ consist of $1$-cocycles $\vp:W_U\ra\wh{T}(A)$ that are continuous with respect to the ind-$\ell$-adic topology on $A$ as in \cite[Remark 2.54]{Zhu20}. Because $\vp|_{W_{\wt{U}}}$ is a homomorphism and $\wh{T}(A)$ is abelian, $\vp$ factors through $W_U\ra W_{\wt{F}/F}/\bO_{\wt{U}}^\times$. Since $\wt{F}^\times\bs\bA_{\wt{F}}^\times/\bO_{\wt{U}}^\times$ is locally pro-$p$, there exists a $\Gal(\wt{F}/F)$-stable pro-$p$ open subgroup $K$ of $\wt{F}^\times_{\wt{Z}}$ such that $\vp$ factors through $W_{\wt{F}/F}/\bO^\times_{\wt{U}}K$. This shows that
\begin{align*}
\LS^{\square,\cl}_{\prescript{L}{}T,U} = \bigcup_{K}\LS_{\prescript{L}{}T,K}^{\square},
\end{align*}
where $K$ runs over $\Gal(\wt{F}/F)$-stable pro-$p$ open subgroups of $\wt{F}_{\wt{Z}}^\times$, and the proof of \cite[Theorem VIII.1.3]{FS21} implies that the transition morphisms are clopen embeddings. Quotienting by $\wh{T}$ yields an analogous description 
\begin{align*}
\LS_{\prescript{L}{}T,U}^{\cl} = \bigcup_{K}\LS_{\prescript{L}{}T,K}.
\end{align*}

Because the finitely generated abelian group $\bT/\bO_U^\times\!\Nm_{\wt{F}/F}(K\otimes_\bZ X_*(T))$ has rank $\dim\wh{T}^{\Gal(\wt{F}/F)}$, its Cartier dual $\ul{\Hom}(\bT/\bO_U^\times\!\Nm_{\wt{F}/F}(K\otimes_\bZ X_*(T)),\bG_m)$ has relative dimension $\dim\wh{T}^{\Gal(\wt{F}/F)}$ over $\La$. Hence Proposition \ref{ss:boundedramification} indicates that $\LS_{\prescript{L}{}T,K}$ has relative dimension $0$ over $\La$. Finally, $\LS_{\prescript{L}{}T,U}$ also has relative virtual dimension $0$ over $\La$, so \cite[Proposition B.0.1]{Han23} implies that $\LS_{\prescript{L}{}T,U}$ is classical.
\end{proof}

\subsection{}
Finally, we prove Theorem G.
\begin{thm*}
Conjecture \ref{ss:globalconjecture} holds when $G=T$ is a torus.
\end{thm*}
\begin{proof}
Using Proposition \ref{ss:spectralcohomologyformula}, we can assume that $Z$ is nonempty. Then Theorem \ref{ss:LStorus} and Proposition \ref{ss:boundedramification} indicate that $\LS_{\prescript{L}{}T,U}$ is naturally a gerbe over
  \begin{align*}
    \ul{\Hom}(\bT/\bO_U^\times,\bG_m)\coloneqq\bigcup_K\ul{\Hom}\big(\bT/\bO_U^\times\!\Nm_{\wt{F}/F}(K\otimes_\bZ X_*(T)),\bG_m\big)
  \end{align*}
  banded by $\wh{T}^{\Gal(\wt{F}/F)}$, where $K$ runs over $\Gal(\wt{F}/F)$-stable pro-$p$ open subgroups of $\wt{F}_{\wt{Z}}^\times$. This implies that $\om_{\LS_{\prescript{L}{}T,U}}$ is isomorphic to $\sO_{\LS_{\prescript{L}{}T,U}}$. Using \cite[Proposition 4.3.2]{Zou24}, this also implies that we have a decomposition
  \begin{align*}
    D_{\perf}^{\qc}(\LS_{\prescript{L}{}T,U}) = \bigoplus_{\chi\in X_*(T)_{W_F}}D_{\perf}^{\qc}(\LS_{\prescript{L}{}T,U})_\chi
  \end{align*}
  such that $D_{\perf}^{\qc}(\LS_{\prescript{L}{}T,U})_{\mathbf1}$ is naturally equivalent to $D_{\perf}^{\qc}(\ul{\Hom}(\bT/\bO_U^\times,\bG_m))$, the latter of which \cite[Lemma 4.2.1]{Zou24} identifies with $D(*/\ul{\bT/\bO_U^\times},\La)^\om$. For all $\chi$ in $X_*(T)_{W_F}$, tensoring with the line bundle associated with $\chi$ also yields an equivalence $D_{\perf}^{\qc}(\LS_{\prescript{L}{}T,U})_{\mathbf1}\ra^\sim D_{\perf}^{\qc}(\LS_{\prescript{L}{}T,U})_\chi$.

For all $z$ in $Z$, \cite[Lemma 6.2.3]{Zou24} and \cite[Lemma 4.2.1]{Zou24} indicate that we have an analogous decomposition
\begin{align*}
D_{\perf}^{\qc}(\LS_{\prescript{L}{}T,F_z}) = \bigoplus_{\chi_z\in X_*(T)_{W_{F_z}}}D_{\perf}^{\qc}(\LS_{\prescript{L}{}T,F_z})_{\chi_z},
\end{align*}
where $D_{\perf}^{\qc}(\LS_{\prescript{L}{}T,F_z})_{\chi_z}$ is naturally equivalent to $D(*/\ul{T(F_z)},\La)^\om$ for all $\chi_z$ in $X_*(T)_{W_{F_z}}$. This induces a decomposition
\begin{align}\label{eq:LST}
\bigotimes_{z\in Z}\Ind D^{\qc}_{\perf}(\LS_{\prescript{L}{}T,F_z}) = \prod_{\chi_\bullet\in\bigoplus_{z\in Z}X_*(T)_{W_{F_z}}}D(*/\ul{T(F_Z)},\La).\tag{$\clubsuit$}
\end{align}
Under these identifications, $\lres_{Z,*}(\om_{\prescript{L}{}T,U})$ corresponds to $C_c^\infty(\bT/\bO_U^\times,\La)$ in the factors where $\chi_\bullet$ lies in $\ker(\bigoplus_{z\in Z}X_*(T)_{W_{F_z}}\ra X_*(T)_{W_F})$ and $0$ in the other factors.

Since $T$ is a torus, every element in the Kottwitz set is basic. Therefore $\Bun_{T,U}$ equals $\Bun_{T,U}^{\semis}$, so Theorem \ref{ss:basiclocus} and Lemma \ref{ss:BunGpullback} naturally identify $\Bun_{T,U}$ with
\begin{align*}
\coprod_b\ul{T(F)}\bs\ul{T(\bA)}/\ul{T(\bO_U)T(F_Z)},
\end{align*}
where $b$ runs over elements in $B(F,T)$ satisfying $\loc_u(b)=1$ for all closed points $u$ of $U$. Hence, for all $b_\bullet$ in $\prod_{z\in Z}B(F_z,T)$, the corresponding factor of $\loc_{Z,!}\!\La$ in
\begin{align}\label{eq:BunT}
\bigotimes_{z\in Z}D(\Bun_{T,F_z},\La) = \prod_{b_\bullet\in\prod_{z\in Z}B(F_z,T)}D(*/\ul{T(F_Z)},\La)\tag{$\spadesuit$}
\end{align}
equals $\bigoplus_{b'}C_c^\infty(T(F)\bs T(\bA)/T(\bO_U),\La)$, where $b'$ runs over elements in $B(F,T)$ satisfying $\loc_z(b)=b_z$ for all $z$ in $Z$ and $\loc_u(b)=1$ for all closed points $u$ of $U$. By \cite[Proposition 15.6]{Kot14}, the set of such $b'$ is nonempty if and only if $(\ka(b_z))_z$ lies in $\ker(\bigoplus_{z\in Z}X_*(T)_{W_{F_z}}\ra X_*(T)_{W_F})$, and when this occurs, \cite[p.~83]{Kot14} shows that this set is a $\ker^1(F,T)$-torsor. Using \ref{ss:toruscohomology}, this implies that the images of $\loc_{Z,!}\!\La$ in (\ref{eq:BunT}) and $\lres_{Z,*}(\om_{\prescript{L}{}T,U})$ in (\ref{eq:LST}) are isomorphic. Finally, \cite[Remark 6.4.6]{Zou24} indicates that $\bigotimes_{z\in Z}\bL_{\psi_z}$ is given precisely by composing the equivalences (\ref{eq:BunT}) and (\ref{eq:LST}).
\end{proof}

\appendix

\section{GAGA over sousperfectoid spaces}\label{s:GAGA}
Our goal in this section is to prove Theorem H. Actually, we prove a generalization of Theorem H over any complete Tate Huber pair, though stating it requires the analytic geometry of Clausen--Scholze \cite{Sch20}; see Theorem \ref{ss:pcGAGA}. While we were finalizing this paper, a generalization of Theorem \ref{ss:pcGAGA} was proved independently by Wang \cite[Theorem 4.4.4]{Wan26} along very similar lines.

The proof of Theorem \ref{ss:pcGAGA} closely follows Clausen--Scholze's proof of GAGA for complex analytic spaces \cite[Theorem 13.10]{CS22}. Namely, given an abstract formalism of ``analytic loci'', first we use discrete Huber pairs to prove a GAGA theorem for \emph{all} solid quasi-coherent sheaves. Next, we prove the analogue of Grauert's coherence theorem in our setting. Finally, we use this to prove that the previous GAGA theorem restricts to an equivalence on pseudocoherent objects.

We conclude by applying Theorem H to prove that the analytification of algebraic stacks of bundles agrees with analytic stacks of bundles, answering a question of Heuer--Xu \cite[Remark 8.1.2]{HX24}. This was also proved independently by Wang \cite{Wan26}.

\subsection*{Notation}
For any (animated) condensed ring $\cA$, write $D_{\cond}(\cA)$ for its derived category of condensed $\cA$-modules, and recall that there is a natural fully faithful functor $D(\cA(*))\hra D_{\cond}(\cA)$.

\subsection{}\label{ss:anringpushout}
We start with some notation on \emph{analytic rings}. Write $\AnRing$ for the $\infty$-category of (normalized animated) analytic rings as in \cite[Definition 12.11]{Sch20}, and for any analytic ring $(\cA,\cM)$, write $D(\cA,\cM)\subseteq D_{\cond}(\cA)$ for the full subcategory of complete objects as in \cite[Remark 12.5]{Sch20}. Recall that $D(\cA,\cM)\subseteq D_{\cond}(\cA)$ determines the analytic ring structure $(\cA,\cM)$ on the underlying condensed ring $\cA$.

Recall that $(\cA,\cM)\mapsto D(\cA,\cM)$ yields a functor $D(-):\AnRing\ra\Sym$. 
\begin{lem*}
  The functor $D(-):\AnRing\ra\Sym$ preserves pushouts.
\end{lem*}
\begin{proof}
  Consider a pushout square in $\AnRing$
  \begin{align*}
    \xymatrix{(\cA,\cM_\cA)\ar[r]\ar[d] & (\cB,\cM_\cB)\ar[d]\\
    (\cC,\cM_\cC)\ar[r] & (\cD,\cM_\cD).}
  \end{align*}
Then \cite[Proposition 12.12]{Sch20} shows that we have a pullback square
  \begin{align*}
    \xymatrix{D(\cA,\cM_\cA) & \ar[l]D(\cB,\cM_B)\\
    D(\cC,\cM_\cC)\ar[u] & \ar[l]\ar[u]D(\cD,\cM_\cD),}
  \end{align*}
  so the result follows from passing to adjoints.
\end{proof}

\subsection{}
Next, let us recall the theory of quasi-coherent modules for analytic rings. For any $(\cA,\cM)$ in $\AnRing$, write $\AnSpec(\cA,\cM)$ for the corresponding object of $\AnRing^{\op}$. Endow $\AnRing^{\op}$ with the topology generated by finite families of jointly conservative steady localizations, as in \cite[p.~95]{Sch20}. Consider the following classes of morphisms $\AnSpec(\cB,\cN)\ra\AnSpec(\cA,\cM)$ in $\AnRing^{\op}$:
\begin{enumerate}[(A)]
\item[($I$)] the functor $D(\cA,\cM)\ra D(\cB,\cN)$ is an open embedding as in \cite[p.~61]{CS22},
\item[($P$)] $(\cB,\cN)$ is induced from $(\cA,\cM)$ along $\cA\ra\cB$ as in \cite[Proposition 12.8]{Sch20}.
\end{enumerate}
Write $E$ for the class of morphisms in $\AnRing^{\op}$ of the form $p\circ i$ for $p$ in $P$ and $i$ in $I$. Then the proof of \cite[Lemma 3.2.5]{RC24}\footnote{While \cite[Lemma 3.2.5]{RC24} restricts to the full subcategory of $\AnRing^{\op}$ consisting of solid affinoids as in \cite[Definition 2.6.6]{RC24}, this hypothesis is not used in the proof of \cite[Lemma 3.2.5]{RC24}.} shows that
\begin{enumerate}[$\bullet$]
\item $I$ and $P$ form a suitable decomposition of $E$ as in \cite[Definition A.5.9]{Man22},
\item the functor $D(-):\AnRing\ra\Sym$ satisfies the conditions in \cite[Proposition  A.5.10]{HM24} with respect to $I$ and $P$,
\end{enumerate}
so \cite[Proposition A.5.10]{HM24} endows $\AnSpec(\cA,\cM)\mapsto D(\cA,\cM)$ with the structure of a 6-functor formalism on $\AnRing^{\op}$ with $E$ as its class of $!$-able morphisms. By \cite[Proposition 12.18]{Sch20}, this $6$-functor formalism is sheafy as in \cite[Definition 3.4.1]{HM24}.

Write $\{\mbox{analytic stacks}\}$ for the $\infty$-category of sheaves of anima on $\AnRing^{\op}$. Then \cite[Theorem 3.4.11]{HM24} extends the above $6$-functor formalism on $\AnRing^{\op}$ to a sheafy $6$-functor formalism on $\{\mbox{analytic stacks}\}$. Note that the $\infty$-category of analytic spaces as in \cite[Definition 13.5]{Sch20} is a full subcategory of $\{\mbox{analytic stacks}\}$.

\subsection{}\label{ss:AnRinginducednuclear}
\emph{Nuclear} modules enjoy the following compatibility with induced analytic ring structures. Let $(\cA,\cM)$ be an analytic ring, let $\cB$ be an algebra in $D(\cA,\cM)$, and write $(\cB,\cN)$ for the analytic ring induced from $(\cA,\cM)$ along $\cA\ra\cB$ as in \cite[Proposition 12.8]{Sch20}. Then Lemma \ref{ss:anringpushout} implies that $D(\cB,\cN)$ is naturally equivalent to $\Mod_\cB(D(\cA,\cM))$, so the image of $-\otimes_{(\cB,\cN)}-$ in $D(\cA,\cM)$ is given by $-\otimes_\cB-$.
\begin{lem*}
  Assume that $\cB$ in $D(\cA,\cM)$ is nuclear as in \cite[Definition 8.5 (1)]{CS22}. Then an object $C$ of $D(\cB,\cN)$ is nuclear if and only if its image in $D(\cA,\cM)$ is nuclear.
\end{lem*}
\begin{proof}
  Recall that $C$ in $D(\cB,\cN)$ is nuclear if and only if, for all extremally disconnected $S$, the natural morphism
  \begin{align*}
    (\ul{\Hom}_\cB(\cN[S],\cB)\otimes_\cB C)(*)\ra\ul{\Hom}_\cB(\cN[S],C)(*)
  \end{align*}
  is an isomorphism. Adjunction shows that $\ul{\Hom}_\cB(\cN[S],-)=\ul{\Hom}_\cA(\cM[S],-)$, and because $\cB$ in $D(\cA,\cM)$ is nuclear, \cite[Proposition 5.35]{And21}  shows that $\ul{\Hom}_\cA(\cM[S],\cB)$ is naturally isomorphic to $\ul{\Hom}_\cA(\cM[S],\cA)\otimes_{(\cA,\cM)}\cB$. Therefore the above becomes
  \begin{align*}
    (\ul{\Hom}_\cA(\cM[S],\cA)\otimes_{(\cA,\cM)}C)(*)\ra\ul{\Hom}_\cA(\cM[S],C)(*),
  \end{align*}
and this is an isomorphism for all extremally disconnected $S$ if and only if the image of $C$ in $D(\cA,\cM)$ is nuclear.
\end{proof}

\subsection{}\label{ss:AnRingnucleargluing}
When the analytic rings are all nuclear over a base, we can check nuclearity on a cover as follows. Let $(\cA,\cM)$ be an analytic ring, let $\cB$ be a nuclear algebra in $D(\cA,\cM)$, and write $(\cB,\cN)$ for the analytic ring induced from $(\cA,\cM)$ along $\cA\ra\cB$ as in \cite[Proposition 12.8]{Sch20}.

Let $\{\cB_j\}_{j=1}^r$ be a finite family of nuclear algebras in $D(\cB,\cN)$, and write $(\cB_j,\cN_j)$ for the analytic ring induced from $(\cB,\cN)$ along $\cB\ra\cB_j$ as in \cite[Proposition 12.8]{Sch20} for all $1\leq j\leq r$. Assume that $\big\{(\cB,\cN)\ra(\cB_j,\cN_j)\big\}_{j=1}^r$ is a family of jointly conservative steady localizations.
\begin{prop*}
An object $C$ of $D(\cB,\cN)$ is nuclear if and only if $C\otimes_\cB\cB_j$ in $D(\cB_j,\cN_j)$ is nuclear for all $1\leq j\leq r$.
\end{prop*}
\begin{proof}
  Since the functor $D(\cB,\cN)\ra D(\cB_j,\cN_j)$ is symmetric monoidal, it preserves trace-class morphisms as in \cite[Definition 8.1]{CS22}. It also preserves colimits, so it preserves basic nuclear objects as in \cite[Definition 8.5 (2)]{CS22} and hence nuclear objects by \cite[Theorem 8.6 (2)]{CS22}.

  Conversely, assume that $C\otimes_\cB\cB_j$ in $D(\cB_j,\cN_j)$ is nuclear for all $1\leq j\leq r$. Then Lemma \ref{ss:AnRinginducednuclear} indicates that the image of $C\otimes_\cB\cB_j$ in $D(\cB,\cN)$ is nuclear. Descent shows that $C$ is naturally isomorphic to a finite limit with terms of the form $C\otimes_\cB\bigotimes_{j\in J}\cB_j$ for nonempty subsets $J$ of $\{1,\dotsc,r\}$, and because the $C\otimes_\cB\bigotimes_{j\in J}\cB_j$ are nuclear, \cite[Theorem 8.6 (1)]{CS22} implies that $C$ is also nuclear.
\end{proof}

\subsection{}\label{ss:AnRingpseudocompact}
In the following situation, we can check \emph{pseudocompactness} on a cover. Let $(\cA,\cM)$ be an analytic ring. Let $\big\{(\cA,\cM)\ra(\cA_j,\cM_j)\big\}_{j=1}^r$ be a finite family of jointly conservative steady localizations such that, for all $1\leq j\leq r$, the functor $D(\cA,\cM)\ra D(\cA_j,\cM_j)$ is the open embedding as in \cite[p.~61]{CS22} corresponding to an idempotent algebra $\cI_j$ in $D^{\leq0}(\cA,\cM)$.
\begin{prop*}
An object $C$ of $D(\cA,\cM)$ is pseudocompact as in \cite[Definition 9.9]{CS22} if and only if its image in $D(\cA_j,\cM_j)$ is pseudocompact for all $1\leq j\leq r$.
\end{prop*}
\begin{proof}
  Because $D(\cA_j,\cM_j)\ra D(\cA,\cM)$ is left $t$-exact and preserves direct sums, if $C$ is pseudocompact, then its image in $D(\cA_j,\cM_j)$ is pseudocompact.

  Conversely, assume that the image of $C$ in $D(\cA_j,\cM_j)$ is pseudocompact for all $1\leq j\leq r$. Since the composition $D(\cA,\cM)\ra D(\cA_j,\cM_j)\ra D(\cA,\cM)$ is given by $\ul{\Hom}_\cA(\fib(\cA\ra\cI_j),-)$ and $\cI_j$ lies in $D^{\leq0}(\cA,\cM)$, we see that $D(\cA,\cM)\ra D(\cA_j,\cM_j)$ sends $D^{\leq0}(\cA,\cM)$ to $D^{\leq1}(\cA_j,\cM_j)$. Descent shows that $\Hom_\cA(C,-)$ is naturally isomorphic to a finite limit with terms of the form
  \begin{align*}
\Hom_{\bigotimes_{j\in J}(\cA_j,\cM_j)}(C\otimes_{(\cA,\cM)}\textstyle\bigotimes_{j\in J}(\cA_j,\cM_j),-\otimes_{(\cA,\cM)}\bigotimes_{j\in J}(\cA_j,\cM_j))
  \end{align*}
  for nonempty subsets $J$ of $\{1,\dotsc,r\}$. The above shows that $-\otimes_{(\cA,\cM)}\bigotimes_{j\in J}(\cA_j,\cM_j)$ sends $D^{\leq0}(\cA,\cM)$ to $D^{\leq(\# J)}\big(\!\bigotimes_{j\in J}(\cA_j,\cM_j)\big)$. Because $C\otimes_{(\cA,\cM)}\bigotimes_{j\in J}(\cA_j,\cM_j)$ is pseudocompact, this implies that $C$ is also pseudocompact.
\end{proof}

\subsection{}\label{ss:nonarchcomputationsZ}
Next, let us recall the analytic rings associated with complete Huber pairs. For any complete Huber pair $(A,A^+)$, write $(A,A^+)_\solid$ for the associated analytic ring as in \cite[Proposition 13.16]{Sch20}, write $(A,A^+)_\solid[-]$ for its functor of measures, and write $D_\solid(A,A^+)$ for $D((A,A^+)_\solid)$. When $A^+$ equals $A^\circ$, we omit it from our notation. By \cite[Proposition 3.29]{And21}, we see that $D_\solid(A,A^+)$ is naturally equivalent to $\Mod_A(D_\solid(A^+))$.
\begin{lem*}
We have the following pushout squares in $\AnRing$:
  \begin{gather*}
        \xymatrix{(\bZ[X,Y],\bZ)_\solid\ar[r]\ar[d] & (\bZ[X,Y],\bZ[X])_\solid\ar[d] \\
    (\bZ[X,Y],\bZ[Y])_\solid\ar[r] & \bZ[X,Y]_\solid,} \\
    \xymatrix{(\bZ[XY],\bZ)_\solid\ar[r]\ar[d] & \bZ[XY]_\solid\ar[d] \\
   (\bZ[X,Y],\bZ)_\solid\ar[r] & (\bZ[X,Y],\bZ[XY])_\solid,}     \xymatrix{(\bZ[X+Y],\bZ)_\solid\ar[r]\ar[d] & \bZ[X+Y]_\solid\ar[d] \\
    (\bZ[X,Y],\bZ)_\solid\ar[r] & (\bZ[X,Y],\bZ[X+Y])_\solid.}
  \end{gather*}
\end{lem*}
\begin{proof}
This follows from combining \cite[Proposition 12.12]{Sch20} with the characterization of $(-,-)_\solid$ from \cite[Proposition 3.32]{And21}.
\end{proof}

\subsection{}\label{ss:Z[T]facts}
The following analytic rings correspond to the complements of the closed\footnote{``Closed'' in the sense of being defined by a non-strict inequality; our closed unit disks are actually open in the affine line, so their complements in the affine line are actually closed.} unit disk centered at the origin and at infinity. Recall that $\bZ\lp{T^{-1}}$ is an idempotent algebra in $D_\solid(\bZ[T],\bZ)$ \cite[Observation 8.7]{Sch19}, and the corresponding open embedding as in \cite[p.~61]{CS22} is $D_\solid(\bZ[T],\bZ)\ra D_\solid(\bZ[T])$ \cite[Observation 8.10]{Sch19}. Also, recall that the complex
\begin{align*}
\xymatrixcolsep{1.5cm}
\xymatrix{\bZ\lb{U}\otimes_\bZ\bZ[T]\ar[r]^-{UT-1}& \bZ\lb{U}\otimes_\bZ\bZ[T]}
\end{align*}
is a projective resolution of $\bZ\lp{T^{-1}}$ in $D_\solid(\bZ[T],\bZ)$ \cite[Observation 8.6]{Sch19}.

Recall that $\bZ\lb{U}$ is naturally isomorphic to  $\bZ_\solid[\bN\cup\{\infty\}]/\bZ[\{\infty\}]$, where multiplication by $U$ is induced by the endomorphism $n\mapsto n+1$ of $\bN$ \cite[Lemma 3.11]{And21}.
\begin{lem*}
  $\bZ\lb{T}$ is also an idempotent algebra in $D_\solid(\bZ[T],\bZ)$, and we have
  \begin{align*}
    \bZ\lp{T^{-1}}\otimes_{(\bZ[T],\bZ)_\solid}\bZ\lb{T}=0.
  \end{align*}
\end{lem*}
\begin{proof}
By \cite[Example 6.4]{Sch19}, we obtain
\begin{align*}
  \bZ\lb{T}\otimes_{(\bZ[T]_\solid,\bZ)}\bZ\lb{T} = (\bZ\lb{T_1}\otimes_{\bZ_\solid}\bZ\lb{T_2})/(T_1-T_2) = \bZ\lb{T_1,T_2}/(T_1-T_2) = \bZ\lb{T}.
\end{align*}
Next, using the above projective resolution of $\bZ\lp{T^{-1}}$, we see that
\begin{align*}
\bZ\lp{T^{-1}}\otimes_{(\bZ[T],\bZ)_\solid}\bZ\lb{T}
\end{align*}
is the solidification of the complex
\begin{align*}
\xymatrixcolsep{1.5cm}
\xymatrix{\bZ\lb{U}\otimes_\bZ\bZ\lb{T}\ar[r]^-{UT-1} & \bZ\lb{U}\otimes_\bZ\bZ\lb{T}.}
\end{align*}
After solidification, \cite[Example 6.4]{Sch19} indicates that both terms become $\bZ\lb{U,T}$. Because $UT-1$ is invertible in $\bZ\lb{U,T}$, the resulting complex is indeed exact.
\end{proof}

\subsection{}\label{ss:closedunitdisk}
Let $(R,R^+)$ be a complete Huber pair.
\begin{lem*}We have a pushout square in $\AnRing$
\begin{align*}
    \xymatrix{\bZ_\solid\ar[r]\ar[d] & \bZ[T]_\solid\ar[d] \\
   (R,R^+)_\solid\ar[r] & (R\ang{T},R^+\ang{T})_\solid.}
\end{align*}
\end{lem*}
\begin{proof}
Write $(\cA,\cM)$ for the pushout. Since \cite[Lemma 3.8]{And21} indicates that $\bZ_\solid\ra\bZ[T]_\solid$ is steady, we see that
  \begin{align*}
    \cA &= R\otimes_{\bZ_\solid}\bZ[T]_\solid & \mbox{by \cite[p.~84]{Sch20},}\\
    &= R\ang{T} & \mbox{by \cite[Proposition 3.14]{And21}.}
  \end{align*}
  Next, \cite[Proposition 12.12]{Sch20} indicates that $D(\cA,\cM)\subseteq D_{\cond}(R\ang{T})$ is the full subcategory of objects whose images in $D_{\cond}(\bZ[T])$ and $D_{\cond}(R)$ lie in $D_\solid(\bZ[T])$ and $D_\solid(R,R^+)$, respectively. Because $R^+\ang{T}$ is topologically generated by $\bZ[T]$ and $R^+$, \cite[Proposition 3.32]{And21} implies that $D(\cA,\cM)$ equals $D_\solid(R\ang{T},R^+\ang{T})$.
\end{proof}

\subsection{}
One can use analytic rings to define the following version of analytification; remarkably, it still remembers algebraic functions. For any $R$-algebra $A$, write $A^{\alg}$ for the pushout in $\AnRing$
\begin{align*}
  \xymatrix{(R_{\disc},R^+_{\disc})_{\solid}\ar[r]\ar[d] & (A_{\disc},(R^+)_{\disc}^\sim)_{\solid}\ar[d]\\
  (R,R^+)_\solid\ar[r] & A^{\alg}.}
\end{align*}
Since $A_{\disc}$ is nuclear in $D(R_{\disc},R^+_{\disc})$, \cite[Proposition 13.14]{Sch20} shows that the top arrow is steady. Therefore underlying condensed ring of $A^{\alg}$ is $A$, by \cite[p.~84]{Sch20}.

For any $f$ in $A$, write $A_{\abs{f}\leq1}$ and $A_{\abs{f}\geq1}$ for the following pushouts in $\AnRing$:
\begin{gather*}
  \xymatrix{(\bZ[T],\bZ)_\solid\ar[r]\ar[d]^-f & \bZ[T]_\solid\ar[d]\\
  A^{\alg}\ar[r]& A_{\abs{f}\leq1},} \quad\quad  \xymatrix{(\bZ[T],\bZ)_\solid\ar[r]\ar[d]^-f & (\bZ[T^{\pm1}],\bZ[T^{-1}])_\solid\ar[d]\\
  A^{\alg}\ar[r]& A_{\abs{f}\geq1}.}
\end{gather*}
Because the top arrows are steady localizations, the same holds for the bottom arrows.

\subsection{}\label{ss:closureofclosedunitdisk}
The following analytic ring corresponds to the closure of the closed unit disk centered at the origin. For the rest of this section, assume that $R$ is Tate.
\begin{prop*}
The object $R\ang{T}$ in $D_\solid(R,R^+)$ is nuclear as in \cite[Definition 8.5 (1)]{CS22}. Moreover, the analytic ring induced from $(R,R^+)_\solid$ along $R\ra R\ang{T}$ as in \cite[Proposition 12.8]{Sch20} is naturally isomorphic to $(R\ang{T},R^++T\cdot R^{\circ\circ}\ang{T})_\solid$, and the induced morphism $R[T]^{\alg}\ra(R\ang{T},R^++T\cdot R^{\circ\circ}\ang{T})_\solid$ is a steady localization.
\end{prop*}
\begin{proof}
  Since $R$ is Tate, the first statement follows from \cite[Lemma 4.4]{And23}. For the second statement, write $(\cA,\cM)$ for the analytic ring induced form $(R,R^+)_\solid$ along $R\ra R\ang{T}$ as in \cite[Proposition 12.8]{Sch20}. Because $R\ang{T}$ lies in $D_\solid(R,R^+)$, we see that $\cA=R\ang{T}$.  Next, \cite[Proposition 12.12]{Sch20} indicates that $D(\cA,\cM)\subseteq D_{\cond}(R\ang{T})$ is the full subcategory of objects whose image in $D_{\cond}(R)$ lies in $D_\solid(R,R^+)$. Since $R^++T\cdot R^{\circ\circ}\ang{T}$ is topologically generated by $R^+$, \cite[Proposition 3.32]{And21} implies that $D(\cA,\cM)$ equals $D_\solid(R\ang{T},R^++T\cdot R^{\circ\circ}\ang{T})$, as desired.

  For the third statement, \cite[Proposition 13.14]{Sch20} and the first statement imply
  \begin{align*}
    (R,R^+)_\solid\ra(R\ang{T},R^++T\cdot R^{\circ\circ}\ang{T})_\solid
  \end{align*}
  is steady, so \cite[Proposition 12.15]{Sch20} shows that 
  \begin{align*}
    R[T]^{\alg}\ra(R\ang{T},R^++T\cdot R^{\circ\circ}\ang{T})_\solid
  \end{align*}
  is also steady. Finally, Lemma \ref{ss:closedunitdisk} implies that the composition
    \begin{align*}
    R[T]^{\alg}\ra(R\ang{T},R^++T\cdot R^{\circ\circ}\ang{T})_\solid\ra(R\ang{T},R^+\ang{T})_\solid
    \end{align*}
    is a localization. Because the right arrow is a localization, the same holds for the left arrow.
\end{proof}

\subsection{}\label{ss:nuclearcover}
For any $R$-algebra $A$ and $f$ in $A$, write $\ov{A}_{\abs{f}\leq1}$ for the pushout in $\AnRing$
\begin{align*}
  \xymatrix{R[T]^{\alg}\ar[r]\ar[d]^-f & (R\ang{T},R^++T\cdot R^{\circ\circ}\ang{T})_\solid\ar[d]\\
  A^{\alg}\ar[r] & \ov{A}_{\abs{f}\leq1}.}
\end{align*}
The following analytic spaces correspond the loci in $\AnSpec A^{\alg}$ where $f$ is analytic and where $f$ is analytically invertible, respectively. Choose a pseudouniformizer $\vpi$ of $R$, form the colimits in $\{\mbox{analytic stacks}\}$
\begin{align*}
\an(A,f)\coloneqq\varinjlim_{n\geq0}\AnSpec A_{\abs{\vpi^nf}\leq1}\mbox{ and } D^{\an}(f)\coloneqq\varinjlim_{n\geq0}\AnSpec A_{\abs{\vpi^{-n}f}\geq1},
\end{align*}
and note that they do not depend on the choice of $\vpi$. Moreover, the factorization
\begin{align*}
(R\ang{\vpi T},R^+\ang{\vpi T})_\solid\ra(R\ang{T},R^++T\cdot R^{\circ\circ}\ang{T})_\solid\ra(R\ang{T},R^+\ang{T})_\solid
\end{align*}
indicates that $\an(A,f)$ is naturally isomorphic to $\varinjlim_{n\geq0}\AnSpec\ov{A}_{\abs{\vpi^nf}\leq1}$.

\subsection{}\label{ss:GAGAaxioms}
The analytic loci satisfy the following basic properties.
\begin{prop*}\hfill
  \begin{enumerate}[1)]
  \item For any $r$ in $R$, we have $\an(R,r)=\AnSpec R^{\alg}$.
  \item Write $\an(R[X,Y],X)\cap\an(R[X,Y],Y)$ for the fiber product in $\{\mbox{analytic stacks}\}$
    \begin{align*}
      \xymatrix{\an(R[X,Y],X)\cap\an(R[X,Y],Y)\ar[r]\ar[d] & \an(R[X,Y],X)\ar[d]\\
      \an(R[X,Y],Y)\ar[r] & \AnSpec R[X,Y]^{\alg}.}
    \end{align*}
    Then $\an(R[X,Y],X)\cap\an(R[X,Y],Y)\ra\AnSpec R[X,Y]^{\alg}$ factors through
    \begin{align*}
      \an(R[X,Y],XY)\ra\AnSpec R[X,Y]^{\alg}\mbox{ and }\an(R[X,Y],X+Y)\ra\AnSpec R[X,Y]^{\alg}.
    \end{align*}
  \item The composition $\an(R[T^{\pm1}],T^{-1})\ra\AnSpec R[T^{\pm1}]^{\alg}\ra\AnSpec R[T]^{\alg}$ is naturally isomorphic to $D^{\an}(T)\ra\AnSpec R[T]^{\alg}$ over $\AnSpec R[T]^{\alg}$.
  \item $\big\{\an(R[T],T),D^{\an}(T)\big\}$ forms a cohomologically \'etale (as in \cite[Definition 6.12]{Sch22}) steady localization cover of $\AnSpec R[T]^{\alg}$.
  \item Write $D^{\an}(X)\cup D^{\an}(Y)$ for the pushout in $\{\mbox{analytic stacks}\}$
    \begin{align*}
      \xymatrix{D^{\an}(X)\cap D^{\an}(Y)\ar[r]\ar[d]& D^{\an}(X)\ar[d]\\
      D^{\an}(Y)\ar[r] & D^{\an}(X)\cup D^{\an}(Y).}
    \end{align*}
    Then $D^{\an}(X+Y)\ra\AnSpec R[X,Y]^{\alg}$ factors through
    \begin{align*}
      D^{\an}(X)\cup D^{\an}(Y)\ra\AnSpec R[X,Y]^{\alg}.
    \end{align*}
  \end{enumerate}
\end{prop*}
\begin{proof}
For part 1), it suffices to show that $R_{\abs{r}\leq1}=R^{\alg}$ when $r$ lies in $R^{\circ\circ}$. By \ref{ss:Z[T]facts}, the functor $D(R[T]^{\alg})\ra D(R[T]_{\abs{T}\leq1})$ is the open embedding as in \cite[p.~61]{CS22} corresponding to the idempotent algebra $\bZ\lp{T^{-1}}\otimes_{\bZ_\solid}(R,R^+)_\solid$ in $D(R[T]^{\alg})$. The descriptions of $\bZ\lp{T}$ and $\bZ\lb{U}$ from \ref{ss:Z[T]facts} show that $\bZ\lp{T^{-1}}\otimes_{\bZ_\solid}(R,R^+)_\solid$ is given by the complex
\begin{align*}
\xymatrixcolsep{1.5cm}
\xymatrix{\displaystyle\frac{(R,R^+)_\solid[\bN\cup\{\infty\}]}{R[\{\infty\}]}\otimes_RR[T]\ar[r]^-{UT-1}& \displaystyle\frac{(R,R^+)_\solid[\bN\cup\{\infty\}]}{R[\{\infty\}]}\otimes_RR[T].}
\end{align*}
Since $r$ lies in $R^{\circ\circ}$, multiplication by $Ur-1$ is invertible on $(R,R^+)_\solid[\bN\cup\{\infty\}]/R[\{\infty\}]$. Therefore, after base changing the above complex along the $R$-algebra morphism $R[T]\mapsto R$ that sends $T\mapsto r$, the resulting complex is indeed exact.

For part 2), it suffices to show that
\begin{align*}
R[X,Y]^{\alg}\ra R[X,Y]_{\abs{X}\leq1}\otimes_{R[X,Y]^{\alg}}R[X,Y]_{\abs{Y}\leq1}
\end{align*}
factors through $R[X,Y]^{\alg}\ra R[X,Y]_{\abs{XY}\leq1}$ and $R[X,Y]^{\alg}\ra R[X,Y]_{\abs{X+Y}\leq1}$. This follows from applying $-\otimes_{\bZ_\solid}(R,R^+)_\solid$ to Lemma \ref{ss:nonarchcomputationsZ} and the morphisms
\begin{align*}
(\bZ[X,Y],\bZ[XY])_\solid\ra\bZ[X,Y]_\solid\mbox{ and }(\bZ[X,Y],\bZ[X+Y])_\solid\ra\bZ[X,Y]_\solid.
\end{align*}

For part 3), note that $T$ is invertible in $(\bZ[T^{\pm1}],\bZ[T^{-1}])_\solid$, and $T\mapsto T^{-1}$ induces an isomorphism $(\bZ[T^{\pm1}],\bZ[T])_\solid\ra^\sim(\bZ[T^{\pm1}],\bZ[T^{-1}])_\solid$. Applying $-\otimes_{\bZ_\solid}(R,R^+)_\solid$ to this shows that $R[T]^{\alg}\ra R[T]_{\abs{T}\geq1}$ factors through $R[T]^{\alg}\ra R[T^{\pm1}]^{\alg}$ and that the analytic rings $R[T]_{\abs{T}\geq1}$ and $R[T^{\pm1}]_{\abs{T^{-1}}\leq1}$ are naturally isomorphic over $R[T^{\pm1}]^{\alg}$. This implies the desired result.

For part 4), \cite[Lemma 3.2.9]{RC24} indicates that it suffices to show that
\begin{align*}
\big\{R[T]_{\abs{T}\leq1},R[T]_{\abs{T}\geq1}\big\}
\end{align*}
forms a family of jointly conservative steady localizations of $R[T]^{\alg}$ such that, after applying $D(-)$, the resulting functors are open embeddings as in \cite[p.~61]{CS22}. This follows from applying $-\otimes_{\bZ_\solid}(R,R^+)_\solid$ to Lemma \ref{ss:Z[T]facts}.

For part 5), Lemma \ref{ss:Z[T]facts} shows that $\bZ\lb{T}$ is an idempotent algebra in $D_\solid(\bZ[T],\bZ)$, and \ref{ss:Z[T]facts} implies that $D_\solid(\bZ[T],\bZ)\ra D_\solid(\bZ[T^{\pm1},T^{-1}])$ is the open embedding as in \cite[p.~61]{CS22} corresponding to $\bZ\lb{T}$. Hence applying $-\otimes_{\bZ_\solid}(R,R^+)_\solid$ to the morphism $\bZ\lb{X+Y}\ra\bZ\lb{X,Y}$ of idempotent algebras in $D_{\solid}(\bZ[X,Y],\bZ)$ yields the desired result.
\end{proof}

\subsection{}\label{ss:AnStackGAGA}
Using discrete Huber pairs, one can prove the following GAGA theorem for \emph{all} (solid) quasi-coherent sheaves as in work of Clausen--Scholze \cite{CS22}. Let $A$ be a finitely generated $R$-algebra, and let $A^+\subseteq A$ be the integral closure of a finitely generated $R$-subalgebra. If $A^+$ equals the integral closure of the $R$-subalgebra generated by $f_1,\dotsc,f_m$, write $\an(A,A^+)$ for the intersection $\bigcap_{i=1}^m\an(A,f_i)$ in $\{\mbox{analytic stacks}\}$ over $\AnSpec A^{\alg}$. By Proposition \ref{ss:GAGAaxioms}.1), Proposition \ref{ss:GAGAaxioms}.2), and proof of \cite[Lemma 6.11]{CS22}, this does not depend on the choice of $f_1,\dotsc,f_m$.
\begin{thm*}
  The assignment $\Spa(A_{\disc},A^+_{\disc})\mapsto\an(A,A^+)$ sends rational covers to steady localization covers. Consequently, the assignments $\Spec{A}\mapsto\an(A,R^\sim)$ and $\Spec{A}\mapsto\an(A,A)$ naturally glue to functors
  \begin{align*}
    (-)^{\alg}\mbox{ and }(-)^{\an}:\{\mbox{schemes locally of finite type over }R\}\ra\{\mbox{analytic stacks}\},
  \end{align*}
respectively, along with a natural transformation $(-)^{\an}\ra(-)^{\alg}$. Moreover, $(-)^{\an}$ sends Zariski covers to cohomologically \'etale steady localization covers. Finally, for any proper scheme $X$ over $R$, the morphism $X^{\an}\ra X^{\alg}$ is an isomorphism.
\end{thm*}
\begin{proof}
  The first statement follows from the description of rational covers of
  \begin{align*}
    \Spa(A_{\disc},A^+_{\disc})
  \end{align*}
  from \cite[Lemma 6.12]{CS22} and Proposition \ref{ss:GAGAaxioms}. The second statement follows from the first statement and the fact that the assignments $\Spec{A}\mapsto\Spa(A_{\disc},R^\sim_{\disc})$ and $\Spec{A}\mapsto\Spa(A_{\disc},A_{\disc})$ send rational covers to rational covers. The third statement follows from Proposition \ref{ss:GAGAaxioms}.4). Finally, the fourth statement follows from the valuative criterion for properness.
\end{proof}
\begin{rem*}
By \cite[Proposition 13.6]{Sch20}, the functors $(-)^{\alg}$ and $(-)^{\an}$ are valued in the full subcategory of analytic spaces as in \cite[Definition 13.5]{Sch20}.
\end{rem*}

\subsection{}\label{ss:resolvable}
We will later need the notion of \emph{pseudocoherent} sheaves on schemes, as well as its relation with finite presentation. For any scheme $X$, write $D_{\pc}(X)$ for the full subcategory of $D_{\qc}(X)$ consisting of pseudocoherent objects as in \cite[Tag 08CB]{stacks-project}. When $X=\Spec{A}$ is affine, recall from \cite[Tag 08E7]{stacks-project} that this is equivalent to the full subcategory $D_{\pc}(A)\subseteq D(A)$ of pseudocoherent objects as in \cite[Tag 064Q]{stacks-project}.
\begin{lem*}Let $A$ be a finitely generated $R$-algebra. Then $A$ is a finitely presented $R$-algebra if and only if, for all $R$-algebra surjections $R[T_1,\dotsc,T_m]\twoheadrightarrow A$, the object $A$ lies in $D_{\pc}(R[T_1,\dotsc,T_m])$.
\end{lem*}
\begin{proof}
  Let $R[T_1,\dotsc,T_m]\twoheadrightarrow A$ be an $R$-algebra surjection with $A$ lying in
  \begin{align*}
  D_{\pc}(R[T_1,\dotsc,T_m]).
  \end{align*}
  Then \cite[Tag 064T]{stacks-project} implies that $A$ is a finitely presented $R$-algebra.

  Conversely, assume that $A$ is a finitely presented $R$-algebra, and let
  \begin{align*}
\al:R[T_1,\dotsc,T_m]\twoheadrightarrow A
  \end{align*}
  be an $R$-algebra surjection. By \cite[Tag 00R2]{stacks-project}, $\ker\al$ is finitely generated. Therefore noetherian approximation yields a noetherian ring $R_0$, a surjection
  \begin{align*}
   \al_0:R_0[T_1,\dotsc,T_m]\twoheadrightarrow A_0
  \end{align*}
  of $R_0$-algebras, and a ring homomorphism $R_0\ra R$ such that the base change of $(A_0,\al_0)$ to $R$ is isomorphic to $(A,\al)$. Because $R_0[T_1,\dotsc,T_m]$ is noetherian, $A_0$ lies in $D_{\pc}(R_0[T_1,\dotsc,T_m])$. Hence $A$ lies in $D_{\pc}(R[T_1,\dotsc,T_m])$.
\end{proof}

\subsection{}\label{ss:AnStackpc}
By construction, the analytifications from Theorem \ref{ss:AnStackGAGA} are covered by analytic rings of the following form. For any finitely generated $R$-algebra $A$ with generators $\ul{f}=(f_1,\dotsc,f_m)$, write $A_{\abs{\ul{f}}\leq1}$ and $\ov{A}_{\abs{\ul{f}}\leq1}$ for these pushouts in $\AnRing$:
\begin{gather*}
  \xymatrix{(\bZ[T_1,\dotsc,T_m],\bZ)_\solid\ar[r]\ar[d]^-{\ul{f}} & \bZ[T_1,\dotsc,T_m]_\solid\ar[d]\\
    A^{\alg}\ar[r] & A_{\abs{\ul{f}}\leq1},}\\
  \xymatrix{R[T_1,\dotsc,T_m]^{\alg}\ar[r]\ar[d]^-{\ul{f}} & (R\ang{T_1,\dotsc,T_m},R^++(T_1,\dotsc,T_m)\cdot R^{\circ\circ}\ang{T_1,\dotsc,T_m})_\solid\ar[d]\\
  A^{\alg}\ar[r] & \ov{A}_{\abs{\ul{f}}\leq1}.}
\end{gather*}

For any scheme $X$ locally of finite type over $R$, Theorem \ref{ss:AnStackGAGA} shows that $X^{\an}$ is glued from objects of the form $\AnSpec A_{\abs{\ul{f}}\leq1}$ along covers as in \ref{ss:AnRingpseudocompact}. Therefore Proposition \ref{ss:AnRingpseudocompact} indicates that the full subcategory of $D(A_{\abs{\ul{f}}\leq1})$ consisting of pseudocompact objects as in \cite[Definition 9.9]{CS22} naturally glues to a full subcategory of $D(X^{\an})$. We say that an object of $D(X^{\an})$ is \emph{pseudocompact} if it lies in this full subcategory.

Theorem \ref{ss:AnStackGAGA} and \ref{ss:nuclearcover} show that $X^{\an}$ is also glued from objects of the form $\AnSpec\ov{A}_{\abs{\ul{f}}\leq1}$ along covers as in \ref{ss:AnRingnucleargluing}. Hence Proposition \ref{ss:AnRingnucleargluing} indicates that the full subcategory of $D(\ov{A}_{\abs{f}\leq1})$ consisting of nuclear objects as in \cite[Definition 8.5 (1)]{CS22} naturally glues to a full subcategory of $D(X^{\an})$. We say that an object of $D(X^{\an})$ is \emph{nuclear} if it lies in this full subcategory.

Finally, write $D_{\pc}(X^{\an})$ for the full subcategory of $D(X^{\an})$ consisting of pseudocompact and nuclear objects. Recall from \cite[Theorem 5.50]{And21}\footnote{While \cite[Section 5]{And21} assumes that $(R,R^+)$ is sheafy, this hypothesis is not used in the proof of \cite[Theorem 5.50]{And21}.} that
\begin{align*}
D_{\pc}((\Spec{R})^{\an})\subseteq D((\Spec{R})^{\an}) = D_\solid(R,R^+)
\end{align*}
equals the image of $D_{\pc}(R)\subseteq D(R)\hra D_\solid(R,R^+)$.

\subsection{}\label{ss:uppershriek}
For the analytifications from Theorem \ref{ss:AnStackGAGA}, $!$-pullback of quasi-coherent sheaves commutes with arbitrary direct sums of uniformly bounded below objects. More precisely, for any scheme $X$ locally of finite type over $R$, write
\begin{align*}
\pi:X^{\an}\ra(\Spec R)^{\an}=\AnSpec(R,R^+)_\solid
\end{align*}
for $(-)^{\an}$ applied to the structure morphism $X\ra\Spec R$.
\begin{prop*}
  Assume that $X$ is of finite presentation over $R$, and let $d$ be an integer. Then there exists an integer $e$ such that, for all families $(M_\al)_\al$ of objects of $D_{\solid}^{\geq d}(R,R^+)$, the family $(\pi^!M_\al)_\al$ lies in $D^{\geq e}(X^{\an})$ and the natural morphism
  \begin{align*}
\bigoplus_\al\pi^!M_\al\ra\pi^!\big(\bigoplus_\al M_\al\big)
  \end{align*}
  is an isomorphism.
\end{prop*}
\begin{proof}
  Theorem \ref{ss:AnStackGAGA} indicates that $(-)^{\an}$ sends Zariski covers to cohomologically \'etale steady localization covers, so by replacing $X$ with an open cover, we can assume that $X=\Spec{A}$ for a finitely presented $R$-algebra $A$.

  Let $\ul{f}=(f_1,\dotsc,f_m)$ be generators of $A$. Then $X^{\an}$ is isomorphic to
  \begin{align*}
\varinjlim_{n\geq0}\AnSpec A_{\abs{\vpi^n\ul{f}}\leq1},
  \end{align*}
  where the transition morphisms are cohomologically \'etale, so it suffices to replace $X^{\an}$ with $\AnSpec A_{\abs{\ul{f}}\leq1}$. Write $\ul{T}=(T_1,\dotsc,T_m)$ for the standard generators of $R[T_1,\dotsc,T_m]$, and note we have a pushout square in $\AnRing$
  \begin{align*}
    \xymatrix{(R[T_1,\dotsc,T_m]_{\disc},R^+_{\disc})_\solid\ar[r]\ar[d]^-{\ul{f}} & R[T_1,\dotsc,T_m]_{\abs{\ul{T}}\leq1}\ar[d]^-{\ul{f}}\\
    (A_{\disc},(R^+)^\sim_{\disc})_\solid\ar[r] & A_{\abs{\ul{f}}\leq1}.}
  \end{align*}
  Now \cite[Proposition 3.22]{And21} and the characterization of $(-,-)_\solid$ from \cite[Proposition 3.32]{And21} imply that the left arrow is induced as in \cite[Proposition 12.8]{Sch20}, so the same holds for the right arrow. Therefore the functor
  \begin{align*}
(\AnSpec\ul{f})^!:D(R[T_1,\dotsc,T_m]_{\abs{\ul{T}}\leq1})\ra D(A_{\abs{\ul{f}}\leq1})
  \end{align*}
  is given by $\ul{\Hom}_{R\ang{T_1,\dotsc,T_m}}(R\ang{T_1,\dotsc,T_m}\otimes_{R[T_1,\dotsc,T_m]}A,-)$, where we identify
  \begin{align*}
    R[T_1,\dotsc,T_m]_{\abs{\ul{T}}\leq1}=(R\ang{T_1,\dotsc,T_m},R^+\ang{T_1,\dotsc,T_m})_\solid
  \end{align*}
  using Lemma \ref{ss:closedunitdisk}. Lemma \ref{ss:resolvable} shows that $A$ lies in $D_{\pc}(R[T_1,\dotsc,T_m])$, so
  \begin{align}\label{eq:Zariskiclosedindisk}
   \tag{$\Bumpeq$} R\ang{T_1,\dotsc,T_m}\otimes_{R[T_1,\dotsc,T_m]}A
  \end{align}
  lies in $D_{\pc}(R\ang{T_1,\dotsc,T_m})$. By \cite[Theorem 5.50]{And21}, (\ref{eq:Zariskiclosedindisk}) yields a pseudocompact object of $D(R[T_1,\dotsc,T_m]_{\abs{T}\leq1})$ as in \cite[Definition 9.9]{CS22}, which implies that $(\AnSpec\ul{f})^!$ commutes with direct sums of objects of $D^{\geq n}(R[T_1,\dotsc,T_m]_{\abs{T}\leq1})$ for all integers $n$. Because (\ref{eq:Zariskiclosedindisk}) also lies in $D^{\leq0}(R[T_1,\dotsc,T_m]_{\abs{T}\leq1})$, the functor $(\AnSpec\ul{f})^!$ is also left $t$-exact.

  Hence we can assume that $A=R[T_1,\dotsc,T_m]$. Consider the morphisms
  \begin{align*}
    \xymatrix{\AnSpec A_{\abs{\ul{T}}\leq1}\ar[r]^-j & \AnSpec A^{\alg}\ar[r]^-p & \AnSpec(R,R^+)_\solid,}
  \end{align*}
  and note that the restriction of $\pi$ to $\AnSpec A_{\abs{\ul{T}}\leq1}$ is $p\circ j$. We claim that $j^!p^!$ is naturally isomorphic to $j^*p^*[m]$. First, \cite[Lemma 3.2.9]{RC24} and \ref{ss:Z[T]facts} show that $j$ is cohomologically \'etale, so $j^*=j^!$. Since $p$ is induced, for all objects $M$ of $D_\solid(R,R^+)$ we have $j^!p^!M=j^!\ul{\Hom}_R(A,M)$. Next, note that $\ul{\Hom}_R(A,M)$ is naturally isomorphic to $M\lp{T_1^{-1},\dotsc,T_m^{-1}}/M[T_1,\dotsc,T_m]$, which is a module for the idempotent algebra $\bZ\lp{T_j^{-1}}\otimes_{\bZ_\solid}(R,R^+)_\solid$ in $D(A^{\alg})$ for all $1\leq j\leq r$. By \ref{ss:Z[T]facts}, $j^!$ is the intersection of the open embeddings as in \cite[p.~61]{CS22} corresponding to these idempotent algebras, so applying $j^!$ to the tensor product of the exact triangles
\begin{align*}
\xymatrix{M[T_j]\ar[r] & M\lp{T_j^{-1}}\ar[r] & M\lp{T_j^{-1}}/M[T_j]\ar[r]^-{+1} &}
\end{align*}
implies that
\begin{align*}
  j^!\ul{\Hom}_R(A,M) &= j^!M\lp{T_1^{-1},\dotsc,T_m^{-1}}/M[T_1,\dotsc,T_m] = j^!M[T_1,\dotsc,T_m][m] \\
  &=j^!p^*M[m] = j^*p^*M[m].
\end{align*}
Finally, the desired result follows immediately from the claim.
\end{proof}

\subsection{}\label{ss:grauert}
As a consequence, we prove the following finitude result for proper pushforwards, which is the analogue of Grauert's coherence theorem in our setting.
\begin{cor*}
  For any proper scheme $X$ of finite presentation over $R$, the functor
  \begin{align*}
    \pi_*:D(X^{\an})\ra D_\solid(R,R^+)
  \end{align*}
  sends $D_{\pc}(X^{\an})$ to $D_{\pc}(R)$.
\end{cor*}
\begin{proof}
  Theorem \ref{ss:AnStackGAGA} indicates that $X^{\an}\ra^\sim X^{\alg}$. Because $X^{\alg}$ is covered by finitely many objects of the form $\AnSpec(\cA,\cM)$, it is quasicompact, so the same holds for $X^{\an}$. Therefore \ref{ss:AnStackpc} shows that $X^{\an}$ is covered by finitely many objects of the form $\AnSpec\ov{A}_{\abs{\ul{f}}\leq1}$. Since $X$ is separated over $R$, the intersection of these objects over $X^{\an}$ are also of this form, so Lemma \ref{ss:AnRinginducednuclear} and \cite[Theorem 8.6 (1)]{CS22} imply that $\pi_*$ sends nuclear objects of $D(X^{\an})$ to nuclear objects of $D_\solid(R,R^+)$ as in \cite[Definition 8.5 (1)]{CS22}.

We claim that $\pi_*$ is naturally isomorphic to $\pi_!$. To see this, note that Proposition \ref{ss:closureofclosedunitdisk} indicates that $(R,R^+)_\solid\ra\ov{A}_{\abs{\ul{f}}\leq1}$ is induced as in \cite[Proposition 12.8]{Sch20}. Hence \cite[Lemma 4.5.5]{HM24} implies that $\AnSpec\ov{A}_{\abs{\ul{f}}\leq1}\ra\AnSpec(R,R^+)_\solid$ is cohomologically prim as in as in \cite[Definition 4.5.1.(b)]{HM24} with trivial codualizing complex, and applying \cite[Remark 4.4.11.(iii)]{HM24} to the cover of $X^{\an}$ by such objects shows that the same holds for $\pi$. This yields the claim.

Let $C$ be a pseudocompact object of $D(X^{\an})$, and let $d$ be an integer. Then Proposition \ref{ss:uppershriek} shows there exists an integer $e$ such that, for all families $(M_\al)_\al$ of objects of $D^{\geq d}_{\solid}(R,R^+)$, the family $(\pi^!M_\al)_\al$ lies in $D^{\geq e}(X^{\an})$ and satisfies
\begin{align*}
\bigoplus_\al\pi^!M_\al=\pi^!\big(\bigoplus_\al M_\al\big).
\end{align*}
Because $C$ is pseudocompact, the claim indicates that
\begin{align*}
  \Hom(\pi_*C,\bigoplus_\al M_\al) &= \Hom(C,\pi^!\bigoplus_\al M_\al) \\
  &= \Hom(C,\bigoplus_\al\pi^!M_\al) = \bigoplus_\al\Hom(C,\pi^!M_\al) = \bigoplus_\al\Hom(\pi_*C,M_\al).
\end{align*}
Therefore $\pi_*C$ is pseudocompact, as desired.
\end{proof}

\subsection{}\label{ss:pcGAGA}
Using Corollary \ref{ss:grauert}, we prove that Theorem \ref{ss:AnStackGAGA} restricts to an equivalence between subcategories of pseudocoherent objects as follows. For any scheme $X$ locally of finite type over $R$, Theorem \ref{ss:AnStackGAGA} shows that $X^{\alg}$ is glued from objects of the form $\AnSpec A^{\alg}$ along covers induced from Zariski covers. Hence \ref{ss:resolvable} indicates that the full subcategory $D_{\pc}(A)\subseteq D(A)\hra D(A^{\alg})$ naturally glues to a full subcategory of $D(X^{\alg})$ that is naturally equivalent to $D_{\pc}(X)$.

Consider the functor $D(X^{\alg})\ra D(X^{\an})$. When $X=\Spec{A}$ is affine, \cite[Theorem 5.50]{And21} and \cite[Theorem 8.6 (1)]{CS22} imply that this sends $D_{\pc}(A)$ to $D_{\pc}((\Spec{A})^{\an})$, so in general this sends $D_{\pc}(X)$ to $D_{\pc}(X^{\an})$.
\begin{thm*}
  For any proper scheme $X$ of finite presentation over $R$, the resulting functor $D_{\pc}(X)\ra D_{\pc}(X^{\an})$ is an equivalence.
\end{thm*}
\begin{proof}
  Full faithfulness follows from Theorem \ref{ss:AnStackGAGA}. For essential surjectivity, let $C$ be an object of $D_{\pc}(X^{\an})$. By full faithfulness, it suffices to prove that, for all affine open subschemes $j:\Spec{A}\hra X$, the image of $C$ under the composition
\begin{align*}
\xymatrix{D_{\pc}(X^{\an})\subseteq D(X^{\an}) &\ar[l]_-\sim D(X^{\alg})\ar[r]^-{j^{\alg,*}} & D(A^{\alg})}
\end{align*}
lies in the image of $D_{\pc}(A)\subseteq D(A)\hra D(A^{\alg})$. Since $j^{\alg,*}$ preserves pseudocompact objects, it suffices to prove that $j^{\alg,*}C$ lies in the image of $D(A)\hra D(A^{\alg})$. Noetherian approximation yields a finitely generated ideal sheaf $\sZ$ of $X$ whose vanishing locus equals the closed subset $X\ssm\Spec{A}\subseteq X$, so after replacing $X$ with its blowup along $\sZ$, we can assume that $\Spec{A}=X\ssm D$ for an effective Cartier divisor $D$ on $X$. Then $j_*^{\alg}j^{\alg,*}C=\varinjlim_{n\geq0}C\otimes\sO_X(-nD)$ and hence
\begin{align*}
\pi_*j_*^{\alg}j^{\alg,*}C = \varinjlim_{n\geq0}\pi_*\big(C\otimes\sO_X(-nD)\big).
\end{align*}
Note that $\pi_*j_*^{\alg}j^{\alg,*}C$ is the image of $j^{\alg,*}C$ in $D_\solid(R,R^+)$. Because $C\otimes\sO_X(-nD)$ lies in $D_{\pc}(X^{\an})$, Corollary \ref{ss:grauert} shows that $\pi_*\big(C\otimes\sO_X(-nD)\big)$ lies in $D_{\pc}(R)$, so the image of $j^{\alg,*}C$ in $D_\solid(R,R^+)$ lies in the image of $D(R)\hra D_\solid(R,R^+)$. Therefore $j^{\alg,*}C$ lies in the image of $D(A)\hra D(A^{\alg})$, as desired.
\end{proof}

\subsection{}\label{ss:appliedGAGA}
Finally, we prove Theorem H. Recall that the assignment
\begin{align*}
\Spa(A,A^+)\mapsto\AnSpec(A,A^+)_\solid
\end{align*}
from Tate affinoid adic spaces to $\{\mbox{analytic stacks}\}$ sends rational covers to steady localization covers \cite[Proposition 2.3.2]{RC24}, so it naturally glues to a functor
\begin{align*}
(-)_\solid:\{\mbox{analytic adic spaces}\}\ra\{\mbox{analytic stacks}\}.
\end{align*}

For the rest of this section, assume that $R$ is sousperfectoid, and write $S$ for $\Spa(R,R^+)$. Then for any smooth scheme $X$ over $R$, the adic space $X^{\an}_S$ as in Definition \ref{ss:Huberanalytification} exists by \cite[Proposition IV.4.9]{FS21}.
\begin{thm*}
  For any smooth proper scheme $X$ over $R$, the functor
  \begin{align*}
    \{\mbox{vector bundles on X}\}\ra\{\mbox{vector bundles on }X^{\an}_S\}
  \end{align*}
is an equivalence of categories.
\end{thm*}
\begin{proof}
  By covering $X^{\an}$ by small enough steady subspaces of the form $\AnSpec A_{\abs{\ul{f}}\leq1}$, we see that $X^{\an}$ is isomorphic to $(X^{\an}_S)_\solid$. Then Theorem \ref{ss:pcGAGA} yields a natural equivalence $D_{\pc}(X)\ra^\sim D_{\pc}((X_S^{\an})_\solid)$, hence a natural equivalence between the full subcategories consisting of perfect objects $C$ as in \cite[Proposition 9.2]{CS22} such that $-\otimes C$ is $t$-exact.

  For $D_{\pc}(X)$, \cite[Tag 0658]{stacks-project} shows that this subcategory is naturally equivalent to $\{\mbox{vector bundles on }X\}$. For $D_{\pc}((X^{\an}_S)_\solid)$, applying \cite[Theorem 5.50]{And21} and \cite[Tag 0658]{stacks-project} to the above cover of $X^{\an}\cong(X^{\an}_S)_\solid$ by objects of the form $\AnSpec A_{\abs{f}\leq1}$ shows that this subcategory is naturally equivalent to $\{\mbox{vector bundles on }X^{\an}_S\}$.  
\end{proof}

\subsection{}\label{ss:HX}
As an application of Theorem \ref{ss:appliedGAGA}, we conclude this section by proving that the analytification of algebraic stacks of bundles agrees with analytic stacks of bundles. This answers a question of Heuer--Xu \cite[Remark 8.1.2]{HX24}.

Let $K$ be an algebraically closed nonarchimedean field over $\bQ_p$, and write $\Perfd_K$ for the category of affinoid perfectoid spaces over $\Spa{K}$. Endow $\Perfd_K$ with the v-topology. Recall from \cite[Definition 8.4.4]{HX24} the diamondification functor
\begin{align*}
(-)^\Diamond:\{\mbox{algebraic stacks over }K\}\ra\{\mbox{small v-stacks on }\Perfd_K\}.
\end{align*}
Let $X$ be a smooth proper scheme over $K$, and let $G$ be a linear algebraic group over $K$. Write $\cB_{G,X}$ and $\cH_{G,X}$ for the algebraic stacks over $K$ of $G$-bundles and $G$-Higgs bundles on $X$, respectively. Write $\Bun_{G,X}$ and $\Hig_{G,X}$ for the small v-stacks on $\Perfd_K$ as in \cite[Definition 7.12.(1)]{Heu24} and \cite[Definition 7.12.(2)]{Heu24}, respectively.
\begin{thm*}
  The natural morphisms of v-stacks on $\Perfd_K$
  \begin{align*}
    \cB_{G,X}^\Diamond\ra\Bun_{G,X}\mbox{ and }\cH_{G,X}^\Diamond\ra\Hig_{G,X}
  \end{align*}
are isomorphisms.
\end{thm*}
\begin{proof}
This follows immediately from Theorem \ref{ss:appliedGAGA}.
\end{proof}

\section{$\ell$-adic realizations of Berkovich motives}\label{s:berkovichmotives}
In \S\ref{s:sheaf} and \S\ref{s:langlands}, we use Berkovich (i.e. overconvergent) motivic sheaves as in \cite{Sch24, Sch25} because they enjoy good formal properties. For example, there are $!$-pushforwards, and when evaluated on condensed anima, they can be described in terms of classical sheaves on profinite sets. We start this section by recalling this material.

Difficulties arise when we need to compare Berkovich motives for (analytifications of) moduli stacks of global shtukas $\Sht^I_{G,V}$ with their classical $\ell$-adic sheaves. Ideally, we would start from a motivic intersection cohomology sheaf on $\Sht^I_{G,V}$, analytify it, and then pass to $\ell$-adic realizations. However, such motivic intersection cohomology sheaves are only known when $G$ is split \cite{RS20}, and one would still need to compare Berkovich motives with classical motives. To avoid this, we develop some tools that keep track of the relationship between $\ell$-adic sheaves on varieties, $\ell$-adic sheaves on their analytifications, and motivic sheaves on said analytifications.

We conclude this section by recalling the description of motivic sheaves on $\breve{F}_v$ and $\Div_v^1$ and describing their relationship with motivic nearby cycles.

\subsection*{Notation}
In this section, we work over $\ov\bF_q$.

\subsection{}
We start by recalling the theory of \emph{Berkovich motives} \cite{Sch24}. Endow the category $\{\mbox{perfectoid rings over }\ov\bF_q\}^{\op}$ with the arc-topology as in \cite[Definition 3.1]{Sch24}, which is subcanonical by \cite[Theorem 3.14]{Sch24}. Write $\{\mbox{arc-sheaves over }\ov\bF_q\}$ for the $\infty$-category of sheaves of anima on $\{\mbox{perfectoid rings over }\ov\bF_q\}^{\op}$, and recall from \cite[p.~54]{Sch24} the $\bZ[\frac1q]$-linear overconvergent motivic $6$-functor formalism $D_{\mot}(-)$ on $\{\mbox{arc-sheaves over }\ov\bF_q\}$.

\subsection{}\label{ss:motivesoncondensedgroupoids}
When evaluating on stacks arising from profinite sets (i.e. \emph{condensed anima}), Berkovich motives enjoy the following description. Recall from \cite[Construction 3.5.16]{HM24} the $\bZ[\frac1q]$-linear $6$-functor formalism $D(-,\bZ[\frac1q])$ on $\{\mbox{condensed anima}\}$, and write
\begin{align*}
\ul{(-)}:\{\mbox{condensed anima}\}\ra\{\mbox{arc-sheaves over }\ov\bF_q\}
\end{align*}
for the left Kan extension of the assignment $X\mapsto\ul{X}$ from profinite sets. Then the proof of \cite[Proposition 2.1]{Sch25} implies we have a morphism of $3$-functor formalisms
\begin{align*}
D(-,\bZ[\textstyle\frac1q])\ra D_{\mot}(\ul{(-)})
\end{align*}
on $\{\mbox{condensed anima}\}$ with respect to $!$-able morphisms for $D(-,\bZ[\frac1q])$.
\begin{prop*}
This induces an isomorphism of $6$-functor formalisms
\begin{align*}
D(-,\bZ[\textstyle\frac1q])\otimes_{D(\bZ[\frac1q])}D_{\mot}(*)\ra^\sim D_{\mot}(\ul{(-)})
\end{align*}
on $\{\mbox{condensed anima}\}$ with respect to $!$-able morphisms for $D(-,\bZ[\frac1q])$.
\end{prop*}
\begin{proof}
When evaluating on finite sets, this is immediate. For any cofiltered limit $X=\varprojlim_iX_i$ of finite sets, we have $\ul{X}=\varprojlim_i\ul{X_i}$, so \cite[Lemma 10.4]{Sch24} shows that
  \begin{align*}
    D(X,\bZ[\textstyle\frac1q])\otimes_{D(\bZ[\frac1q])}D_{\mot}(*) &= \big[\varinjlim_iD(X_i,\bZ[\textstyle\frac1q])\big]\otimes_{D(\bZ[\frac1q])}D_{\mot}(*) \\
                                                                               &=\varinjlim_i\big[D(X_i,\bZ[\textstyle\frac1q])\otimes_{D(\bZ[\frac1q])}D_{\mot}(*)\big]\\
    &=\varinjlim_i D_{\mot}(\ul{X_i}) = D_{\mot}(\ul{X}).
  \end{align*}
Finally, the result follows from the uniqueness in \cite[Theorem 3.4.11]{Sch25}.
\end{proof}

\subsection{}\label{ss:ladicrealization}
Berkovich motives are related to \'etale sheaves in the following way. Write
\begin{align*}
a'^*:\{\mbox{small v-stacks over }\ov\bF_q\}\ra\{\mbox{arc-sheaves over }\ov\bF_q\}
\end{align*}
for the left Kan extension of the assignment $\Spa(R,R^+)\mapsto\Hom_{\ov\bF_q}(R,-)$ \cite[p.~3]{Sch25}. For all positive integers $m$ and any small v-stack $X$ over $\ov\bF_q$, \cite[Proposition 12.3]{Sch24} yields a natural fully faithful symmetric monoidal functor
\begin{align*}
\Mod_{\bZ/\ell^m}(D_{\mot}(a'^*X))\hra D_{\et}(X,\bZ/\ell^m)
\end{align*}
that is compatible with pullback and whose essential image consists of overconvergent objects as in \cite[Proposition IV.2.4]{FS21}.

Using this, one can define $\ell$-adic realization functors as follows. Write $\wh{D}_{\et}(X,\bZ_\ell)$ for $\varprojlim_{m\geq1}D_{\et}(X,\bZ/\ell^m)$, so that taking $\varprojlim_{m\geq1}$ of the above induces a functor
\begin{align*}
\wh{r}_{\ell,X}:D_{\mot}(a'^*X)\ra\varprojlim_{m\geq1}\Mod_{\bZ/\ell^m}(D_{\mot}(a'^*X))\hra\wh{D}_{\et}(X,\bZ_\ell).
\end{align*}
Write $\wh{D}(\bZ_\ell)$ for $\varprojlim_{m\geq1}D(\bZ/\ell^m)$. Then the natural functor $D_{\perf}(\bZ_\ell)\ra\wh{D}(\bZ_\ell)$ is fully faithful \cite[Lemma 4.2]{Bha16}, and \cite[p.~168]{FS21} identifies $\wh{D}_{\et}(*,\bZ_\ell)$ with $\wh{D}(\bZ_\ell)$. Under this identification, \cite[Theorem 11.1]{Sch24} shows that $\wh{r}_{\ell,*}$ restricts to a functor $D_{\mot}(*)^\om\ra D_{\perf}(\bZ_\ell)$, so \cite[Theorem 11.1]{Sch24} indicates that applying $\Ind$ yields a functor $r_{\ell,*}:D_{\mot}(*)\ra D(\bZ_\ell)$.

\subsection{}\label{ss:classifyingstack}
For the rest of this section, we will use the following base change of Berkovich motives. For any small v-stack $X$ over $\ov\bF_q$, write $D_{\mot}(X)$ for $D_{\mot}(a'^*X)$.

Let $\La$ be a $\bZ_\ell$-algebra, and consider the composition
\begin{align*}
  D_{\mot}(*)\lra^{r_{\ell,*}}(\bZ_\ell)\lra D(\La).
\end{align*}
Write $D(X,\La)$ for the symmetric monoidal $\infty$-category $D_{\mot}(X)\otimes_{D_{\mot}(*)}D(\La)$, so that $D(-,\La)$ is a $\La$-linear $6$-functor formalism on $\{\mbox{small v-stacks over }\ov\bF_q\}$.

As a consequence of Proposition \ref{ss:motivesoncondensedgroupoids}, evaluating $D(-,\La)$ on classifying stacks yields categories of smooth representations. More precisely, let $G$ be a locally pro-$p$ group, let $K$ be a compact open subgroup of $G$, and write $j:*/\ul{K}\ra*/\ul{G}$ for the associated morphism of small v-stacks.
\begin{cor*}
The symmetric monoidal $\infty$-category $D(*/\ul{G},\La)$ is naturally equivalent to the derived category of smooth representations of $G$ over $\La$. Under this identification, $j_!:D(*/\ul{K},\La)\ra D(*/\ul{G},\La)$ corresponds to $\cInd_K^G$.
\end{cor*}
\begin{proof}
View $G$ as a condensed set, and consider the condensed anima $*/G$. Then the arc-sheaf $\ul{*/G}$ is naturally isomorphic to $a'^*(*/\ul{G})$, so the result follows from Proposition \ref{ss:motivesoncondensedgroupoids}, \cite[Proposition 5.3.10]{HM24}, and \cite[Proposition 5.4.4]{HM24}.
\end{proof}

\subsection{}\label{ss:ladicrealizationvarieties}
To facilitate comparisons with classical \'etale $\ell$-adic sheaf theories for varieties over $\ov\bF_q$ in \S\ref{s:sheaf}, we will consider the following notion of \emph{\ul{Z}ariski-\ul{c}onstructible} motives (or \'etale sheaves) relative to a base. Let $Z$ be a separated finite type scheme over $\ov\bF_q$, and let $C$ be a small v-stack over $Z^\Diamond$.
\begin{defn*}For any separated finite type scheme $X$ over $Z$,
  \begin{enumerate}[a)]
  \item Write $D_{\mot}^{Z\bs C}(X^\Diamond_C)\subseteq D_{\mot}(X^\Diamond_C)$ for the full subcategory generated under cones and retracts by objects of the form $f_{C,!}^\Diamond\bZ[\frac1q](n)$, where $f:Y\ra X$ runs over separated finite type morphisms and $n$ runs over integers.
  \item For all positive integers $m$, write $D_{\et}^{Z\bs C}(X^\Diamond_C,\bZ/\ell^m)\subseteq D_{\et}(X^\Diamond_C,\bZ/\ell^m)$ for the full subcategory generated under cones and retracts by objects of the form $f_{C,!}^\Diamond(\bZ/\ell^m)$, where $f:Y\ra X$ runs over separated finite type morphisms.
  \item Write $D_{\et}^{Z\bs C}(X^\Diamond_C,\bZ_\ell)\subseteq\wh{D}_{\et}(X^\Diamond_C,\bZ_\ell)$ for the full subcategory
    \begin{align*}
      \varprojlim_{m\geq1}D_{\et}^{Z\bs C}(X^\Diamond_C,\bZ/\ell^m).
    \end{align*}
  \end{enumerate}
\end{defn*}

\subsection{}\label{ss:morphismsof6FFs}
One immediately has the following relationship between \'etale sheaves on varieties over $\ov\bF_q$ and \'etale sheaves on their analytifications. Note that $D_{\mot}((-)^\Diamond_C)$, $\wh{D}_{\et}((-)^\Diamond_C,\bZ_\ell)$ and $D(-,\bZ_\ell)$ are $6$-functor formalisms on
\begin{align*}
\{\mbox{separated finite type schemes over }Z\}.
\end{align*}
Now \cite[Proposition 27.1]{Sch17}, \cite[Proposition 27.2]{Sch17}, and \cite[Proposition 27.4]{Sch17} yield a fully faithful morphism of $3$-functor formalisms
\begin{align*}
c^*:D(-,\bZ_\ell)\hra \wh{D}_{\et}((-)^\Diamond,\bZ_\ell)
\end{align*}
on $\{\mbox{separated finite type schemes over }Z\}$, and proper base change shows that pullback induces a morphism of $3$-functor formalisms
\begin{align*}
|_C:\wh{D}_{\et}((-)^\Diamond,\bZ_\ell)\ra\wh{D}_{\et}((-)^\Diamond_C,\bZ_\ell)
\end{align*}
on $\{\mbox{separated finite type schemes over }Z\}$.

\subsection{}\label{prop:ladicrealizationvarieties}
One also has the following relationship between Zariski-constructible motives and \'etale sheaves.
\begin{prop*}\hfill
  \begin{enumerate}[1)]
  \item $D_{\mot}((-)^\Diamond_C)$ restricts to a $3$-functor formalism $D_{\mot}^{Z\bs C}((-)^\Diamond_C)$.
  \item $\wh{D}_{\et}((-)^\Diamond_C,\bZ_\ell)$ restricts to a $3$-functor formalism $D^{Z\bs C}_{\et}((-)^\Diamond_C,\bZ_\ell)$.
  \item The composition $|_C\circ c^*$ sends $D_{\cons}(-,\bZ_\ell)$ to $D^{Z\bs C}_{\et}((-)^\Diamond_C,\bZ_\ell)$.
  \item The functor $\wh{r}_{\ell,(-)^\Diamond_C}$ sends $D^{Z\bs C}_{\mot}((-)^\Diamond_C)$ to $D^{Z\bs C}_{\et}((-)^\Diamond_C,\bZ_\ell)$ and induces a morphism of $3$-functor formalisms on $\{\mbox{separated finite type schemes over }Z\}$.
  \end{enumerate}
\end{prop*}
\begin{proof}
  Part 1) and part 2) follow immediately from proper base change. For part 3), recall that $D_{\cons}(X,\bZ/\ell^m)$ is generated under cones and retracts by objects of the form $f_!(\bZ/\ell^m)$, where $f:Y\ra X$ runs over separated finite type morphisms. Therefore \ref{ss:morphismsof6FFs} and proper base change show that $|_C\circ c^*$ sends $D_{\cons}(X,\bZ/\ell^m)$ to $D_{\et}^{Z\bs C}((-)_C^\Diamond,\bZ/\ell^m)$. Finally, taking $\varprojlim_{m\geq1}$ yields the desired result.

  For the first statement in part 4), let $f:Y\ra X$ be a separated finite type morphism, and let $n$ be an integer. By Nagata compactification, there exists a proper morphism $\ov{f}:\ov{Y}\ra X$ and an open embedding $j:Y\ra\ov{Y}$ such that $f=\ov{f}\circ j$. Then $j^\Diamond_C$ is a partially proper open embedding, so the image of $j_{C,!}^\Diamond\bZ[\textstyle\frac1q](n)$ in $D_{\et}(\ov{Y}^\Diamond_C,\bZ/\ell^m)$ is isomorphic to $j^\Diamond_{C,!}(\bZ/\ell^m)$. Because $f^\Diamond_C$ is proper and of finite cohomological dimension as in \cite[Definition 4.17]{Sch24}, we see from \cite[Proposition 12.3]{Sch24} that the image of $f^\Diamond_{C,!}\bZ[\frac1q](n)=\ov{f}^\Diamond_{C,*}j^\Diamond_{C,!}\bZ[\frac1q](n)$ in $D_{\et}(X^\Diamond_C,\bZ/\ell^m)$ is isomorphic to $\ov{f}^\Diamond_{C,*}j_{C,!}^\Diamond(\bZ/\ell^m) = f^\Diamond_{C,!}(\bZ/\ell^m)$. Hence taking $\varprojlim_{m\geq1}$ yields the desired statement. Finally, the second statement in part 4) follows from \ref{ss:ladicrealization} and the above argument.
\end{proof}

\subsection{}\label{ss:indextension}
Finally, we $\Ind$-extend the preceding discussion. Proposition \ref{prop:ladicrealizationvarieties}.2) implies that $\Ind$-extending yields a morphism
\begin{align*}
\Ups:\Ind D_{\mot}^{Z\bs C}((-)^\Diamond_C)\ra D_{\mot}((-)^\Diamond_C)
\end{align*}
of $3$-functor formalisms on $\{\mbox{separated finite type schemes over }Z\}$, and Proposition \ref{prop:ladicrealizationvarieties}.3) indicates that applying $\Ind$ to $|_C\circ c^*:D_{\cons}(-,\bZ_\ell)\ra D^{Z\bs C}_{\et}((-)^\Diamond_C,\bZ_\ell)$ yields a morphism of $3$-functor formalisms
\begin{align*}
\rho: D(-,\bZ_\ell)\ra\Ind D^{Z\bs C}_{\et}((-)^\Diamond_C,\bZ_\ell)
\end{align*}
on $\{\mbox{separated finite type schemes over }Z\}$. Finally, Proposition \ref{prop:ladicrealizationvarieties}.4) shows that $\wh{r}_{\ell,(-)^\Diamond_C}$ induces a morphism of $3$-functor formalisms
\begin{align*}
D^{Z\bs C}_{\mot}((-)^\Diamond_C)\ra D^{Z\bs C}_{\et}((-)^\Diamond_C,\bZ_\ell)
\end{align*}
on $\{\mbox{separated finite type schemes over }Z\}$, so applying $\Ind$ yields a morphism 
\begin{align*}
r_{\ell,(-)^\Diamond_C}:\Ind D^{Z\bs C}_{\mot}((-)^\Diamond_C)\ra\Ind D^{Z\bs C}_{\et}((-)^\Diamond_C,\bZ_\ell)
\end{align*}
of $3$-functor formalisms on $\{\mbox{separated finite type schemes over }Z\}$.

\subsection{}\label{ss:controlSpaC}
Write $F$ for $\bF_q\lp{\pi}$, fix a separable closure $\ov{F}$ of $F$, and write $C$ for the completion of $\ov{F}$. Then there is a simple description of Zariski-constructible \'etale sheaves on $\Spa{C}$ (relative to any $Z$). More precisely, $\wh{D}_{\et}(\Spa{C},\bZ_\ell)$ is naturally isomorphic to $\wh{D}(\bZ_\ell)$, and under this identification, \ref{ss:morphismsof6FFs} implies that
\begin{align*}
D_{\et}^{Z\bs\!\Spa{C}}(\Spa{C},\bZ_\ell)=D_{\perf}(\bZ_\ell).
\end{align*}
We prove an analogous description of Zariski-constructible motives on $\Spa{C}$:
\begin{prop*}
We have $D_{\mot}^{Z\bs\!\Spa C}(\Spa C)= D_{\mot}(\Spa C)^\om$.
\end{prop*}
Proposition \ref{ss:controlSpaC} and \cite[Proposition 10.1]{Sch24} imply that
\begin{align*}
\Ups:\Ind D^{Z\bs\Spa{C}}_{\mot}(\Spa{C})\ra D_{\mot}(\Spa{C})
\end{align*}
is an equivalence.
\begin{proof}
  By considering $f:Z\times X\ra Z$ for smooth projective schemes $X$ over $\ov\bF_q$, \cite[Proposition 10.1]{Sch24} shows that $D_{\mot}(\Spa C)^\om$ lies in $D_{\mot}^{Z\bs\!\Spa{C}}(\Spa C)$. For the reverse inclusion, it suffices to prove that, for any separated finite type morphism $f:Y\ra\Spec C$, the object $f^{\an,\Diamond}_!\bZ[\frac1q]$ in $D_{\mot}(\Spa C)$ is compact. Since $(-)^{\an,\Diamond}$ factors through limit perfection, we can assume that $Y$ is reduced.

  We now induct on $\dim{Y}$, where the $\dim{Y}=0$ case is immediate. By excision \cite[Proposition 4.25]{Sch24} and the induction hypothesis, we can assume that $Y$ is irreducible. Then \cite[Theorem 4.1]{deJ96} yields a smooth projective morphism $\wt{f}:\wt{Y}\ra\Spec C$, a dense open subspace $V\subseteq Y$, a dense open subspace $i:\wt{V}\hra\wt{Y}$, and a finite \'etale morphism $j:\wt{V}\ra V$ over $C$. By excision \cite[Proposition 4.25]{Sch24} and the induction hypothesis, we can assume that $Y=V$. Write $d$ for the product of all primes $\ell\neq p$ dividing $\deg{j}$. Because $(f_!^{\an,\Diamond}\bZ[\frac1q])\otimes_{\bZ[\frac1q]}\bZ/d^m=f_!^{\an,\Diamond}(\bZ/d^m)$, spreading out implies that the image of $f_!^{\an,\Diamond}\bZ[\textstyle\frac1q]$ in
\begin{align*}
  \varprojlim_{m\geq1}\Mod_{\bZ/d^m}(D_{\mot}(\Spa{C})) = \wh{D}\big(\prod_{\ell\mid d}\bZ_\ell\big)
\end{align*}
is compact. Therefore it suffices to prove that the image $f_!^{\an,\Diamond}\bZ[\frac1{dq}]$ of $f_!^{\an,\Diamond}\bZ[\frac1q]$ in $\Mod_{\bZ[\frac1{dq}]}(D_{\mot}(\Spa{C}))$ is compact.

The proof of \cite[Proposition 10.1]{Sch24} shows that $\wt{f}_!^{\an,\Diamond}\bZ[\frac1{dq}]$ is compact in
\begin{align*}
\Mod_{\bZ[\frac1{dq}]}(D_{\mot}(\Spa{C})),
\end{align*}
so excision \cite[Proposition 4.25]{Sch24} and the induction hypothesis imply that the same holds for $(\wt{f}\circ i)_!^{\an,\Diamond}\bZ[\frac1{dq}]$. Finally, the trace map exhibits $f_!^{\an,\Diamond}\bZ[\frac1{dq}]$ as a retract of $(\wt{f}\circ i)_!^{\an,\Diamond}\bZ[\frac1{dq}]$, which yields the desired result.
\end{proof}

\subsection{}\label{ss:WeilDeligne}
Write $\breve{F}$ for $\ov\bF_q\lp{\pi}$, write $\phi:\Spd\breve{F}\ra\Spd\breve{F}$ for the geometric $q$-Frobenius automorphism over $\ov\bF_q$, and write $\Div^1$ for the small v-sheaf $(\Spd\breve{F})/\phi^\bZ$ over $\ov\bF_q$.

Let us recall the following descriptions of motives on $\Spd\breve{F}$ and $\Div^1$. Fix a group isomorphism $\ov\bF_q^\times\cong\bQ/\bZ[\frac1p]$, and choose a compatible system $\{\pi^{1/n}\}_{p\nmid n}$ of $n$-th roots of $\pi$ in $\ov{F}$. Recall that these choices induce a short exact sequence of certain group schemes over $\bZ[\frac1q]$ \cite[p.~10]{Sch25}
\begin{align*}
\xymatrix{1\ar[r]& \ID_{F}\ar[r] & \WD_{F}\ar[r] & \ul{\bZ}\ar[r] & 1.}
\end{align*}
\begin{lem*}
  These choices induce compatible equivalences
  \begin{align*}
    D(\Spd\breve{F},\La)\cong D_{\qcoh}\big((\Spec\La)/\!\ID_{F}\!\big)\mbox{ and }D(\Div^1,\La)\cong D_{\qcoh}\big((\Spec\La)/\!\WD_{F}\!\big).
  \end{align*}
\end{lem*}
\begin{proof}
  Write $MG_{\ov\bF_q}$ for the algebraic stack over $\bZ[\frac1q]$ from \cite[p.~62]{Sch24}, and recall that we have a natural functor $D_{\qcoh}(MG_{\ov\bF_q})\ra D_{\mot}(*)$ \cite[p.~62]{Sch24}. Consider the morphism $\Spec\bZ[\frac1q]\ra MG_{\ov\bF_q}$ induced by the isomorphism $\ov\bF_q^\times\cong\bQ/\bZ[\frac1p]$ as in \cite[p.~10]{Sch25}; note that the isomorphism $\ov\bF_q^\times\cong\bQ/\bZ[\frac1p]$ also induces a commutative square
\begin{align*}
  \xymatrix{D_{\qcoh}(MG_{\ov\bF_q})\ar[r]\ar[d] & D_{\mot}(*)\ar[d]^-{r_{\ell,*}}\\
  D(\bZ[\frac1q])\ar[r] & D(\bZ_\ell).}
\end{align*}
Write $\wt{X}_{\Div^1}$ and $X_{\Div^1}$ for the algebraic stacks over $MG_{\ov\bF_q}$ from \cite[Definition 4.1]{Sch25}. Then \cite[Theorem 4.2]{Sch25} and its proof yield compatible equivalences
\begin{align*}
  D_{\qcoh}(\wt{X}_{\Div^1})\otimes_{D_{\qcoh}(MG_{\ov\bF_q})}D_{\mot}(*)&\ra^\sim D_{\mot}(\Spd\breve{F}),\\
  D_{\qcoh}(X_{\Div^1})\otimes_{D_{\qcoh}(MG_{\ov\bF_q})}D_{\mot}(*)&\ra^\sim D_{\mot}(\Div^1).
\end{align*}
so the above commutative square induces compatible equivalences
\begin{align*}
  D_{\qcoh}(\wt{X}_{\Div^1}\times_{MG_{\ov\bF_q}}\Spec\La)&\ra^\sim D(\Spd\breve{F},\La),\\
  D_{\qcoh}(X_{\Div^1}\times_{MG_{\ov\bF_q}}\Spec\La)&\ra^\sim D(\Div^1,\La).
\end{align*}
Finally, \cite[p.~10]{Sch25} compatibly identifies
\begin{align*}
  \wt{X}_{\Div^1}\times_{MG_{\ov\bF_q}}\Spec\La&\cong(\Spec\La)/\!\ID_{F},\\
  X_{\Div^1}\times_{MG_{\ov\bF_q}}\Spec\La&\cong(\Spec\La)/\!\WD_{F}\!.\qedhere
\end{align*}
\end{proof}

\subsection{}\label{ss:nearbycycles}
We conclude this section by describing how \emph{motivic nearby cycles} behave after $\ell$-adic realization. Write $f:\Spa{C}\ra*$ for the structure morphism. Recall that $f_*$ yields an equivalence $D_{\mot}(\Spa{C})\ra^\sim\Mod_{f_*\bZ[\frac1q]}(D_{\mot}(*))$ and that our choice of $\{\pi^{1/n}\}_{p\nmid n}$ from \ref{ss:WeilDeligne} induces a morphism $f_*\bZ[\frac1q]\ra\bZ[\textstyle\frac1q]$ of $\bE_\infty$-algebras in $D_{\mot}(*)$ \cite[p.~61]{Sch24}. Under this identification, write $\Psi:D_{\mot}(\Spa{C})\ra D_{\mot}(*)$ for the functor corresponding to $-\otimes_{f_*\bZ[\frac1q]}\bZ[\textstyle\frac1q]$ \cite[Definition 11.3]{Sch24}.

Recall that pullback yields an equivalence \cite[p.~168]{FS21}
\begin{align*}
\wh{D}_{\et}(*,\bZ_\ell)\ra^\sim\wh{D}_{\et}(\Spa{C},\bZ_\ell).
\end{align*}
\begin{prop*}
The functor $\Psi$ preserves compact objects. Moreover, we have a commutative square
\begin{align*}
  \xymatrixcolsep{1.5cm}
  \xymatrix{D_{\mot}(\Spa C)\ar[r]^-{\wh{r}_{\ell,\Spa C}}\ar[d]^-{\Psi} & \wh{D}_{\et}(\Spa C,\bZ_\ell)\\
  D_{\mot}(*)\ar[r]^-{\wh{r}_{\ell,*}} & \wh{D}_{\et}(*,\bZ_\ell).\ar[u]_-{\rotatebox{90}{$\sim$}}}
\end{align*}
\end{prop*}
Proposition \ref{ss:nearbycycles} and \ref{ss:ladicrealization} imply that $\wh{r}_{\ell,\Spa C}$ restricts to a functor
\begin{align*}
D_{\mot}(\Spa C)^\om\ra D_{\perf}(\bZ_\ell),
\end{align*}
so \cite[Proposition 10.1]{Sch24} indicates that applying $\Ind$ yields a functor
\begin{align*}
r_{\ell,\Spa C}:D_{\mot}(\Spa C)\ra D(\bZ_\ell).
\end{align*}
\begin{proof}
  Because $f^*:D_{\mot}(*)\ra D_{\mot}(\Spa{C})$ is left adjoint to $f_*$, it corresponds to $-\otimes_{\bZ[\frac1q]}f_*\bZ[\textstyle\frac1q]$. Therefore $\Psi\circ f^*$ is naturally isomorphic to the identity, so the first statement follows from \cite[Proposition 10.1]{Sch24} and \cite[Theorem 11.1]{Sch24}.

  For the second statement, the proof of \cite[Proposition 11.4]{Sch24} implies that, for all positive integers $m$, applying $-\otimes_{\bZ[\frac1q]}\bZ/\ell^m$ to $f_*\bZ[\frac1q]\ra\bZ[\textstyle\frac1q]$ gives an isomorphism. This yields the desired result.
\end{proof}

\section{Global rigid inner forms, by Peter Dillery}\label{s:dillery}
In the following, $G$ always denotes a connected reductive group over $F$ a global function field associated to the projective curve $C$ with constant field $\mathbb{F}_{q}$. We will break from the notation of the rest of the paper by denoting $F^{s}$ a separable closure of $F$ inside a fixed algebraic closure $\ov{F}$. The set of all places of $F$ will be denoted by $V_{F}$. We denote the Kottwitz gerbe over $F$ by $\tn{Kott}_{F}$, and the Kottwitz gerbe over each $F_{v}$ by $\tn{Kott}_{v}$. For a set of places $\Sigma$ of $F$, we denote by $O_{F,\Sigma}$ the ring of $\Sigma$-integers of $F$.

\subsection{}
Theorem \ref{ss:basiclocus} says that the $B(F,G)_{\tn{basic}}$-inner forms $G_{b}$ of a quasi-split connected reductive group $G$ appear naturally in the description of the semistable locus of $\Bun_{G,F}$. The goal of this appendix is to use results in the main body of the paper to prove fundamental geometric properties of the ``extended'' stack $\tn{Bun}_{G,F}^{e}$ introduced in \cite{Dil26}, which has geometry analogous to that of $\tn{Bun}_{G,F}$ but whose semistable locus encodes global rigid inner forms of $G$, in the sense of \cite{Dillery23b}, rather than just $B(F,G)_{\tn{basic}}$-inner forms. The advantage of using global rigid inner forms is that they enable a statement (\cite[Conjecture 5.7]{Dil26}) of a multiplicity formula for discrete automorphic representations that applies to arbitrary connected reductive groups over $F$. It is reasonable to believe that a $\tn{Bun}_{G,F}^{e}$-adapted version of the machinery in the main part of this paper together with the conjectural $\tn{Kott}_{v} \times_{F_{v}} \tn{Kal}_{v}$-adapted refined Langlands correspondence from \cite[Conjecture 5.1]{Dil26} at each $v$ can be used to prove cases of this multiplicity formula. We emphasize that the construction presented here is inspired by the work of \cite{Fargues} in the local case.

\subsection{}
We review the notion of local and global rigid inner forms. One can use class field theory to construct a global gerbe $\tn{Kal}_{F} \to \tn{Spec}(F)$ (\cite[Proposition 3.21]{Dillery23b}) and for each place $v \in V_{F}$ a local gerbe $\tn{Kal}_{v} \to \tn{Spec}(F_{v})$ (\cite[Theorem 3.4]{Dil23a}) which are banded by profinite multiplicative group schemes $P$ and $u_{v}$, respectively. We refer the reader to \cite[\S 2.1]{Dil26} for a more detailed summary of these two gerbes and their bands. We give these gerbes the \'{e}tale topology inherited from $\tn{Spec}(F)$.

For $? \in \{F, v\}$, the basic cohomology $H^{1}(\tn{Kal}_{?}, G)_{\tn{basic}}$ is defined as all isomorphism classes of \'{e}tale $G$-torsors on $\tn{Kal}_{?}$ whose restriction to the band of $\tn{Kal}_{?}$ is an $F$-rational homomorphism $P \to Z_{G}$ or $u_{v} \to Z_{G}$. A key property of these gerbes is that the natural map
\begin{equation*}
H^{1}(\tn{Kal}_{?}, G)_{\tn{basic}} \to H^{1}(F,G_{\tn{ad}})
\end{equation*}
is always surjective. In other words, any inner form of $G$ can be realized as coming from a $G$-torsor on one of these gerbes. 

There is a variant of this gerbe defined in \cite{Fargues} and \cite{Dil26} which is better--suited to geometry:
\begin{defn*}
Define $B_{e}(G)$ (resp. $B_{e}(G_{F_{v}})$) to be all isomorphism classes of \'{e}tale $G$-torsors on the gerbe $\tn{Kott}_{F} \times_{F} \tn{Kal}_{F}$ (resp. $\tn{Kott}_{v} \times_{F_{v}} \tn{Kal}_{v}$) whose restriction to the band of $\tn{Kal}_{F}$ (resp. $\tn{Kal}_{F_{v}}$) factors through $Z_{G}$. Define the subset $B_{e}(G)_{\tn{basic}}$ (resp. $B_{e}(G_{F_{v}})_{\tn{basic}}$) to be those classes whose restriction to the entire band of the gerbe $\tn{Kott}_{F} \times_{F} \tn{Kal}_{F}$ (resp. $\tn{Kott}_{v} \times_{F_{v}} \tn{Kal}_{v}$) factors through $Z_{G}$.
\end{defn*}

In the same way that there is a bijection $B(F,G)_{\tn{basic}} \xrightarrow{\sim} (\pi_{1}(G)[V_{F^{s}}]_{0})_{\tn{Gal}(F^{s}/F)}$, one has a linear-algebraic description of $B_{e}(G)_{\tn{basic}}$ given by \cite[Theorem 3.10]{Dil26} which recovers the $B(F,G)$-bijection and the $H^{1}(\tn{Kal}_{F}, G)_{\tn{basic}}$-bijection from \cite[Theorem 4.11]{Dillery23b}. We omit the details here for sake of brevity, but remark that the proof relies on an explicit Tannakian description \cite[Proposition 3.1]{Dil26} of the gerbe $\tn{Kott}_{F} \times_{F} \tn{Kal}_{F}$ generalizing the description of $\tn{Kott}_{F}$ using Drinfeld isoshtukas. 

\subsection{}
The key ingredient for defining the stack $\tn{Bun}_{G,F}^{e}$ is a generalization of $\tn{Kal}_{F}$ which is defined over a fixed affine open subset $U = C \setminus |\Sigma|$ of the curve $C$ for a finite set of places $\Sigma$. This gerbe is denoted by $\tn{Kal}_{F,\Sigma} \to U$.

The gerbe $\tn{Kal}_{F,\Sigma}$ is built from finite level gerbes. Fix a set of lifts $\dot{\Sigma}$ of $\Sigma$ in $V_{F^{s}}$ and denote by $F_{\Sigma}$ the maximal $\Sigma$-unramified extension of $F$. Fix also a cofinal system of pairs $\{(E_{i}, n_{i})\}_{i \in \mathbb{N}}$ such that $n_{i} \in \mathbb{N}$, $F \subseteq E_{i} \subseteq F_{\Sigma}$ and if $j \geq i$ then $E_{i} \subseteq E_{j}$ and $n_{i} \mid n_{j}$. Define $P_{\dot{\Sigma},i}$ to be the $O_{F,\Sigma}$-group scheme Cartier dual to  $(\frac{1}{n_{i}}\mathbb{Z}/\mathbb{Z})[\Gal(E_{i}/F) \times \Sigma_{F_{i}}]_{0,0}$ consisting of all elements whose $[(\sigma,w)]$-coefficient is zero unless $\sigma^{-1}(w) \in \dot{\Sigma}_{E_{i}}$. The subscript ``$0,0$'' means we insist that the elements must be killed by both augmentation maps. There is (\cite[Corollary 2.12]{Dil26}) a ``level $i$'' canonical class $\xi_{i} \in H_{\tn{fppf}}^{2}(O_{F,\Sigma}, P_{\dot{\Sigma},i})$, and these classes are compatible via a system of transition maps $\{P_{\dot{\Sigma},j} \to P_{\dot{\Sigma},i} \}$ defined loc. cit. such that $P_{\dot{\Sigma}} = \varprojlim P_{\dot{\Sigma},i}$.

To define the canonical class corresponding to $\tn{Kal}_{F,\Sigma}$, one takes the unique lift of $\varprojlim \xi_{i}$ in $H^{2}_{\tn{fppf}}(F, P_{\dot{\Sigma}})$ which, for each $v \in V_{F}$, localizes in some suitable sense to the class in $H^{2}_{\tn{fppf}}(F_{v}, u_{v})$ corresponding to $\tn{Kal}_{v}$---for the full details see \cite[Theorem 2.16]{Dil26}. One can define a transition map $\tn{Kal}_{F,\Sigma'} \to \tn{Kal}_{F,\Sigma}$ for all $\Sigma \subseteq \Sigma'$ and pullback by these maps induces a bijection
\begin{equation*}
\varinjlim_{\Sigma} H^{1}_{\tn{\'{e}t}}(\tn{Kal}_{F,\Sigma}, G) \xrightarrow{\sim} H^{1}_{\tn{\'{e}t}}(\tn{Kal}_{F}, G).
\end{equation*}

\subsection{}
We now recall the definition of the stack $\Bun_{G,F}^{e}$ from \cite[\S 4]{Dil26}, which is inspired by the ideas of \cite{Fargues}. Each gerbe $\tn{Kal}_{F,\Sigma}$ defines (\cite[Lemma 4.2]{Dil26}) an adic gerbe $\tn{Kal}_{F,\Sigma}^{\ad} \to \tn{Spa}(O_{F,\Sigma})$ (for the v-topology), and, for an affinoid perfectoid space $S$ over $\mathbb{F}_{q}$, define $H^{1}_{\tn{\'{e}t}}(\mathfrak{U}_{S}, G)_{\tn{$\tn{Kal}$-basic}}$ as all isomorphism classes of \'{e}tale $G$-torsors $\mathscr{G}$ on $\mathfrak{U}_{S} := U_{S}^{\tn{an}} \times_{\tn{Spa}(O_{F,\Sigma})} \tn{Kal}_{F,\Sigma}^{\ad}$ whose restriction to the band is a morphism 
\begin{equation*}
(P_{\dot{\Sigma}}^{\tn{ad}})_{S} \xrightarrow{\lambda_{\mathscr{G}}} Z_{G,S}.
\end{equation*}
We denote all such torsors (not isomorphism classes) by $Z^{1}_{\tn{\'{e}t}}(\mathfrak{U}_{S}, G)_{\tn{$\Kal$-basic}}$.

The main definition is then (\cite[Definition 4.6]{Dil26}):
\begin{defn*}
\begin{enumerate}
\item{Given $U$ a dense affine open of $C$, define a functor $\tn{Bun}_{G,U}^{e}$ from affinoid perfectoid spaces over $\mathbb{F}_{q}$ to groupoids by
\begin{equation*}
S \mapsto \{(\mathscr{G}, \phi) | \mathscr{G} \in Z^{1}_{\tn{\'{e}t}}(\mathfrak{U}_{S}, G)_{\tn{$\Kal$-basic}}, \hspace{.2cm} \phi \colon \mathscr{G} \xrightarrow{\sim} (\tn{Frob}_{S} \times \tn{id})^{*}\mathscr{G}\};
\end{equation*}}
\item{Define a functor $\tn{Bun}_{G,F}^{e}$ from affinoid perfectoid spaces over $\mathbb{F}_{q}$ to groupoids by 
\begin{equation*}
S \mapsto \varinjlim_{U} \tn{Bun}_{G,U}^{e}(S).
\end{equation*}}
\end{enumerate}
\end{defn*}

It is proved in \cite[Lemma 4.7]{Dil26} that both of these functors are small v-stacks.

\subsection{}\label{ss:BunGeDecomp}
The assignment $(\mathscr{G}, \phi) \in \tn{Bun}_{G,U}^{e}(S) \mapsto \lambda_{\mathscr{G}}$ defines a map
\begin{equation*}
\tn{Bun}_{G,U}^{e} \xrightarrow{\Lambda} \underline{\tn{Hom}}_{U}(P_{\dot{\Sigma}}, Z_{G}),
\end{equation*}
where $\underline{\tn{Hom}}_{U}(P_{\dot{\Sigma}}, Z_{G})$ is the functor sending $S$ to $\tn{Hom}_{U_{S}^{\tn{an}}}((P_{\dot{\Sigma}}^{\tn{ad}})_{S}, Z_{G,S})$. We now recall one more definition \cite[Definition 4.11]{Dil26}.

Observe that we have a canonical isomorphism for every geometric point $\bar{s}$ of a perfectoid space $S$ over $\mathbb{F}_{q}$
\begin{equation*}
\Hom_{U}(P_{\dot{\Sigma}}, Z_{G}) \xrightarrow{\sim} \Hom_{U_{\bar{s}}^{\tn{an}}}((P_{\dot{\Sigma}}^{\tn{ad}})_{\bar{s}}, Z_{G,\bar{s}}).
\end{equation*}
In view of this, given $\lambda \in \Hom(P_{\dot{\Sigma}}^{\tn{ad}}, Z_{G})$, define $\tn{Bun}^{e,\lambda}_{G,U}$  as the substack characterized by 
\begin{equation*}
\tn{Bun}^{e,\lambda}_{G,U}(S) = \{x \in \tn{Bun}^{e}_{G,U}(S) | \lambda_{\bar{x}} = \lambda \in  \Hom_{U_{\bar{s}}^{\tn{an}}}((P_{\dot{\Sigma}}^{\tn{ad}})_{\bar{s}}, Z_{G,\bar{s}})\},
\end{equation*}
where $\bar{s}$ runs over all geometric points of $S$ and $\bar{x}$ denotes the image of $x$ at $\bar{s}$.

A key result which allows us to use the results in the main body of paper in this new setting is \cite[Corollary 4.13]{Dil26}, which says that $\Lambda$ induces a decomposition (as a v-stack)
\begin{equation*}
\tn{Bun}_{G,U}^{e} \xrightarrow{\sim} \bigsqcup_{\lambda \in \tn{Hom}_{F}(P_{\dot{\Sigma}}, Z_{G})}  \tn{Bun}_{G,U}^{e,\lambda},
\end{equation*}
and analogously for $\tn{Bun}_{G,F}^{e}$.

There is an analogue \cite[Definition 4.9]{Dil26} of the notions introduced in \S \ref{ss:twistingb}:
\begin{defn*}
\begin{enumerate}
\item{Define $\tn{Bun}^{e,1}_{G,F}$ to be the substack characterized by 
\begin{equation*}
\tn{Bun}^{e,1}_{G,F}(S) = \{x \in \tn{Bun}^{e}_{G,F}(S) | \bar{x} = \tn{triv} \in  \tn{Bun}_{G,F}(\ov{s})\},
\end{equation*}
where $\bar{s}$ runs over all geometric points of $S$ and $\bar{x}$ denotes the image of $x$ at $\bar{s}$.}
\item{Given $b$ an element of $B_{e}(G)_{\tn{basic}}$, define $\tn{Bun}^{e,b}_{G,F}$  as the substack characterized by $\tn{Bun}^{e,b}_{G,F}(S) = \{x \in \tn{Bun}^{e}_{G,F}(S) | \bar{x} = b \in  \tn{Bun}_{G,F}(\ov{s})\}$, where $\bar{s}$ runs over all geometric points of $S$ and $\bar{x}$ denotes the image of $x$ at $\bar{s}$.} 
\end{enumerate}
\end{defn*}

The identical twisting argument from the proof of Proposition \ref{ss:twistingb} shows that:

\begin{prop*} \begin{enumerate}
\item{For $b \in B_{e}(G)_{\tn{basic}}$ pulled back from $\tn{Kott}_{F,\Sigma} \times_{U} \tn{Kal}_{F,\Sigma}$, there is a natural isomorphism $\tn{Bun}^{e}_{G,U} \xrightarrow{\sim} \tn{Bun}^{e}_{G_{b},U}$ sending $\tn{Bun}^{e,b}_{G,U}$ to $\tn{Bun}^{e,1}_{G_{b},U}$.}
\item{For $b \in H^{1}(\tn{Kott}_{F,\Sigma} \times_{U} \tn{Kal}_{F,\Sigma})_{\tn{basic}}$ above a fixed $\lambda \in \Hom_{U}(P_{\dot{\Sigma}}, Z_{G})$, we have an isomorphism $\tn{Bun}_{G,U}^{e,\lambda} \xrightarrow{\sim} \tn{Bun}_{G_{b},U}$. }
\end{enumerate}
\end{prop*}
In the above $\tn{Kott}_{F,\Sigma}$ is a ``$U$-level'' version of $\tn{Kott}_{F}$ defined in \cite[\S 2.3]{Dil26}. Combining this result with \cite[Corollary 4.13]{Dil26} proves that the stack $\tn{Bun}_{G,U}^{e}$ is isomorphic to a disjoint union of the stacks $\tn{Bun}_{G',U}$ for varying inner forms $G'$ of $G$ (this is a global analogue of \cite[Exemple 12.6]{Fargues}).

We therefore deduce from Proposition \ref{ss:localizationproperties} that:
\begin{cor*}
Each $\tn{Bun}_{G,U}^{e}$ is an Artin $v$-stack.
\end{cor*}

\subsection{}
We now prove some more refined results about the geometry of $\tn{Bun}^{e}_{G,F}$.

 First, we deduce the following analogue of Theorem \ref{ss:BunGred}, which, in particular, identifies $\Bun_{G,U}^{e}(\Spd \ov{\mathbb{F}_{q}} )$ with $B_{e}(G)$:
\begin{thm*}
  The perfect v-stack $\Bun_{G,U}^{e,\red}$ over $\bF_q$ is the v-sheafification of
  \begin{align}
 \Spec{B}\mapsto\left\{
    \begin{tabular}{c}
      $\tn{Kal}_{F}$-basic $G$-torsors $\mathscr{G}$ on $U_B \times_{U} \tn{Kal}_{F,\Sigma}$ equipped with\\
      an isomorphism $\phi:\mathscr{G}\ra^\sim(\Frob_B \times \tn{id})^*\mathscr{G}$
    \end{tabular}
    \right\}.
  \end{align}
Moreover, when each connected component of $\Spec{B}$ is a valuation ring, no sheafification is needed.
\end{thm*}

\begin{proof}
\cite[Corollary 4.13]{Dil26} implies that $\Bun_{G,U}^{e,\red} \xrightarrow{\sim} \bigsqcup_{\lambda \in \tn{Hom}_{F}(P_{\dot{\Sigma}}, Z_{G})}  \tn{Bun}_{G,U}^{e,\lambda,\red}$. We have seen in Proposition \ref{ss:BunGeDecomp} that picking $b \in H^{1}(\text{Kott}_{F,\Sigma} \times \text{Kal}_{F,\Sigma})_{\text{basic}}$ above a fixed $\lambda$ (we can do this by \cite[Proposition 2.29]{Dil26}) gives an identification $ \tn{Bun}_{G,U}^{e,\lambda} \xrightarrow{\sim} \tn{Bun}_{G_{b},U}$. A similar (but easier) argument as in the proof of \cite[Proposition 4.12]{Dil26} gives an analogous disjoint union decomposition of the functor (1) according to $\Hom_{U}(P_{\dot{\Sigma}}, Z_{G})$, and twisting by $b$ gives an isomorphism from the presheaf mapping $\Spec{B}$ to all $\mathscr{G}$ as in (1) such that $\lambda_{\mathscr{G}} = \lambda$ at all geometric points of $B$ to the functor in \eqref{eq:BunGred}. We can then use Theorem \ref{ss:BunGred} to deduce the result.
\end{proof}

\subsection{}
This subsection is purely review from \cite{Dil26} and concerns the localization properties of the stack $\Bun^{e}_{G,F}$. First, we define for a place $v \in V_{F}$ the stack $\tn{Bun}_{G,F_{v}}^{e}$ to be the same as $\tn{Bun}_{G,U}^{e}$ but with $F_{v}$ in place of $O_{F,\Sigma}$ and $\tn{Kal}_{v}$ in place of $\tn{Kal}_{F,\Sigma}$, as defined in \cite[\S 12]{Fargues}.

We recall from \cite[\S 4.2]{Dil26} the functor $\tn{Bun}^{e}_{G,\mathbb{A}}$ defined by  
\begin{equation*}
S \mapsto \varinjlim_{\Sigma} \prod_{v \notin \Sigma} \tn{Bun}_{G,O_{v}}(S) \times \prod_{v \in \Sigma} \tn{Bun}^{e}_{G,F_{v}}(S),
\end{equation*}
where the limit is over all finite subsets of places $\Sigma$. This defines a small v-stack, and we have a morphism
\begin{equation*}
\tn{Bun}^{e}_{G,F} \to \tn{Bun}^{e}_{G,\mathbb{A}}
\end{equation*}
induced by a localization map $\tn{Kal}_{v} \to \tn{Kal}_{F}$ at each place as discussed in \cite[\S 2.1]{Dil26}.

One then has \cite[Proposition 4.16]{Dil26}:
\begin{prop*}
There is a morphism 
\begin{equation*}
\tn{Bun}^{e}_{G,U} \to \prod_{v \notin \Sigma} \tn{Bun}_{G,O_{v}} \times \prod_{v \in \Sigma} \tn{Bun}^{e,(\Sigma)}_{G,F_{v}} \subseteq \prod_{v \notin \Sigma} \tn{Bun}_{G,O_{v}} \times \prod_{v \in \Sigma} \tn{Bun}^{e}_{G,F_{v}}
\end{equation*}
which makes the square
\[
\begin{tikzcd}
\tn{Bun}^{e}_{G,U} \arrow{r} \arrow{d} &  \prod_{v \notin \Sigma} \tn{Bun}_{G,O_{v}} \times \prod_{v \in \Sigma} \tn{Bun}^{e,(\Sigma)}_{G,F_{v}}  \arrow{d} \\
\tn{Bun}^{e}_{G,F} \arrow{r} &  \tn{Bun}^{e}_{G,\mathbb{A}}
\end{tikzcd}
\]
Cartesian.
\end{prop*}
In the above, $\tn{Bun}^{e,(\Sigma)}_{G,F_{v}}$ is a minor modification of the stack $\tn{Bun}^{e}_{G,F_{v}}$ adapted to the completion of $F_{\Sigma}$ at a fixed place above $v$, see \cite[\S 4.2]{Dil26} for the full details.

\subsection{}
There is an analogue of the semistable locus of $\Bun_{G,F}$:

\begin{defn*}
Set $\tn{Bun}_{G,F}^{e,\tn{ss}}$ to be the preimage of 
\begin{equation*}
\varinjlim \prod_{v \notin \Sigma} \tn{Bun}_{G,O_{v}} \times \prod_{v \in \Sigma} \tn{Bun}^{e,\tn{ss}}_{G,F_{v}}
\end{equation*}
in $\tn{Bun}^{e}_{F,G}$, where $\tn{Bun}^{e,\tn{ss}}_{G,F_{v}}$ is the local semistable locus as defined in \cite{Fargues}. This does not depend on the choice of localization maps, is a small v-stack, and is an open substack of $\tn{Bun}_{G,F}^{e,\tn{ss}}$.
\end{defn*}
One can explicitly describe the semistable locus of $\Bun_{G,F}^{e}$ using $B_{e}(G)_{\tn{basic}}$. 

\begin{thm*}
\begin{enumerate}
\item{The substack $\tn{Bun}^{e,1}_{G,U} \subseteq \tn{Bun}^{e}_{G,U}$ is open, and is isomorphic to $*/\underline{G(O_{F,\Sigma})}$. }
\item{For $b \in B_{e}(G)_{\tn{basic}}$, the substack $\tn{Bun}_{G,F}^{e,b} \subseteq \tn{Bun}_{G,F}^{e}$ is open and isomorphic to $\ast/\underline{G_{b}(F)}$.}
\item{We have an equality of stacks $\tn{Bun}_{G,F}^{e,\tn{ss}} = \bigsqcup_{b \in B_{e}(G)_{\tn{basic}}} \tn{Bun}_{G,F}^{e,b}$.}
\end{enumerate}
\end{thm*}

\begin{proof}
For the first statement, we observe that the image of $\tn{Bun}^{e,1}_{G,U}$ via $\Lambda$ is determined by its value at any geometric point, and is thus the trivial morphism $0_{P}$. By Proposition \ref{ss:BunGeDecomp} we can identify the inclusion $\tn{Bun}^{e,1}_{G,U} \hookrightarrow \text{Bun}_{G,U}^{e, 0_{P}}$ with the inclusion $\text{Bun}_{G,U}^{1} \hookrightarrow \text{Bun}_{G,U}$, so that the desired result is true by combining \cite[Proposition 4.12]{Dil26} with Theorem \ref{ss:localBunG}.

The second part is a consequence of the first statement and Proposition \ref{ss:BunGeDecomp}. The disjointness part of the second statement can be deduced from the disjointness of the $\tn{Bun}_{G,F}^{b}$ after decomposing $\tn{Bun}_{G,F}^{e}$ using \cite[Corollary 4.13]{Dil26}. Similarly, the verification that this union is exhaustive can be done using the decomposition (as stacks)
\begin{equation*}
\tn{Bun}_{G,U}^{e,\tn{ss}} \xrightarrow{\sim} \bigsqcup_{\lambda \in \tn{Hom}_{F}(P_{\dot{\Sigma}}, Z_{G})}  \tn{Bun}_{G,U}^{e,\tn{ss},\lambda}
\end{equation*}
induced by \cite[Corollary 4.13]{Dil26} and then observing that the identification $\tn{Bun}_{G,U}^{e,\lambda} \xrightarrow{\sim} \tn{Bun}_{G_{b},U}^{e}$ for $b \in B_{e}(G)_{\text{basic}}$ identifies $\tn{Bun}_{G,U}^{e,\tn{ss},\lambda}$ with $\tn{Bun}_{G_{b},U}^{\tn{ss}}$. Now one can use Theorem \ref{ss:basiclocus}.
\end{proof}

\bibliographystyle{../habbrv}
\bibliography{biblio}
\end{document}